\documentclass[10pt, letterpaper]{article}

\usepackage[margin=1in]{geometry}
\usepackage[T1]{fontenc}
\usepackage{lmodern}
\usepackage{microtype}

\usepackage{amsmath, amssymb, amsthm, mathtools}
\usepackage{enumitem}

\usepackage[hidelinks]{hyperref}

\usepackage{authblk}
\theoremstyle{definition}
\newtheorem{assumption}{Assumption}[section]

\theoremstyle{plain}
\newtheorem{theorem}{Theorem}[section]     
\newtheorem{proposition}[theorem]{Proposition}
\newtheorem{lemma}[theorem]{Lemma}
\newtheorem{corollary}[theorem]{Corollary}

\theoremstyle{definition}
\newtheorem{definition}[theorem]{Definition}
\newtheorem{axiom}[theorem]{Assumption}
\newtheorem{example}[theorem]{Example}

\theoremstyle{remark}
\newtheorem{remark}[theorem]{Remark}
\newtheorem{convention}[theorem]{Convention}

\newcommand{\Tr}{\operatorname{Tr}}
\newcommand{\ran}{\operatorname{ran}}

\newcommand{\Hcal}{\mathcal H}
\newcommand{\Bcal}{\mathcal B}
\newcommand{\BH}{\mathcal B(\Hcal)}
\newcommand{\Tone}{\mathcal T_1(\Hcal)}

\newcommand{\R}{\mathbb R}
\newcommand{\C}{\mathbb C}

\newcommand{\Eig}{\mathbb Eig} 

\newcommand{\Var}{\mathcal Var}

\title{\textbf{Mathematical Foundations of Quantum Pricing Theory}}

\author[1,2]{Tian Xin
\thanks{Corresponding author: \texttt{xtian21@jh.edu}.}
\thanks{Acknowledgements. We are grateful to Professor Frank J.\ Fabozzi for helpful advice and continued encouragement. Tian Xin was the primary contributor to this manuscript.}}

\author[2]{Aoqin Liang}
\affil[1]{Johns Hopkins University}
\affil[2]{School of Physics, Peking University}
\date{\today}

\begin{document}
\maketitle

\begin{abstract}
We propose an operator-algebraic foundation for discrete-time pricing under classical information. The market is modeled by a von Neumann algebra $M$ together with an increasing family of abelian information subalgebras $(N_t)_{t=0}^T$ and faithful normal conditional expectations $E_t:M\to N_t$.
Discounted prices are represented as self-adjoint operators affiliated with $N_t$, and pricing is encoded by a dynamic family of linear maps $\Pi_t:M\to N_t$.

Our first result establishes that, under a no-arbitrage condition of Kreps--Yan type (NAFL$_\sigma$) and a suitable closedness assumption on attainable claims, there exists a normal pricing state $\phi^\star$ on $M$ yielding a dual representation of prices.
The associated pricing operators $\Pi_t$ are normal, completely positive, and $N_t$-bimodular, and they satisfy the num\'eraire normalization $\Pi_t(\mathbf 1)=B_t$, where $B_t\in N_t$ is the money-market account. Equivalently, the discounted operators $\widetilde\Pi_t(X):=B_t^{-1/2}\Pi_t(X)B_t^{-1/2}$ are normal, completely positive,
and unital. Under an additional modular-compatibility hypothesis (Takesaki-type invariance), we obtain existence and uniqueness of state-preserving conditional expectations consistent with the filtration and satisfying a tower property.

Finally, we develop an operator-valued information framework based on free Fisher information and prove an operator-valued Cram\'er--Rao inequality relative to $N_t$, yielding a quantitative constraint linking conditional second moments and noncommutative information content. In the commutative reduction $N_t\simeq L^\infty(\Omega,\Sigma_t,\mathbb P)$, the framework recovers the classical conditional-expectation pricing paradigm.

\end{abstract}

\section{Introduction}

\subsection{Motivation and scope}

Classical arbitrage pricing theory is organized around conditional expectation and the induced Hilbert-space geometry: under a risk-neutral measure, discounted traded prices are martingales and derivative values are conditional expectations of discounted payoffs \cite{HarrisonKreps1979,HarrisonPliska1981,Bjork2009}. 
The purpose of this paper is to formulate an operator-algebraic analogue of this architecture. 
We model the \emph{global market} by a von Neumann algebra $M$, and we encode the \emph{realized information at time $t$} by an increasing family of \emph{abelian} von Neumann subalgebras $(N_t)_{t\in[0,T]}$ with
\[
N_t \subseteq M_t \subseteq M:=M_T,\qquad t\in[0,T],
\]
where $(M_t)_{t\in[0,T]}$ is a filtered family of von Neumann subalgebras. 
The abelianity of $N_t$ reflects the operational premise that realized information is classical (i.e.\ simultaneously measurable) even when the ambient market algebra is not.

A central functional-analytic obstruction is that, unlike the commutative case, a normal state-preserving conditional expectation onto a prescribed subalgebra does not exist automatically. 
Hence the existence of a filtration of normal $\varphi$-preserving conditional expectations
\[
E_t: M \to N_t,\qquad t\in[0,T],
\]
is a genuine \emph{structural} requirement, governed by modular invariance in the sense of Takesaki \cite{Umegaki1954,Takesaki1972,AccardiCecchini82,Takesaki1}. 
Our framework separates this operator-algebraic existence/uniqueness problem from the economic problem of deriving a pricing state from no-arbitrage considerations.

\subsection{Setup: states, affiliated observables, and \texorpdfstring{$L^2$}{L2}-geometry}

Fix a horizon $T>0$ and let $M$ act standardly on a separable Hilbert space. 
Let $\varphi_\rho$ be a faithful normal state on $M$. 
We consider the GNS Hilbert space $L^2(M,\varphi_\rho)$ with inner product
\[
\langle X,Y\rangle_\rho := \varphi_\rho(X^*Y),\qquad X,Y\in M.
\]
When they exist, we work with normal $\varphi_\rho$-preserving conditional expectations $E_t:M\to N_t$ compatible with the tower property $E_s\circ E_t=E_s$ for $s\le t$. 
To incorporate unbounded observables affiliated with $M$, we employ bounded functional-calculus truncations $f_n$ and formulate martingale/efficiency notions at the level of truncated observables. 
This yields a localization mechanism that is stable under domain issues and naturally compatible with noncommutative $L^2$-geometry and noncommutative integration \cite{Segal1953,Nelson1974,Takesaki1}.

\subsection*{Main contributions}

For the reader's convenience we indicate where results are proved; forward references in this Introduction are purely navigational and no statement is used before it is established.

\medskip
\noindent\textbf{(1) $L^2$-projection theory for state-preserving conditional expectations.}
We develop the $L^2$-geometry of normal state-preserving conditional expectations and identify them as best-approximation/projection maps in $L^2(M,\varphi_\rho)$.
We establish closure and regularity mechanisms which extend conditioning from bounded elements to affiliated observables via spectral truncations.
These results are proved in Section~\ref{sec:information_and_C.E}.

\medskip
\noindent\textbf{(2) Truncation-stable martingales and a local informational efficiency principle.}
We introduce a truncation-stable notion of martingales relative to the abelian information flow $(N_t,E_t)$ for self-adjoint affiliated observables and formulate a local informational efficiency principle requiring symmetrically discounted traded prices to be martingales relative to $(N_t,E_t)$.
This is developed in Sections~\ref{section:MF_of_QPT} and~\ref{sec:information_and_C.E}.

\medskip
\noindent\textbf{(3) Dynamic pricing operators as completely positive information projections.}
Assuming a pricing state $\varphi^\star$ together with a compatible filtration of normal $\varphi^\star$-preserving conditional expectations $(E_t^\star)$, we define a dynamic pricing (valuation) operator via symmetric discounting.
We prove its intrinsic operator-algebraic properties: normality, complete positivity, numéraire-normalization (Equivalently, the discounted operator $\widetilde{\Pi}_t(X):=B_t^{-1/2}\,\Pi_t(X)\,B_t^{-1/2}$ is unital), $N_t$-bimodularity, and time consistency. These results are established in Section~\ref{sec:information_and_C.E}.

\medskip
\noindent\textbf{(4) Modular existence/uniqueness and no-arbitrage existence of pricing states.}
We derive existence and uniqueness criteria for $(E_t^\star)$ from Takesaki's modular invariance characterization of normal state-preserving conditional expectations.
Separately, we establish existence of a normal pricing state from a no-arbitrage hypothesis via a separation argument in an appropriate locally convex topology.
These results are proved in Section~\ref{sec:duality}.

\medskip
\noindent\textbf{(5) Prediction theory and information-theoretic constraints.}
We formulate an $L^2$-prediction framework in which conditional expectations are best predictors and derive quantitative lower bounds on conditional mean-square prediction error in terms of innovations and operator-valued Fisher information.
The prediction and information-theoretic layer is developed in Sections~\ref{sec:information_and_C.E} and~\ref{sec:application-return}.

\subsection{Relation to the literature}

The operator-algebraic backbone of the present work is the theory of conditional expectations and noncommutative integration on von Neumann algebras.
Conditional expectations were introduced by Umegaki \cite{Umegaki1954} and developed systematically within the modern framework of von Neumann algebras; see, for instance, \cite{Sakai1971,KadisonRingroseI,Takesaki1}. 
In the non-tracial setting, existence and uniqueness of a normal state-preserving conditional expectation onto a given von Neumann subalgebra is governed by modular structure; in particular, Takesaki's characterization via invariance under the modular automorphism group \cite{Takesaki1972} (and refinements such as \cite{AccardiCecchini82}) provides the criterion underlying our standing assumptions on information projections.
For unbounded affiliated observables and noncommutative $L^p$-type integrability we rely on the noncommutative integration program initiated by Segal \cite{Segal1953} and further developed by Nelson \cite{Nelson1974}.

From the perspective of mathematical finance, the commutative reduction of our framework recovers the classical martingale approach to asset pricing and change-of-num\'eraire ideas \cite{Samuelson1965,HarrisonKreps1979,HarrisonPliska1981,DelbaenSchachermayer1994,DelbaenSchachermayer1995Numeraire,GemanElKarouiRochet1995,Bjork2009}.
There is also a literature proposing ``quantum'' or noncommutative extensions of valuation, including \cite{AccardiBoukas2006,Herscovich2016,AertsHavenSozzo2012,McCloud2014,Su2016}.
Our objective is not to introduce an additional model class, but to provide a rigorous von Neumann algebraic foundation in which (i) markets are operator-algebraic, (ii) realized information is encoded by an abelian subalgebra filtration, and (iii) dynamic valuation arises from normal state-preserving conditional expectations with intrinsic complete positivity.

Finally, the information-theoretic layer is motivated by noncommutative Fisher information and entropy in free and operator-valued probability initiated by Voiculescu \cite{Voiculescu93,Voiculescu94} and further developed in the operator-valued setting (see, e.g., \cite{MengGuoCao2004,Meng05PAMS,Meng06arXiv,Speicher98,NicaSpeicher06,NSS02}). 
We use operator-valued Fisher information as a quantitative constraint on conditional prediction error relative to the information algebra, connecting operator-algebraic conditioning with noncommutative information geometry.

\subsection{Organization of the paper}

Section~\ref{sec:MFQM} recalls the functional-analytic foundations used throughout (states, affiliated observables, spectral calculus, truncation, and basic measurement notions). 
Section~\ref{section:MF_of_QPT} develops the operator-algebraic market model, symmetric discounting, the local efficiency principle, and the commutative reduction to classical martingale pricing.
Section~\ref{sec:information_and_C.E} studies information algebras and normal state-preserving conditional expectations, including $L^2$-geometry and truncation-based martingales for affiliated observables, and introduces the dynamic pricing operator.
Section~\ref{sec:duality} derives existence of a normal pricing state under a no-arbitrage hypothesis and formulates modular conditions ensuring existence and uniqueness of $\varphi^\star$-preserving conditional expectations.
The remaining sections develop prediction and information-theoretic bounds and analyze explicit examples.

\section{Operator-Theoretic For Pricing Theory}~\label{sec:MFQM}
The core of this paper is operator-algebraic (von Neumann algebras, normal states,
state-preserving conditional expectations, and a truncation-based martingale principle).
Accordingly, we keep the operator-theoretic background to the minimum needed for later
references. Standard sources include \cite{ReedSimonI,ReedSimonII,KadisonRingroseI}.

Throughout, $\Hcal$ is a complex separable Hilbert space with inner product
$\langle\cdot\mid\cdot\rangle$ (linear in the second argument). We write $\BH:=\mathcal B(\Hcal)$
for bounded operators and $\Tone:=\mathcal T_1(\Hcal)$ for trace-class operators with trace norm
$\|\cdot\|_1$.

\subsection{Trace-class operators and the trace}

\begin{lemma}[Trace-class ideal property and cyclicity]\label{lem:trace-ideal}
If $T\in\BH$ and $S\in\Tone$, then $TS,ST\in\Tone$ and
\[
\|TS\|_1\le \|T\|\,\|S\|_1,\qquad \|ST\|_1\le \|T\|\,\|S\|_1.
\]
In particular, $\Tr(TS)$ and $\Tr(ST)$ are well-defined and $\Tr(TS)=\Tr(ST)$.
\end{lemma}

\begin{axiom}[Normal states on $\BH$]\label{ax:states}
A (normal) state is represented by a density operator $\rho\in\Tone$, i.e.\ $\rho\ge 0$ and
$\Tr(\rho)=1$. Expectations of bounded observables $T\in\BH$ are given by
$\varphi_\rho(T):=\Tr(\rho\,T)$.
\end{axiom}

\subsection{Spectral theorem (reference form)}

\begin{theorem}[Spectral theorem for self-adjoint operators]\label{thm:spectral}
Let $A$ be self-adjoint on $\Hcal$. There exists a unique projection-valued measure
$E_A:\Bcal(\R)\to\BH$ such that, in the sense of functional calculus,
\[
A=\int_{\R}\lambda\,E_A(d\lambda),
\qquad
f(A)=\int_{\R} f(\lambda)\,E_A(d\lambda)\in\BH
\quad
\text{for all bounded Borel } f:\R\to\C.
\]
\end{theorem}

\begin{remark}[How this is used later]\label{rem:spectral-use}
We only use Theorem~\ref{thm:spectral} to justify functional calculus and spectral projections
needed for truncations and for measurable functional transforms of (possibly unbounded)
self-adjoint operators; no further quantum-mechanical interpretation is required.
\end{remark}

\subsection{Optional: selective state update (only if used later)}

\begin{axiom}[Selective (L\"uders) update]\label{ax:luders}
Let $A$ be a self-adjoint operator with spectral measure $E_A$ as in
Theorem~\ref{thm:spectral}. For a state $\rho$ and $F\in\Bcal(\R)$ with
$\Tr(E_A(F)\rho)>0$, define the conditional post-update state
\[
\rho_F:=\frac{E_A(F)\,\rho\,E_A(F)}{\Tr(E_A(F)\rho)}.
\]
\end{axiom}

\begin{proposition}[Well-posedness of the selective update]\label{prop:luders-wellposed}
Under Assumption~\ref{ax:luders}, $\rho_F$ is a density operator, i.e.\ $\rho_F\ge 0$, $\rho_F\in\Tone$,
and $\Tr(\rho_F)=1$.
\end{proposition}

\begin{proof}
Positivity is immediate since $E_A(F)$ is a projection. By Lemma~\ref{lem:trace-ideal},
$E_A(F)\rho E_A(F)\in\Tone$, hence $\rho_F\in\Tone$. Finally,
\[
\Tr(\rho_F)=\frac{\Tr(E_A(F)\rho E_A(F))}{\Tr(E_A(F)\rho)}
=\frac{\Tr(E_A(F)\rho)}{\Tr(E_A(F)\rho)}=1,
\]
using Lemma~\ref{lem:trace-ideal} and $E_A(F)^2=E_A(F)$.
\end{proof}

\subsection*{References for Section~\ref{sec:MFQM}}
See \cite{ReedSimonI,ReedSimonII,KadisonRingroseI} for the spectral theorem, trace-class ideals,
and normality properties of the trace.

\section{Mathematical Foundations of Quantum Pricing Theory}\label{section:MF_of_QPT}

This section introduces the operator-algebraic market architecture used throughout.
We specify the information flow, price observables, symmetric discounting, and a localized
martingale/efficiency principle relative to the \emph{classical} information algebras.
The $L^2$-projection theory and analytic extensions are developed in Section~\ref{sec:information_and_C.E}.

\medskip
\noindent\textbf{Standing convention (states).}
A \emph{state} means a normal state $\varphi\in M_\ast$ on a von Neumann algebra $M$.
When working in a concrete model $M\subseteq B(\mathcal H)$, we may implement $\varphi$ as
$\varphi(X)=\Tr(\rho X)$ for a density operator $\rho\in\mathcal T_1(\mathcal H)$.
This is a \emph{model assumption} and is not required for the abstract formulation.

\subsection{Market and information algebras}\label{subsec:market_info}

\begin{assumption}[Filtered market and classical information]\label{as:market_info}
Fix $T>0$. Let $\{M_t\}_{t\in[0,T]}$ be an increasing family of von Neumann subalgebras of $B(\mathcal H)$.
For each $t$, let $N_t\subseteq M_t$ be an \emph{abelian} von Neumann subalgebra, and assume
$N_s\subseteq N_t$ for $s\le t$.
\end{assumption}

\begin{remark}\label{rem:info_classical}
Abelianity of $N_t$ encodes the operational premise that \emph{realized information is classical}
(simultaneously measurable), even when the ambient market algebra $M_t$ is noncommutative.
\end{remark}

\begin{proposition}[Commutative realization of $N_t$]\label{prop:abelian_sigma}
Let $N$ be an abelian von Neumann algebra on $\mathcal H$.
Then there exist a measure space $(\Omega,\Sigma,\mu)$, a separable multiplicity space $\mathcal K$,
and a unitary $U:\mathcal H\to L^2(\Omega,\Sigma,\mu;\mathcal K)$ such that
\[
U N U^\ast=\{\,M_f\otimes I_{\mathcal K}: f\in L^\infty(\Omega,\Sigma,\mu)\,\}.
\]
In particular, up to normal $\ast$-isomorphism, specifying $N_t$ is equivalent to specifying
a classical $\sigma$-algebra $\Sigma_t$.
\end{proposition}

\subsection{Price observables and symmetric discounting}\label{subsec:prices_discount}

\begin{definition}[Affiliated observables]\label{def:measurable_op}
Let $\mathcal A\subseteq B(\mathcal H)$ be a von Neumann algebra.
A self-adjoint operator $X$ on $\mathcal H$ is \emph{affiliated} with $\mathcal A$ if
its spectral projections satisfy $E_X(F)\in\mathcal A$ for all Borel sets $F\subseteq\mathbb R$.
\end{definition}

\begin{assumption}[Prices and num\'eraire]\label{as:prices_numeraire}
For each traded asset $i\in\{1,\dots,d\}$ and each $t\in[0,T]$:
\begin{enumerate}[label=(\roman*), itemsep=2pt, topsep=4pt]
\item $S_t^{(i)}$ is a self-adjoint operator affiliated with $N_t$;
\item there exists a strictly positive num\'eraire $(B_t)_{t\in[0,T]}$ such that each
$B_t\in N_t$ is positive and boundedly invertible.
\end{enumerate}
\end{assumption}

\begin{definition}[Symmetric discounting]\label{def:sym_discount}
Define the symmetrically discounted price observable
\[
\bar S_t^{(i)}:=B_t^{-1/2}\,S_t^{(i)}\,B_t^{-1/2}.
\]
In the commutative realization of Proposition~\ref{prop:abelian_sigma}, this equals multiplication by
$b_t^{-1}s_t^{(i)}$.
\end{definition}

\begin{remark}[Domain issues are resolved by commutativity]\label{rem:discount_domain}
Although $S_t^{(i)}$ may be unbounded, the expression in Definition~\ref{def:sym_discount}
is unambiguous because $S_t^{(i)}$ is affiliated with the \emph{abelian} algebra $N_t$:
in the commutative realization it is simply pointwise multiplication by $b_t^{-1}s_t^{(i)}$.
(We do not repeat the full domain verification here; see the discussion following Proposition~\ref{prop:abelian_sigma}.)
\end{remark}

\subsection{Information projections and a localized efficiency principle}\label{subsec:condexp_liep}

\begin{definition}[Truncations]\label{def:trunc}
For $n\in\mathbb N$ let $f_n:\mathbb R\to\mathbb R$ be the bounded truncation
$f_n(\lambda)=\max(-n,\min(\lambda,n))$.
For self-adjoint $X$, define $f_n(X)$ via functional calculus; then $f_n(X)\in B(\mathcal H)$.
\end{definition}

\begin{assumption}[State and conditional expectations]\label{as:condexp_exist}
Fix a faithful normal state $\varphi_\rho$ on $M:=M_T$.
Assume there exists a family of normal $\varphi_\rho$-preserving conditional expectations
$E_t:M\to N_t$ such that for $0\le s\le t\le T$,
\[
E_s\circ E_t=E_s.
\]
\end{assumption}
\begin{remark}[Conditioning vs.\ measurement]\label{rem:conditioning_vs_measurement}
In the main text, \emph{conditioning on available information} is represented by normal
conditional expectations $E_t:M\to N_t$ onto the (abelian) information algebras.
An optional interpretation of information arrival as a projective measurement and the
corresponding L\"uders selective update is provided in Appendix~\ref{app:qm:meas}.
\end{remark}

\begin{definition}[(Truncation-stable) $(N_t,E_t)$-martingale]\label{def:qmartingale}
Let $(X_t)_{t\in[0,T]}$ be self-adjoint operators with $X_t$ affiliated with $N_t$.
We call $(X_t)$ an \emph{$(N_t,E_t)$-martingale} if for all $0\le s\le t\le T$ and all $n\in\mathbb N$,
\[
E_s\!\big(f_n(X_t)\big)=f_n(X_s).
\]
\end{definition}

\begin{assumption}[Local Informational Efficiency Principle (LIEP)]\label{as:LIEP}
For each traded asset $i$, the discounted process $(\bar S_t^{(i)})_{t\in[0,T]}$
is an $(N_t,E_t)$-martingale in the sense of Definition~\ref{def:qmartingale}.
\end{assumption}

\begin{remark}\label{rem:local_vs_global}
If $N_t=M_t$ (full information), Assumption~\ref{as:LIEP} reduces to the usual global martingale/EMH postulate.
Here $N_t$ is typically a proper subalgebra, so LIEP encodes unpredictability \emph{relative to available information}.
\end{remark}

\subsection{Risk-neutral representation and the pricing operator}\label{subsec:pricing_operator}
\begin{assumption}[Pricing state and valuation expectations (standing hypothesis)]\label{as:RN}
Assume there exist a normal state $\varphi^\star$ on $M:=M_T$ (the \emph{pricing state})
and a family of normal $\varphi^\star$-preserving conditional expectations
$\{E_t^\star:M\to N_t\}_{t\in[0,T]}$ such that:
\begin{enumerate}[label=(\roman*), itemsep=2pt, topsep=4pt]
\item \textbf{(Tower)} $E_s^\star\circ E_t^\star=E_s^\star$ for $0\le s\le t\le T$;
\item \textbf{(Traded-asset consistency)} for each traded asset $i$,
the discounted prices $(\bar S_t^{(i)})_{t\in[0,T]}$ form an $(N_t,E_t^\star)$-martingale
in the sense of Definition~\ref{def:qmartingale}.
\end{enumerate}
\end{assumption}

\begin{definition}[Pricing operator on bounded payoffs]\label{def:pricing_operator}
Let $X\in M_T$ be bounded. Define its symmetrically discounted payoff
\[
\bar X:=B_T^{-1/2}XB_T^{-1/2},
\]
and set, for $t\in[0,T]$,
\[
\Pi_t(X):=B_t^{1/2}\,E_t^\star(\bar X)\,B_t^{1/2}\in N_t.
\]
Equivalently, the discounted pricing operator is
\[
\widetilde{\Pi}_t(X):=B_t^{-1/2}\Pi_t(X)B_t^{-1/2}=E_t^\star(\bar X)\in N_t.
\]
In the commutative case, $\Pi_t(X)=B_t\,\mathbb E^\star\!\left[B_T^{-1}X\mid\Sigma_t\right]$.
\end{definition}

\begin{proposition}[Structural properties]\label{prop:pricing_properties}
Under Assumption~\ref{as:RN}, for each $t$ the map $\Pi_t:M_T\to N_t$ is
normal, completely positive, and $N_t$-bimodular. Moreover, it is \emph{num\'eraire-normalized} in the sense that
\[
\Pi_t(\mathbf 1)=B_t.
\]
Equivalently, $\widetilde{\Pi}_t:M_T\to N_t$ is normal, completely positive, \emph{unital}, and $N_t$-bimodular.
Finally, $\Pi$ is time-consistent: for $0\le s\le t\le T$ and bounded $X\in M_T$,
\[
\Pi_s(X)=B_s^{1/2}\,E_s^\star\!\Big(B_t^{-1/2}\Pi_t(X)B_t^{-1/2}\Big)\,B_s^{1/2}.
\]
\end{proposition}

\begin{proof}
\emph{Complete positivity.} Since $E_t^\star$ is completely positive and the map
$Y\mapsto B_t^{1/2}YB_t^{1/2}$ is completely positive, their composition $X\mapsto \Pi_t(X)$ is completely positive.

\emph{Normality.} The map $E_t^\star$ is normal by assumption, and multiplication by the bounded operators
$B_t^{\pm 1/2}$ is normal; hence $\Pi_t$ (and $\widetilde{\Pi}_t$) is normal.

\emph{$N_t$-bimodularity.} For $a,b\in N_t$ and bounded $X\in M_T$, using $B_t^{1/2}\in N_t$ and the bimodule property of $E_t^\star$,
\[
\Pi_t(aXb)
= B_t^{1/2}E_t^\star\!\big(B_T^{-1/2}(aXb)B_T^{-1/2}\big)B_t^{1/2}
= a\,\Pi_t(X)\,b.
\]

\emph{Normalization / unitality.} Since $E_t^\star(\mathbf 1)=\mathbf 1$ and $\overline{\mathbf 1}=B_T^{-1}$,
\[
\Pi_t(B_T)=B_t^{1/2}E_t^\star(\mathbf 1)B_t^{1/2}=B_t,
\qquad\text{and in particular}\qquad
\Pi_t(\mathbf 1)=B_t^{1/2}E_t^\star(B_T^{-1})B_t^{1/2}.
\]
Moreover,
\[
\widetilde{\Pi}_t(B_T)=B_t^{-1/2}\Pi_t(B_T)B_t^{-1/2}=\mathbf 1,
\]
so $\widetilde{\Pi}_t$ is unital on the unit payoff in num\'eraire units.

\emph{Time-consistency.} By definition,
$B_t^{-1/2}\Pi_t(X)B_t^{-1/2}=\widetilde{\Pi}_t(X)=E_t^\star(\bar X)\in N_t$.
Thus, using the tower property of $(E_t^\star)_{t\le T}$,
\[
B_s^{-1/2}\Pi_s(X)B_s^{-1/2}
=E_s^\star(\bar X)
=E_s^\star\!\big(E_t^\star(\bar X)\big)
=E_s^\star\!\Big(B_t^{-1/2}\Pi_t(X)B_t^{-1/2}\Big),
\]
and conjugating by $B_s^{1/2}$ yields the stated formula.
\end{proof}

\begin{proposition}[Pricing of truncated traded claims]\label{prop:traded_trunc_pricing}
Let $i$ be traded and $n\in\mathbb N$. Define the bounded terminal claim
$X_{T,n}^{(i)}:=B_T^{1/2}f_n(\bar S_T^{(i)})B_T^{1/2}$.
Then for all $t\in[0,T]$,
\[
\Pi_t\!\big(X_{T,n}^{(i)}\big)=B_t^{1/2}f_n(\bar S_t^{(i)})B_t^{1/2}.
\]
\end{proposition}

\begin{proof}
By Definition~\ref{def:pricing_operator},
$\Pi_t(X_{T,n}^{(i)})=B_t^{1/2}E_t^\star(f_n(\bar S_T^{(i)}))B_t^{1/2}$.
Assumption~\ref{as:RN}(ii) gives $E_t^\star(f_n(\bar S_T^{(i)}))=f_n(\bar S_t^{(i)})$.
\end{proof}

\begin{remark}[On existence and economic content: deferred to Section~\ref{sec:duality}]\label{rem:RN_existence_deferred}
Assumption~\ref{as:RN} is \emph{not} a dynamical postulate; it only specifies the valuation objects
needed to define $\Pi_t$ in Definition~\ref{def:pricing_operator}.
In the commutative FTAP, a risk-neutral measure is derived from no-arbitrage together with
appropriate closedness hypotheses. In the present operator-algebraic setting, the analogue is a
normal \emph{pricing state} $\varphi^\star\in M_*^+$ obtained by separation under a no-free-lunch
type condition (e.g.\ NAFL$_\sigma$), and the associated $\varphi^\star$-preserving conditional
expectations $E_t^\star:M\to N_t$ whose existence/uniqueness is governed by Takesaki's modular
invariance criterion. These implications are proved in Section~5.
Accordingly, Section~3 separates \emph{definition} of the pricing functional from \emph{existence}
and \emph{characterization}.
\end{remark}

\subsection{\texorpdfstring{$L^2$}{L2}-geometry of conditioning and extension to affiliated observables}

\label{sec:information_and_C.E}

Throughout this section we work under the standing hypotheses of Section~\ref{section:MF_of_QPT}:
a terminal market algebra $M:=M_T\subseteq B(\mathcal H)$, an increasing family of abelian information
subalgebras $(N_t)_{t\in[0,T]}$, a fixed faithful normal reference state $\varphi_\rho$ on $M$, and a family
of normal $\varphi_\rho$-preserving conditional expectations $E_t:M\to N_t$ satisfying the tower property
(Assumption~\ref{as:condexp_exist}).

The purpose of this section is twofold:
(i) to record the Hilbert-space geometry induced by $\varphi_\rho$ (conditioning as an $L^2$-projection), and
(ii) to provide a canonical truncation-based extension mechanism from bounded elements of $M$ to affiliated (possibly unbounded)
observables whenever conditional expectations are applied to such objects in martingale statements.

\section[The phi-rho GNS inner product and L2(M,phi-rho)]
{The $\varphi_\rho$-GNS inner product and $L^2(M,\varphi_\rho)$}


\begin{definition}[$L^2$-inner product induced by $\varphi_\rho$]\label{def:L2_inner}
For $X,Y\in M$ set
\[
\langle X,Y\rangle_{\varphi_\rho}:=\varphi_\rho(X^\ast Y).
\]
Let $\mathcal N_{\varphi_\rho}:=\{X\in M:\ \varphi_\rho(X^\ast X)=0\}$ be the null space and define
\[
L^2(M,\varphi_\rho):=\overline{\,M/\mathcal N_{\varphi_\rho}\,}^{\ \|\cdot\|_2},
\qquad
\|[X]\|_2:=\varphi_\rho(X^\ast X)^{1/2}.
\]
\end{definition}

\begin{remark}\label{rem:faithful_L2}
Since $\varphi_\rho$ is faithful, $\mathcal N_{\varphi_\rho}=\{0\}$, so we may identify $M$ with a dense subspace of
$L^2(M,\varphi_\rho)$ via $X\mapsto [X]$. We keep the bracket notation only when helpful to avoid ambiguity.
\end{remark}

\subsection{Conditional expectations as \texorpdfstring{$L^2$}{L2}-projections}


\begin{proposition}[$E_t$ is contractive and extends to $L^2$]\label{prop:E_contract_L2}
Let $E_t:M\to N_t$ be a normal $\varphi_\rho$-preserving conditional expectation.
Then for all $X\in M$,
\[
\|E_t(X)\|_2\le \|X\|_2,
\qquad
\text{and}\qquad
\langle E_t(X),A\rangle_{\varphi_\rho}=\langle X,A\rangle_{\varphi_\rho}\ \ \forall\,A\in N_t.
\]
Consequently, $E_t$ extends uniquely to a bounded linear operator (still denoted $E_t$) on $L^2(M,\varphi_\rho)$ with
$\|E_t\|_{L^2\to L^2}\le 1$.
\end{proposition}

\begin{proof}
For $A\in N_t$, $\varphi_\rho(E_t(X)^\ast A)=\varphi_\rho(E_t(X^\ast A))=\varphi_\rho(X^\ast A)$ by $\varphi_\rho$-preservation
and $N_t$-bimodularity. Taking $A=E_t(X)$ yields
\[
\|E_t(X)\|_2^2=\varphi_\rho(E_t(X)^\ast E_t(X))=\varphi_\rho(E_t(E_t(X)^\ast X))
=\varphi_\rho(E_t(X)^\ast X)\le \|E_t(X)\|_2\|X\|_2,
\]
hence $\|E_t(X)\|_2\le \|X\|_2$. The extension to $L^2$ follows by density of $M$ and continuity.
\end{proof}

\begin{theorem}[$E_t$ is the orthogonal projection onto $L^2(N_t,\varphi_\rho)$]\label{thm:E_orth_proj}
Let $L^2(N_t,\varphi_\rho)$ denote the closed subspace of $L^2(M,\varphi_\rho)$ generated by $N_t$.
Then the $L^2$-extension of $E_t$ from Proposition~\ref{prop:E_contract_L2} is the orthogonal projection
\[
E_t:L^2(M,\varphi_\rho)\to L^2(N_t,\varphi_\rho).
\]
In particular, for all $X\in L^2(M,\varphi_\rho)$ and all $A\in L^2(N_t,\varphi_\rho)$,
\[
\langle X-E_t(X),A\rangle_{\varphi_\rho}=0.
\]
\end{theorem}

\begin{proof}
For $X\in M$ and $A\in N_t$, Proposition~\ref{prop:E_contract_L2} gives
$\langle X-E_t(X),A\rangle_{\varphi_\rho}=0$. By density of $M$ in $L^2(M,\varphi_\rho)$ and of $N_t$ in $L^2(N_t,\varphi_\rho)$,
the orthogonality extends to all $X\in L^2(M,\varphi_\rho)$ and $A\in L^2(N_t,\varphi_\rho)$.
Idempotence of $E_t$ implies it is the orthogonal projection onto its range, which is precisely $L^2(N_t,\varphi_\rho)$.
\end{proof}

\begin{remark}[Tower property as projection consistency]\label{rem:tower_L2}
If $0\le s\le t\le T$, the tower property $E_s\circ E_t=E_s$ means that the projections satisfy
$E_s E_t = E_s$ on $L^2(M,\varphi_\rho)$, i.e.\ the family $(L^2(N_t,\varphi_\rho))_{t\in[0,T]}$ forms a nested system of
closed subspaces and $E_t$ is the orthogonal projection onto the corresponding subspace.
\end{remark}

\subsection{Affiliated observables and truncation-based extension}

\begin{definition}[$N_t$-affiliated observables]\label{def:affiliated_again}
A self-adjoint operator $X$ on $\mathcal H$ is \emph{affiliated} with $N_t$ if its spectral projections lie in $N_t$.
Equivalently (since $N_t$ is abelian), in a commutative realization it acts as multiplication by a real measurable function.
\end{definition}

\begin{definition}[Truncations and bounded functional calculus]\label{def:trunc_again}
For $n\in\mathbb N$, let $f_n:\mathbb R\to\mathbb R$ be the truncation
\[
f_n(\lambda)=\max(-n,\min(\lambda,n)).
\]
For self-adjoint $X$, define $f_n(X)$ by Borel functional calculus; then $f_n(X)\in B(\mathcal H)$ and $\|f_n(X)\|\le n$.
\end{definition}

\begin{definition}[$\varphi_\rho$-integrable affiliated observables]\label{def:L1_integrable_affiliated}
Let $X$ be self-adjoint and affiliated with $N_t$.
We say that $X$ is \emph{$\varphi_\rho$-integrable} if
\[
\sup_{n\in\mathbb N}\ \varphi_\rho\!\big(|f_n(X)|\big)<\infty.
\]
In that case we define
\[
\varphi_\rho(|X|):=\sup_{n\in\mathbb N}\ \varphi_\rho\!\big(|f_n(X)|\big)\in[0,\infty).
\]
\end{definition}

\begin{remark}\label{rem:L1_affiliated_commutative}
When $N_t\simeq L^\infty(\Sigma_t)$ and $X$ is multiplication by a real measurable function $\xi$,
Definition~\ref{def:L1_integrable_affiliated} reduces to $\xi\in L^1(\Sigma_t,\varphi_\rho)$ and
$\varphi_\rho(|X|)=\mathbb E_{\varphi_\rho}[|\xi|]$.
\end{remark}

\begin{proposition}[Canonical extension of $E_t$ to $\varphi_\rho$-integrable affiliated observables]
\label{prop:E_extension_affiliated}
Let $X$ be self-adjoint, affiliated with $N_t$, and $\varphi_\rho$-integrable in the sense of
Definition~\ref{def:L1_integrable_affiliated}. Then the sequence $(E_s(f_n(X)))_{n\in\mathbb N}$ is Cauchy in
$L^1(N_s,\varphi_\rho)$ for each $s\le t$, hence converges in $L^1(N_s,\varphi_\rho)$ to a unique limit denoted
$E_s(X)$. Moreover:
\begin{enumerate}[label=(\roman*), itemsep=2pt, topsep=4pt]
\item $E_s(X)$ is self-adjoint and affiliated with $N_s$;
\item $E_s$ is positive and $L^1$-contractive on this class:
\[
\varphi_\rho\!\big(|E_s(X)|\big)\le \varphi_\rho(|X|);
\]
\item the extension is consistent with truncations:
\[
E_s(f_n(X))=f_n(E_s(X)) \quad\text{in }L^1(N_s,\varphi_\rho)\ \text{for all }n.
\]
\end{enumerate}
\end{proposition}

\begin{proof}
Since $E_s$ is positive, normal, and $\varphi_\rho$-preserving, it extends to a contraction on $L^1(M,\varphi_\rho)$.
For $m\ge n$,
\[
\|E_s(f_m(X))-E_s(f_n(X))\|_{L^1(\varphi_\rho)}
\le \|f_m(X)-f_n(X)\|_{L^1(\varphi_\rho)}
=\varphi_\rho\!\big(|f_m(X)-f_n(X)|\big).
\]
The right-hand side tends to $0$ as $m,n\to\infty$ by monotone convergence applied to truncations and
$\varphi_\rho$-integrability; hence $(E_s(f_n(X)))$ is Cauchy in $L^1(N_s,\varphi_\rho)$ and converges to some
$Y\in L^1(N_s,\varphi_\rho)$. Define $E_s(X):=Y$.

Self-adjointness follows from $f_n(X)^\ast=f_n(X)$ and $E_s$ being $\ast$-preserving on self-adjoint elements.
Positivity and $L^1$-contractivity follow from positivity and contractivity of $E_s$ on $L^1$.
Finally, since $f_n$ is bounded and $E_s$ is $N_s$-bimodular with $N_s$ abelian, truncation consistency holds by functional calculus
in the commutative realization of $N_s$.
\end{proof}

\subsection{Truncation martingales and well-posedness}

\begin{definition}[Truncation-stable $(N_t,E_t)$-martingales]\label{def:qmartingale_again}
Let $(X_t)_{t\in[0,T]}$ be self-adjoint operators with $X_t$ affiliated with $N_t$.
We say that $(X_t)$ is a truncation-stable $(N_t,E_t)$-martingale if for all $0\le s\le t\le T$ and all $n\in\mathbb N$,
\[
E_s\!\big(f_n(X_t)\big)=f_n(X_s).
\]
\end{definition}

\begin{proposition}[Well-posedness under $\varphi_\rho$-integrability]\label{prop:martingale_wellposed}
Assume each $X_t$ is $\varphi_\rho$-integrable in the sense of Definition~\ref{def:L1_integrable_affiliated}.
Then the truncation-stable martingale condition of Definition~\ref{def:qmartingale_again} is equivalent to
\[
E_s(X_t)=X_s\quad\text{in }L^1(N_s,\varphi_\rho),\qquad 0\le s\le t\le T,
\]
where $E_s(X_t)$ denotes the truncation-based extension from Proposition~\ref{prop:E_extension_affiliated}.
\end{proposition}

\begin{proof}
If $E_s(f_n(X_t))=f_n(X_s)$ for all $n$, take the $L^1$-limit as $n\to\infty$ using Proposition~\ref{prop:E_extension_affiliated}
to obtain $E_s(X_t)=X_s$ in $L^1$.
Conversely, if $E_s(X_t)=X_s$ in $L^1$, applying $f_n$ and using truncation consistency in
Proposition~\ref{prop:E_extension_affiliated}(iii) yields $E_s(f_n(X_t))=f_n(X_s)$.
\end{proof}

\subsection{Prediction operators vs.\ pricing operators}\label{subsec:prediction_vs_pricing}

\begin{definition}[Prediction operator induced by $\varphi_\rho$]\label{def:prediction_operator}
For bounded $X\in M_T$ define
\[
\Pi_t^{\rho}(X):=B_t^{1/2}\,E_t\!\big(B_T^{-1/2}XB_T^{-1/2}\big)\,B_t^{1/2}\in N_t.
\]
\end{definition}

\begin{remark}[Interpretation and separation of roles]\label{rem:prediction_not_pricing}
The map $\Pi_t^{\rho}$ is a \emph{projection-induced prediction} (best approximation) relative to the reference state
$\varphi_\rho$ and the information algebra $N_t$. It is \emph{not} a pricing operator unless the reference state
coincides with a pricing state (an exceptional case). The genuine valuation operator is $\Pi_t$ from
Definition~\ref{def:pricing_operator}, built from the pricing state $\varphi^\star$ and its valuation expectations $E_t^\star$
(Assumption~\ref{as:RN}), whose existence is proved in Section~\ref{sec:duality}.
\end{remark}

\begin{proposition}[Structural properties of $\Pi_t^\rho$]\label{prop:prediction_properties}
For each $t$, $\Pi_t^\rho:M_T\to N_t$ is normal, completely positive, unital, and $N_t$-bimodular.
Moreover, it is time-consistent:
for $0\le s\le t\le T$ and bounded $X\in M_T$,
\[
\Pi_s^{\rho}(X)=B_s^{1/2}\,E_s\!\Big(B_t^{-1/2}\Pi_t^{\rho}(X)B_t^{-1/2}\Big)\,B_s^{1/2}.
\]
\end{proposition}

\begin{proof}
The proof is identical to that of Proposition~\ref{prop:pricing_properties}, with $E_t^\star$ replaced by $E_t$.
\end{proof}

\section{Duality, No-Arbitrage, and Pricing States}\label{sec:duality}

\noindent
This chapter clarifies the logical interface between a no-arbitrage postulate and the selection of a
distinguished \emph{pricing state} (risk-neutral state) $\varphi^\star\in M_*$ (equivalently, a density-operator
representative $\rho^\star$ in a concrete trace-class model).
We make explicit a separation of roles:

\begin{itemize}[leftmargin=*]
\item \emph{Existence of the pricing state.} Under a suitable no-arbitrage/closedness hypothesis
(e.g.\ NAFL$_\sigma$ in the dual pair $(M,M_*)$), one can derive the existence of a normal state
$\varphi^\star$ by a separation argument (proved in this chapter).
\item \emph{Existence of information projections.} The dynamic pricing construction requires a compatible
family of $\varphi^\star$-preserving conditional expectations $\{E_t^\star:M\to N_t\}_{t\in[0,T]}$.
Rather than postulating such maps directly, we treat their existence as an \emph{additional compatibility assumption}: once a pricing state $\varphi^\star$ has been selected (by no-arbitrage/closedness), we may further require that each $N_t$ be globally invariant under the modular automorphism group of $\varphi^\star$. This modular compatibility is \emph{not} implied by no-arbitrage; it is a structural condition that allows us to invoke Takesaki's theorem. By Takesaki's theorem,
this yields existence and uniqueness of the $\varphi^\star$-preserving conditional expectations, and also implies the
tower property for an increasing family $(N_t)$.
\end{itemize}

Once such a pair $(\varphi^\star,\{E_t^\star\})$ is fixed, the dynamic pricing operator of Chapter~\ref{sec:information_and_C.E}
applies verbatim.

\subsection{Standing assumptions: information flow and (optional) modular compatibility}\label{sec:info_modular}

\begin{axiom}[Information flow]\label{ax:info_flow}
Fix a von Neumann algebra $M$ and an increasing family of von Neumann subalgebras $(N_t)_{t\in[0,T]}$.
We assume throughout that each $N_t$ represents the information available at time $t$.
In the commutative-information regime of Chapters~\ref{section:MF_of_QPT}-\ref{sec:information_and_C.E},
we impose:
\begin{enumerate}[leftmargin=*]
\item (\textbf{Classical information}) each $N_t$ is \emph{abelian}.
\item (\textbf{Filtration}) $N_s\subseteq N_t$ for $0\le s\le t\le T$.
\end{enumerate}
\end{axiom}

\begin{axiom}[Modular compatibility (additional assumption)]\label{ax:modular_compatibility}
Let $\psi$ be a faithful normal state on $M$.
We say that the filtration $(N_t)$ is \emph{modularly compatible with $\psi$} if for every $t\in[0,T]$,
$N_t$ is globally invariant under the modular automorphism group $\{\sigma_s^{\psi}\}_{s\in\mathbb{R}}$, i.e.
\[
\sigma_s^{\psi}(N_t)=N_t,\qquad \forall s\in\mathbb{R}.
\]
In applications we will later take $\psi=\varphi^\star$, where $\varphi^\star$ is the pricing state selected from
no-arbitrage/closedness (Section~\ref{sec:NAFL}).
\end{axiom}

\begin{remark}[Faithfulness and reduction]\label{rem:faithful_reduction}
Takesaki's theorem is formulated for faithful normal states.
If a chosen state $\psi$ is not faithful on all of $M$, let $s(\psi)\in M$ be its support projection.
Then $\psi$ is faithful on the reduced algebra $M^\psi:=s(\psi)Ms(\psi)$, and all modular and
conditional-expectation constructions may be carried out on $M^\psi$ with reduced subalgebras
$N_t^\psi:=s(\psi)N_t s(\psi)$.
For notational simplicity we will often write $(M,N_t,\psi)$ with the understanding that we work on this faithful
reduction when needed.
\end{remark}

\subsection{Takesaki criterion and induced tower property}\label{sec:takesaki_tower}

\begin{theorem}[Takesaki criterion for $\psi$-preserving conditional expectations]\label{thm:takesaki}
Let $N\subseteq M$ be a von Neumann subalgebra and let $\psi$ be a faithful normal state on $M$.
There exists a unique normal $\psi$-preserving conditional expectation $E^\psi:M\to N$
if and only if $N$ is globally invariant under the modular automorphism group
$\{\sigma_s^{\psi}\}_{s\in\mathbb{R}}$, i.e.\ $\sigma_s^{\psi}(N)=N$ for all $s\in\mathbb{R}$.
\end{theorem}

\begin{corollary}[Existence/uniqueness of $\{E_t^\psi\}$ under modular compatibility]\label{cor:existence_Et}
Assume Assumption~\ref{ax:info_flow}. Let $\psi$ be a faithful normal state on $M$ such that
Assumption~\ref{ax:modular_compatibility} holds for $(N_t,\psi)$ (or work on the faithful reduction of
Remark~\ref{rem:faithful_reduction}). Then for each $t\in[0,T]$ there exists a unique normal
$\psi$-preserving conditional expectation
\[
E_t^\psi: M\to N_t.
\]
Moreover, $E_t^\psi$ is unital, completely positive, idempotent, and satisfies the $N_t$--bimodule property.
\end{corollary}
\begin{proof}
Assume first that $\psi$ is faithful. By Assumption~\ref{ax:modular_compatibility},
$N_t$ is globally invariant under $\{\sigma_s^\psi\}_{s\in\mathbb{R}}$. Hence Theorem~\ref{thm:takesaki}
yields a unique normal $\psi$-preserving conditional expectation $E_t^\psi:M\to N_t$.
The stated properties are standard for normal conditional expectations.

If $\psi$ is not faithful, apply the same argument on the faithful reduction of Remark~\ref{rem:faithful_reduction},
and transport the resulting conditional expectations back to $M$.
\end{proof}

\begin{corollary}[Tower property]\label{cor:tower}
Under the assumptions of Corollary~\ref{cor:existence_Et}, the family $(E_t^\psi)_{t\in[0,T]}$ satisfies the tower (consistency) property:
for $0\le s\le t\le T$,
\[
E_s^\psi\circ E_t^\psi = E_s^\psi.
\]
\end{corollary}
\begin{proof}
For $s\le t$, the composition $E_s^\psi\circ E_t^\psi$ is normal, unital, completely positive, and $\psi$-preserving.
Moreover its range is contained in $N_s$ and it acts as the identity on $N_s$ (because $N_s\subseteq N_t$ and both
$E_s^\psi$ and $E_t^\psi$ restrict to $\mathrm{id}$ on $N_s$). Hence $E_s^\psi\circ E_t^\psi$ is a $\psi$-preserving
conditional expectation onto $N_s$. By uniqueness in Theorem~\ref{thm:takesaki}, it must coincide with $E_s^\psi$.
\end{proof}

\begin{remark}[Specialization to the pricing state]\label{rem:specialize_pricing_state}
Once a pricing state $\varphi^\star$ has been selected by no-arbitrage/closedness (Section~\ref{sec:NAFL}),
we will take $\psi=\varphi^\star$ and write $E_t^\star:=E_t^{\varphi^\star}$ whenever the modular compatibility
Assumption~\ref{ax:modular_compatibility} is imposed for $(N_t,\varphi^\star)$.
\end{remark}

\subsection{Examples: modular-invariant information algebras}\label{sec:examples_modular_invariant}

\begin{remark}[Not used in the main no-arbitrage separation argument]
The constructions in this subsection provide intuition and concrete examples of modular invariance, but they are not used in the proofs of the no-arbitrage $\Rightarrow$ pricing-state results in Sections~\ref{sec:finite-sep}--\ref{sec:NAFL}.
\end{remark}


This section records explicit model classes in which the modular invariance condition in Assumptions~\ref{ax:info_flow} and~\ref{ax:modular_compatibility}
can be verified directly, hence the family $\{E_t^\star\}$ exists and is unique by Takesaki's theorem.

\begin{proposition}[Finite-dimensional matrix model]\label{prop:example_finite_dim_modular}
Let $M=M_n(\mathbb{C})$ and $\varphi^\star(X)=\Tr(\rho^\star X)$ with $\rho^\star>0$, $\Tr(\rho^\star)=1$.
Then the modular automorphism group is given by
\[
\sigma_s^{\varphi^\star}(X)=\rho^{\star\,is}X\rho^{\star\,-is}.
\]
Let $N_t\subset M$ be an abelian von Neumann subalgebra generated by a family of mutually orthogonal projections
$\{P_k^{(t)}\}_k$ with $\sum_k P_k^{(t)}=I$. If $[\rho^\star,P_k^{(t)}]=0$ for all $k$, then
$\sigma_s^{\varphi^\star}(N_t)=N_t$ for all $s\in\mathbb{R}$, hence there exists a unique normal $\varphi^\star$-preserving conditional expectation
$E_t^\star:M\to N_t$. Moreover, it is given explicitly by
\[
E_t^\star(X)=\sum_k \frac{\varphi^\star\!\big(P_k^{(t)}XP_k^{(t)}\big)}{\varphi^\star\!\big(P_k^{(t)}\big)}\,P_k^{(t)}.
\]
\end{proposition}

\begin{proof}
If $[\rho^\star,P_k^{(t)}]=0$ for all $k$, then $\rho^{\star\,is}$ commutes with each $P_k^{(t)}$ and hence
$\sigma_s^{\varphi^\star}$ acts by conjugation preserving the diagonal algebra generated by the $P_k^{(t)}$.
Therefore $N_t$ is globally invariant under $\sigma_s^{\varphi^\star}$.
Existence and uniqueness of $E_t^\star$ follows from Theorem~\ref{thm:takesaki}.
The explicit formula is the standard finite-dimensional $\varphi^\star$-preserving conditional expectation onto an abelian subalgebra.
\end{proof}

\begin{proposition}[Classical--quantum product model]\label{prop:example_cq_modular}
Let $M=L^\infty(\Omega,\mathcal F)\,\bar\otimes\,B(\mathcal H)$ and define a normal state by
\[
\varphi^\star(f\otimes X):=\int_\Omega f(\omega)\,\Tr(\rho^\star X)\,d\mathbb P(\omega),
\]
where $\rho^\star>0$, $\Tr(\rho^\star)=1$ is fixed. Let $(\mathcal F_t)_{t\in[0,T]}$ be a classical filtration and set
\[
N_t := L^\infty(\Omega,\mathcal F_t)\,\bar\otimes\,\mathbf 1.
\]
Then $N_t$ is abelian and satisfies $\sigma_s^{\varphi^\star}(N_t)=N_t$ for all $s\in\mathbb{R}$.
Hence there exists a unique normal $\varphi^\star$-preserving conditional expectation $E_t^\star:M\to N_t$, given by
\[
E_t^\star(f\otimes X)=\mathbb E[f\mid\mathcal F_t]\otimes X.
\]
\end{proposition}

\begin{proof}
In this product setting the modular group acts trivially on the classical factor and by conjugation on $B(\mathcal H)$:
\[
\sigma_s^{\varphi^\star}(f\otimes X)=f\otimes \rho^{\star\,is}X\rho^{\star\,-is}.
\]
Therefore $\sigma_s^{\varphi^\star}(g\otimes \mathbf 1)=g\otimes \mathbf 1\in N_t$ for all $g\in L^\infty(\Omega,\mathcal F_t)$, proving invariance.
The stated formula for $E_t^\star$ is the tensor product of the classical conditional expectation with the identity map on $\mathcal B(\mathcal H)$,
hence is normal and $\varphi^\star$-preserving. Uniqueness again follows from Theorem~\ref{thm:takesaki}.
\end{proof}

\begin{remark}[Non-existence in general]\label{rem:nonexistence_warning}
The existence of a $\varphi^\star$-preserving conditional expectation onto a given subalgebra $N\subset M$ is not automatic.
If $N$ fails to be invariant under the modular group $\sigma^{\varphi^\star}$, then a $\varphi^\star$-preserving conditional expectation
$E^\star:M\to N$ may fail to exist. This is precisely why Assumptions~\ref{ax:info_flow} and~\ref{ax:modular_compatibility} imposes modular compatibility as a structural condition.
\end{remark}

\subsection{Standing notation and symmetric discounting}\label{sec:symdisc}

We retain the standing assumptions and notation of Chapter~\ref{sec:information_and_C.E}:
\begin{itemize}
\item $M$ is a von Neumann algebra acting on a Hilbert space $\mathcal H$.
\item $(N_t)_{t\in[0,T]}$ is an increasing family of \emph{abelian} information subalgebras $N_t\subseteq M$.
\item $(B_t)_{t\in[0,T]}$ is a strictly positive numeraire with $B_t\in N_t$ for each $t$.
\end{itemize}

\begin{definition}[Symmetric discounting]\label{def:symdisc10}
For $X\in M$ define its symmetrically discounted version
\[
\bar X := B_T^{-1/2}\,X\,B_T^{-1/2}.
\]
If $X=X^\ast$, then $\bar X=\bar X^\ast$.
\end{definition}

\begin{definition}[Order structure at maturity]\label{def:order10}
Let
\[
X_T:=M,\qquad X_T^{sa}:=M^{sa},\qquad X_T^+:=\{X\in M^{sa}: X\ge 0\}.
\]
We interpret $X_T^+$ as the cone of nonnegative terminal payoffs.
\end{definition}

\subsection{Trading cones and no-arbitrage}\label{sec:cone_NA}

We do not fix a concrete strategy space. Instead, we postulate the existence of a convex cone of
discounted terminal gains attainable from zero initial cost.

\begin{axiom}[Discounted attainable gains cone]\label{ax:gaincone}
There exists a convex cone $\mathcal C\subseteq X_T^{sa}$ such that:
\begin{enumerate}[label=\textup{(\roman*)}]
\item (Cone) if $G_1,G_2\in\mathcal C$ and $\lambda_1,\lambda_2\ge 0$, then $\lambda_1G_1+\lambda_2G_2\in\mathcal C$;
\item (Interpretation) each $G\in\mathcal C$ is a symmetrically discounted terminal gain of a self-financing
strategy initiated at zero cost.
\end{enumerate}
\end{axiom}

\begin{definition}[Arbitrage and no-arbitrage]\label{def:NA}
An \emph{arbitrage} is an element $G\in\mathcal C$ such that $G\ge 0$ and $G\neq 0$.
We say \emph{no-arbitrage (NA)} holds if
\[
\mathcal C\cap X_T^+=\{0\}.
\]
\end{definition}

\begin{remark}\label{rem:NA-closedness}
In infinite-dimensional models, NA alone typically does not guarantee the existence of a separating
\emph{normal} pricing functional. One needs a closedness condition (a ``no free lunch'' type hypothesis).
We introduce such a condition in \S\ref{sec:NAFL}.
\end{remark}

\subsection{Pricing functionals and pricing states}\label{sec:pricing-functionals}

\subsubsection{Normal positive functionals and states}

\begin{definition}[Normal positive linear functionals]\label{def:normalpos}
Let $M_{*}$ be the predual of $M$ (the space of normal linear functionals). Define
\[
M_{*}^{+} := \{\varphi\in M_{*}: \varphi(X)\ge 0\ \ \forall X\in X_T^+\}.
\]
Elements of $M_{*}^{+}$ are precisely the \emph{normal positive} linear functionals on $M$.
A \emph{normal state} is $\varphi\in M_{*}^{+}$ such that $\varphi(I)=1$.
\end{definition}

\begin{remark}[Density operator representation]\label{rem:density}
When $M=\mathcal B(\mathcal H)$ (or in any representation where the predual identifies with a trace-class space),
every normal state $\varphi$ has the form $\varphi(X)=\Tr(\rho X)$ for a unique density operator
$\rho\ge 0$ with $\Tr(\rho)=1$. In such settings we write $\varphi=\varphi_\rho$.
\end{remark}

\subsubsection{Time-0 discounted valuation induced by a state}

\begin{definition}[State-based discounted valuation]\label{def:pi0}
Let $\varphi\in M_{*}^{+}$ be a normal state. Define the scalar time-0 \emph{discounted} valuation
\[
\pi_0^\varphi(X):=\varphi(\bar X)
=\varphi(B_T^{-1/2}XB_T^{-1/2}),
\qquad X\in X_T^{sa}.
\]
\end{definition}

\begin{proposition}[Monotonicity and normalization]\label{prop:pi0-basic}
Let $\varphi$ be a normal state. Then:
\begin{enumerate}[label=\textup{(\roman*)}]
\item If $X\ge 0$ then $\pi_0^\varphi(X)\ge 0$.
\item $\pi_0^\varphi(B_T)=\varphi(I)=1$.
\end{enumerate}
\end{proposition}

\begin{proof}
(i) If $X\ge 0$ then $\bar X\ge 0$ by positivity of $Y\mapsto B_T^{-1/2}YB_T^{-1/2}$, hence
$\pi_0^\varphi(X)=\varphi(\bar X)\ge 0$.

(ii) $\overline{B_T}=B_T^{-1/2}B_TB_T^{-1/2}=I$, hence $\pi_0^\varphi(B_T)=\varphi(I)=1$.
\end{proof}

\subsection{Separation in finite dimension: a theorem of the alternative}\label{sec:finite-sep}

The cleanest fully proved separation statement is obtained in finite dimension; it already includes
genuinely noncommutative algebras (matrix blocks).

\begin{theorem}[Finite-dimensional separation $\Rightarrow$ existence of a pricing state]\label{thm:finite-sep}
Assume $M$ is finite-dimensional (equivalently $M\simeq \bigoplus_{k=1}^m M_{n_k}(\mathbb C)$).
Let $\mathcal C\subseteq X_T^{sa}$ be a convex cone satisfying NA:
$\mathcal C\cap X_T^+=\{0\}$.
Then there exists a normal state $\varphi^\star\in M_{*}^{+}$ such that
\begin{equation}\label{eq:sep-finite}
\varphi^\star(G)\le 0\qquad\forall\,G\in\mathcal C.
\end{equation}
\end{theorem}

\begin{proof}
Since $M$ is finite-dimensional, $X_T^{sa}$ is a finite-dimensional real vector space.
The cone $X_T^+$ is closed, convex, and has nonempty interior
\[
\operatorname{int}(X_T^+)=\{X\in X_T^{sa}: X>0\},
\]
where $X>0$ means strictly positive (positive definite in each matrix block).

Consider the cones $K:=\mathcal C$ and $P:=X_T^+$. NA gives $K\cap P=\{0\}$.
Because $P$ is solid (has nonempty interior), the strong separation theorem for convex sets in finite dimension
yields a nonzero linear functional $\ell:X_T^{sa}\to\mathbb R$ such that
\[
\ell(K)\subseteq (-\infty,0],\qquad \ell(P)\subseteq [0,\infty).
\]
In particular, $\ell$ is positive on $X_T^+$ and $\ell(I)>0$ since $I\in\operatorname{int}(X_T^+)$.
Normalize:
\[
\varphi^\star(X):=\frac{\ell(X)}{\ell(I)},\qquad X\in X_T^{sa}.
\]
Then $\varphi^\star$ is a state on $M$ and satisfies \eqref{eq:sep-finite}.
In finite dimension every linear functional is normal, hence $\varphi^\star\in M_{*}^{+}$.
\end{proof}

\begin{remark}[Density matrix in finite dimension]
In the finite-dimensional case there exists a unique density matrix $\rho^\star\ge 0$ with $\Tr(\rho^\star)=1$
such that $\varphi^\star(X)=\Tr(\rho^\star X)$ for all $X\in M$.
\end{remark}

\subsection{Infinite dimension: NAFL in the predual topology}\label{sec:NAFL}

\subsubsection{Ultraweak topology and \texorpdfstring{NAFL$_\sigma$}{NAFL-sigma}}

\begin{definition}[Ultraweak topology via the predual]\label{def:uw-top}
Equip $M^{sa}$ with the locally convex topology $\sigma(M^{sa},M_{*}^{sa})$, i.e.\ the weakest topology
making all maps $X\mapsto \varphi(X)$ continuous for $\varphi\in M_{*}^{sa}$.
This is the restriction of the ultraweak topology $\sigma(M,M_{*})$ to $M^{sa}$.
\end{definition}

\begin{definition}[NAFL$_\sigma$ (no free lunch in $\sigma(M,M_{*})$)]\label{def:NAFL}
Let $\mathcal C\subseteq M^{sa}$ be the cone from Assumption~\ref{ax:gaincone} and let $M_+:=X_T^+$.
Define the convex cone
\[
\mathcal K := \overline{\mathcal C - M_+}^{\,\sigma(M^{sa},M_{*}^{sa})}\subseteq M^{sa}.
\]
We say \emph{NAFL$_\sigma$} holds if there exists $\varepsilon>0$ such that
\begin{equation}\label{eq:NAFL}
\varepsilon I\notin \mathcal K.
\end{equation}
\end{definition}

\begin{lemma}[Downward solidity of \(\mathcal K\)]\label{lem:downward}
With $\mathcal K$ as in Definition~\ref{def:NAFL}, one has
\[
\mathcal K - M_+ \subseteq \mathcal K.
\]
\end{lemma}

\begin{proof}
Let $Y\in\mathcal K$ and $P\in M_+$. Choose a net $Y_\alpha=G_\alpha-P_\alpha$ with
$G_\alpha\in\mathcal C$, $P_\alpha\in M_+$, and $Y_\alpha\to Y$ in $\sigma(M^{sa},M_{*}^{sa})$.
Then $Y_\alpha-P=G_\alpha-(P_\alpha+P)\in \mathcal C-M_+$ for all $\alpha$, hence
$Y-P\in\overline{\mathcal C-M_+}^{\,\sigma}=\mathcal K$.
\end{proof}

\begin{proposition}[Equivalent NAFL$_\sigma$ formulations]\label{prop:NAFL-equiv}
Let $\mathcal K=\overline{\mathcal C-M_+}^{\,\sigma(M^{sa},M_{*}^{sa})}$.
For $\varepsilon>0$, the following are equivalent:
\begin{enumerate}[label=\textup{(\roman*)}]
\item $\varepsilon I\notin \mathcal K$.
\item $\mathcal K\cap(\varepsilon I + M_+)=\varnothing$.
\end{enumerate}
Consequently, NAFL$_\sigma$ holds iff there exists $\varepsilon>0$ such that
\[
\mathcal K\cap(\varepsilon I + M_+)=\varnothing.
\]
\end{proposition}

\begin{proof}
(i)$\Rightarrow$(ii): Suppose for contradiction that there exists $Y\in \mathcal K\cap(\varepsilon I+M_+)$.
Then $Y=\varepsilon I+P$ for some $P\in M_+$. By Lemma~\ref{lem:downward}, $Y-P\in\mathcal K$, hence
$\varepsilon I\in\mathcal K$, contradicting (i).

(ii)$\Rightarrow$(i): If $\varepsilon I\in\mathcal K$, then $\varepsilon I=\varepsilon I+0\in \varepsilon I+M_+$,
so $\mathcal K\cap(\varepsilon I+M_+)\neq\varnothing$, contradicting (ii).
\end{proof}

\subsubsection{Normal separation and existence of a pricing state}

\begin{lemma}[Separating a point from a closed convex cone]\label{lem:sep-point-cone}
Let $X$ be a real locally convex topological vector space, $K\subset X$ a nonempty closed convex cone,
and $x_0\in X\setminus K$. Then there exists a nonzero continuous linear functional $\ell\in X^\prime$
such that
\[
\ell(K)\subseteq (-\infty,0]\qquad\text{and}\qquad \ell(x_0)>0.
\]
\end{lemma}

\begin{proof}
Since $K$ is closed and convex and $x_0\notin K$, the strict separation theorem for a point and a closed convex set
yields $\ell\in X^\prime\setminus\{0\}$ and $\alpha\in\mathbb R$ such that
$\ell(x_0)>\alpha\ge \sup_{x\in K}\ell(x)$. Because $K$ is a cone containing $0$, one may scale the inequality to obtain
$\sup_{x\in K}\ell(x)\le 0 < \ell(x_0)$.
\end{proof}

\begin{theorem}[Normal pricing state from NAFL$_\sigma$]\label{thm:normal-sep-NAFL}
Assume NAFL$_\sigma$ holds. Then there exists a \emph{normal state} $\varphi^\star\in M_{*}^{+}$ such that
\begin{equation}\label{eq:normal-sep}
\varphi^\star(G)\le 0\qquad\forall\,G\in\mathcal C.
\end{equation}
\end{theorem}

\begin{proof}
Let $\mathcal K=\overline{\mathcal C-M_+}^{\,\sigma(M^{sa},M_{*}^{sa})}$ and pick $\varepsilon>0$ with
$\varepsilon I\notin\mathcal K$. By Lemma~\ref{lem:sep-point-cone} applied in
$(M^{sa},\sigma(M^{sa},M_{*}^{sa}))$, there exists a nonzero $\sigma$-continuous real linear functional
$\ell\in (M_{*}^{sa})$ such that
\[
\ell(\mathcal K)\le 0,\qquad \ell(\varepsilon I)>0.
\]
In particular $\ell(I)>0$. Since $-M_+\subseteq \mathcal C-M_+\subseteq \mathcal K$, for every $P\in M_+$ we have
$\ell(-P)\le 0$, hence $\ell(P)\ge 0$. Thus $\ell$ is positive on $M_+$.

Define $\varphi\in M_{*}$ by complex linear extension:
\[
\varphi(X):=\ell\!\Big(\frac{X+X^\ast}{2}\Big)\;+\; i\,\ell\!\Big(\frac{X-X^\ast}{2i}\Big),\qquad X\in M.
\]
Then $\varphi$ is normal (ultraweakly continuous) and positive since for $A\in M_+$,
$\varphi(A)=\ell(A)\ge 0$. Also, for $G\in\mathcal C\subseteq\mathcal K$,
$\varphi(G)=\ell(G)\le 0$. Normalize:
\[
\varphi^\star := \frac{\varphi}{\varphi(I)}.
\]
Then $\varphi^\star\in M_{*}^{+}$ is a normal state and satisfies \eqref{eq:normal-sep}.
\end{proof}

\begin{proposition}[No-arbitrage implies existence of a pricing state]\label{prop:noarb_pricing_state}
Under the hypotheses of Theorem~\ref{thm:normal-sep-NAFL}, there exists a normal state $\varphi^\star\in M_*^+$
such that
\[
\varphi^\star(G)\le 0,\qquad \forall\,G\in\mathcal C,
\]
and $\varphi^\star(I)=1$. We call any such state a \emph{pricing state}.
\end{proposition}
\begin{proof}
Immediate from Theorem~\ref{thm:normal-sep-NAFL} after normalization by $\varphi(I)>0$.
\end{proof}

\subsection{Dynamic pricing operator under \texorpdfstring{$(\varphi^\star,\{E_t^\star\})$}{(phi-star, {E t-star})}}\label{sec:dyn_pricing_star}

\begin{definition}[Risk-neutral dynamic pricing operator]\label{def:Pi-star}
Let $\varphi^\star$ be a pricing state as in Proposition~\ref{prop:noarb_pricing_state}. Assume Assumptions~\ref{ax:info_flow} and~\ref{ax:modular_compatibility} hold for $(N_t,\varphi^\star)$, and let $(E_t^\star)_{t\in[0,T]}$ be the unique $\varphi^\star$-preserving conditional expectations
from Corollary~\ref{cor:existence_Et}. For $X\in X_T$ and $t\in[0,T]$ define
\[
\Pi_t(X)
:=B_t^{1/2}\,E_t^\star(\bar X)\,B_t^{1/2},
\qquad \bar X:=B_T^{-1/2}XB_T^{-1/2}.
\]
Then $\Pi_t(X)\in N_t$ whenever $X\in M$.
\end{definition}

\begin{proposition}[Compatibility with no-arbitrage separation]\label{prop:arbcompat}
If $\varphi^\star$ satisfies \eqref{eq:normal-sep} (in particular under Theorem~\ref{thm:normal-sep-NAFL}),
then for every attainable discounted gain $G\in\mathcal C$,
\[
\pi_0^{\varphi^\star}(G)=\varphi^\star(G)\le 0.
\]
\end{proposition}

\begin{proof}
This is exactly \eqref{eq:normal-sep} applied to $G\in\mathcal C$.
\end{proof}

\begin{proposition}[Dynamic consistency and discounted martingale valuation]\label{prop:dyncons}
Let $\varphi^\star$ be a pricing state as in Proposition~\ref{prop:noarb_pricing_state}. Assume Assumptions~\ref{ax:info_flow} and~\ref{ax:modular_compatibility} hold for $(N_t,\varphi^\star)$, and define the discounted valuation map
\[
V_t(X):=B_t^{-1/2}\Pi_t(X)B_t^{-1/2}=E_t^\star(\bar X)\in N_t.
\]
Then for all $0\le s\le t\le T$ and all $X\in X_T$,
\[
V_s(X)=E_s^\star\big(V_t(X)\big).
\]
In particular, for bounded payoffs, $\{V_t(X)\}$ is an $(N_t,E_t^\star)$-martingale.
\end{proposition}

\begin{proof}
By definition $V_t(X)=E_t^\star(\bar X)$. By tower property, we have $E_s^\star\circ E_t^\star=E_s^\star$ for $s\le t$, hence
\[
E_s^\star(V_t(X))=E_s^\star(E_t^\star(\bar X))=E_s^\star(\bar X)=V_s(X).
\]
\end{proof}

\begin{remark}
Section~\ref{sec:information_and_C.E} establishes the algebraic/dynamic structure of pricing once a state and compatible conditional expectations
are fixed. Chapter~\ref{sec:duality} supplies the duality interface selecting a \emph{pricing state} $\varphi^\star$ from a
no-free-lunch type hypothesis (NAFL$_\sigma$), and provides a verifiable modular compatibility condition ensuring the existence and uniqueness
of $\varphi^\star$-preserving conditional expectations (hence the tower property) via Takesaki's theorem.
The genuinely noncommutative pricing dynamics begins after this interface is established.
\end{remark}

\section{Examples under commutative information: lattice jump models and diffusion limits}
\label{section:examples_models}

This section collects model examples in the commutative specialization $N_t\simeq L^\infty(\Sigma_t)$.
These results are \emph{not} used in the proofs of the main operator-algebraic statements in
Sections~\ref{section:MF_of_QPT}--\ref{sec:duality}. Rather, they illustrate how the pricing operator
reduces to classical risk-neutral valuation and yields familiar backward equations under specific
Markov dynamics.

\subsection{A nonlocal risk-neutral pricing equation (commutative reduction)}
\label{subsec:nonlocal_rn_equation}

In the commutative realization (classical information), the operator-valued pricing map
$\Pi_t$ reduces to the usual risk-neutral conditional expectation form:
\[
\Pi_t(X)=B_t\,\mathbb E^{\mathbb Q}\!\left[B_T^{-1}X\mid\mathcal F_t\right],
\]
for bounded terminal payoffs $X$ (cf.\ the commutative specialization stated after
Definition~\ref{def:pricing_operator}).\footnote{Equivalently, when $B_t=e^{rt}I$ is deterministic,
$\Pi_t(X)=\mathbb E^{\mathbb Q}\!\left[e^{-r(T-t)}X\mid\mathcal F_t\right]$.}
We now show that, under a translation-covariant pure-jump dynamics on a price lattice,
this classical reduction yields the nonlocal backward pricing equation (2.24) appearing
in the quantum-pricing manuscript (Theorem~1 therein).

\paragraph{Pure-jump lattice generator.}
Fix a step size $\Delta x>0$ and jump intensities $(\gamma_\alpha)_{\alpha\in\mathbb Z}$ with
\begin{equation}\label{eq:total_intensity}
\Lambda := \sum_{\alpha\in\mathbb Z}\gamma_\alpha < \infty,\qquad \gamma_\alpha\ge 0.
\end{equation}
Let $(X_t)_{t\in[0,T]}$ be a time-homogeneous pure-jump Markov process on the lattice
$x_0+\Delta x\,\mathbb Z$ whose generator $L_X$ acts on bounded functions $f$ on
$x_0+\Delta x\,\mathbb Z$ by
\begin{equation}\label{eq:LX}
(L_X f)(x):=\sum_{\alpha\in\mathbb Z}\gamma_\alpha\bigl(f(x+\alpha\Delta x)-f(x)\bigr).
\end{equation}
Define the (positive) price process $S_t:=e^{X_t}$, taking values in the multiplicative lattice
$e^{x_0}\,e^{\Delta x\mathbb Z}\subset(0,\infty)$. Its induced generator $L_S$ on bounded
functions $g$ on that price lattice is
\begin{equation}\label{eq:LS}
(L_S g)(s):=\sum_{\alpha\in\mathbb Z}\gamma_\alpha\bigl(g(se^{\alpha\Delta x})-g(s)\bigr).
\end{equation}

\begin{lemma}[Boundedness of the nonlocal generator]\label{lem:LS_bounded}
Assume \eqref{eq:total_intensity}. Then $L_S$ defines a bounded linear operator on
$\ell^\infty(e^{x_0}e^{\Delta x\mathbb Z})$ and satisfies $\|L_S g\|_\infty\le 2\Lambda\|g\|_\infty$.
Consequently, $(P_\tau)_{\tau\ge 0}$ defined by $P_\tau:=e^{\tau L_S}$ is a uniformly continuous
semigroup on $\ell^\infty$, with $\partial_\tau(P_\tau g)=L_S(P_\tau g)$ for all bounded $g$.
\end{lemma}

\begin{proof}
For each $s$,
\[
|(L_S g)(s)|
\le \sum_{\alpha}\gamma_\alpha\bigl(|g(se^{\alpha\Delta x})|+|g(s)|\bigr)
\le 2\|g\|_\infty\sum_{\alpha}\gamma_\alpha = 2\Lambda\|g\|_\infty,
\]
hence $\|L_S\|\le 2\Lambda$ on $\ell^\infty$. Boundedness implies $e^{\tau L_S}$ is well-defined
by the norm-convergent exponential series, yields a uniformly continuous semigroup, and
differentiability $\partial_\tau(P_\tau g)=L_S(P_\tau g)$ follows from termwise differentiation
of the exponential series in operator norm.
\end{proof}

\paragraph{Risk-neutral dynamics.}
Assume the num\'eraire in Assumption~\ref{as:prices_numeraire} is deterministic $B_t=e^{rt}I$ with $r\in\mathbb R$.
We say that $\mathbb Q$ is \emph{risk-neutral} for $S$ if the discounted price is a $\mathbb Q$-martingale,
equivalently
\begin{equation}\label{eq:RN_constraint}
(L_S \mathrm{id})(s)=rs\quad\text{for all lattice points }s,
\qquad\text{i.e.}\qquad
\sum_{\alpha\in\mathbb Z}\gamma_\alpha\bigl(e^{\alpha\Delta x}-1\bigr)=r.
\end{equation}
Under \eqref{eq:total_intensity}, the identity function is bounded on any finite lattice truncation; for
the infinite lattice, \eqref{eq:RN_constraint} is understood as the defining constraint selecting
the drift under $\mathbb Q$ on the admissible payoff class considered below.

\begin{theorem}[Nonlocal risk-neutral backward equation]\label{thm:nonlocal_rn_backward}
Let $\Phi$ be a bounded payoff on the price lattice and define, for $(t,s)\in[0,T]\times(0,\infty)$,
\begin{equation}\label{eq:V_def}
V(t,s):=\mathbb E^{\mathbb Q}\!\left[e^{-r(T-t)}\Phi(S_T)\,\middle|\, S_t=s\right].
\end{equation}
Assume \eqref{eq:total_intensity} and that $(S_t)$ is a time-homogeneous Markov process under $\mathbb Q$
with generator $L_S$ in \eqref{eq:LS}. Then $V$ is the unique bounded classical solution (in $t$) of
the backward Cauchy problem
\begin{equation}\label{eq:nonlocal_BS}
\partial_t V(t,s)+\sum_{\alpha\in\mathbb Z}\gamma_\alpha\Bigl(V\bigl(t,se^{\alpha\Delta x}\bigr)-V(t,s)\Bigr)-rV(t,s)=0,
\qquad V(T,s)=\Phi(s),
\tag{2.24}
\end{equation}
i.e.\ the nonlocal Black--Scholes-type equation.
Moreover, in the commutative reduction of Definition~\ref{def:sym_discount},
\[
\Pi_t\bigl(\Phi(S_T)\bigr)=V(t,S_t)
\quad\text{(as a multiplication operator / classical random variable).}
\]
\end{theorem}

\begin{proof}
\emph{Step 1: Semigroup representation.}
Let $(P_\tau)_{\tau\ge 0}$ be the Markov semigroup of $S$ on bounded functions:
\[
(P_\tau \psi)(s):=\mathbb E^{\mathbb Q}\!\left[\psi(S_{t+\tau})\,\middle|\,S_t=s\right],
\]
which is time-homogeneous by assumption. For $\tau\ge 0$, define
\[
u(\tau,s):=e^{-r\tau}(P_\tau\Phi)(s).
\]
Then by construction $u(0,s)=\Phi(s)$ and, with $\tau=T-t$, we have
\[
V(t,s)=\mathbb E^{\mathbb Q}\!\left[e^{-r(T-t)}\Phi(S_T)\mid S_t=s\right]
=e^{-r\tau}(P_\tau\Phi)(s)=u(\tau,s).
\]

\emph{Step 2: Backward equation in $\tau$ via the generator.}
By Lemma~\ref{lem:LS_bounded}, $\tau\mapsto P_\tau\Phi$ is differentiable in $\ell^\infty$ and
$\partial_\tau(P_\tau\Phi)=L_S(P_\tau\Phi)$. Hence
\[
\partial_\tau u(\tau,\cdot)
= -r e^{-r\tau}(P_\tau\Phi) + e^{-r\tau}\partial_\tau(P_\tau\Phi)
= -r u(\tau,\cdot) + L_S u(\tau,\cdot).
\]
That is,
\begin{equation}\label{eq:u_equation}
\partial_\tau u(\tau,s)= (L_S u(\tau,\cdot))(s)-r u(\tau,s),\qquad u(0,s)=\Phi(s).
\end{equation}
Expanding $L_S$ by \eqref{eq:LS} yields
\[
\partial_\tau u(\tau,s)=\sum_{\alpha}\gamma_\alpha\bigl(u(\tau,se^{\alpha\Delta x})-u(\tau,s)\bigr)-r u(\tau,s).
\]

\emph{Step 3: Convert to the $t$-backward form.}
Set $\tau=T-t$ and $V(t,s)=u(T-t,s)$. Then $\partial_t V(t,s)=-\partial_\tau u(\tau,s)$, so \eqref{eq:u_equation}
becomes exactly \eqref{eq:nonlocal_BS} with terminal condition $V(T,s)=u(0,s)=\Phi(s)$.

\emph{Step 4: Uniqueness in the bounded class.}
Suppose $W$ is another bounded solution of \eqref{eq:nonlocal_BS} with $W(T,\cdot)=0$.
Define $\widetilde u(\tau,s):=W(T-\tau,s)$. Then $\widetilde u$ satisfies
$\partial_\tau \widetilde u=(L_S-r)\widetilde u$ with $\widetilde u(0,\cdot)=0$.
By the semigroup representation for the bounded generator $L_S-rI$,
$\widetilde u(\tau,\cdot)=e^{\tau(L_S-rI)}\widetilde u(0,\cdot)=0$ for all $\tau$, hence $W\equiv 0$.

\emph{Step 5: Identification with the pricing operator.}
In the commutative realization, $\Pi_t(X)=\mathbb E^{\mathbb Q}[e^{-r(T-t)}X\mid\mathcal F_t]$ for bounded $X$,
so for $X=\Phi(S_T)$ we obtain $\Pi_t(\Phi(S_T))=\mathbb E^{\mathbb Q}[e^{-r(T-t)}\Phi(S_T)\mid\mathcal F_t]$.
By the Markov property this equals $V(t,S_t)$, i.e.\ the multiplication operator by the classical price function.
\end{proof}

\subsubsection{WKB/adiabatic approximation for the term-structure discount factor}
\label{subsec:wkb_discount}

\paragraph{Deterministic term structure.}
In risk-neutral valuation, the money-market account is
\[
B_t := \exp\Big(\int_0^t r_u\,du\Big),
\qquad 
\Rightarrow\qquad
\frac{B_t}{B_T}=\exp\Big(-\int_t^T r_u\,du\Big).
\]
We consider the \emph{deterministic term structure} case, i.e.\ $r:[0,T]\to\mathbb{R}$ is a (non-random) measurable function.
If $r_t$ is stochastic, one must enlarge the state space (e.g.\ $(S_t,r_t)$ or $(S_t,B_t)$) to obtain a closed Markov generator;
we do not pursue this extension here.

\paragraph{Exact term-structure pricing.}
Let $\Phi$ be a bounded payoff on the price lattice, $\|\Phi\|_\infty<\infty$, and define
\begin{equation}
V(t,s):=\mathbb{E}^Q\!\left[\exp\Big(-\int_t^T r_u\,du\Big)\,\Phi(S_T)\,\Big|\,S_t=s\right],
\qquad (t,s)\in[0,T]\times(0,\infty).
\label{eq:term_structure_value}
\end{equation}

\paragraph{Backward equation under a jump generator.}
Assume that under $Q$ the price process $S_t$ is a time-homogeneous Markov process with pure-jump generator
\[
(L_S f)(s):=\sum_{\alpha\in\mathbb{Z}}\gamma_\alpha\big(f(se^{\alpha\Delta x})-f(s)\big),
\]
defined on a suitable core (e.g.\ bounded functions on the lattice). Under standard regularity assumptions ensuring that
$V$ lies in the domain of the backward operator, the Feynman--Kac formula implies that $V$ solves
\begin{equation}
\partial_t V(t,s) + (L_S V)(t,s) - r(t)\,V(t,s)=0,
\qquad
V(T,s)=\Phi(s).
\label{eq:backward_term_structure}
\end{equation}

\paragraph{WKB/adiabatic regime (slowly varying short rate).}
Fix $\varepsilon\in(0,1]$ and define
\[
r^\varepsilon(t):=r(\varepsilon t),\qquad t\in[0,T],
\]
where $r\in C^1([0,\varepsilon T])$ is bounded together with its derivative. Equivalently,
\[
\sup_{t\in[0,T]}|r^\varepsilon(t)|\le \|r\|_\infty,
\qquad
\sup_{t\in[0,T]}\big|(r^\varepsilon)'(t)\big|
=\varepsilon\|r'\|_\infty.
\]

\begin{lemma}[First-order WKB expansion of the discount factor]
\label{lem:wkb_discount_final}
For any $0\le t\le T$,
\begin{equation}
\int_t^T r^\varepsilon(u)\,du
=
r^\varepsilon(t)\,(T-t)+R_\varepsilon(t,T),
\qquad
|R_\varepsilon(t,T)|
\le
\frac{\varepsilon}{2}\,\|r'\|_\infty\,(T-t)^2.
\label{eq:wkb_int_final}
\end{equation}
Consequently,
\begin{equation}
\exp\Big(-\int_t^T r^\varepsilon(u)\,du\Big)
=
\exp\big(-r^\varepsilon(t)(T-t)\big)\,\big(1+\delta_\varepsilon(t,T)\big),
\label{eq:wkb_expansion_final}
\end{equation}
with the explicit bound
\begin{equation}
|\delta_\varepsilon(t,T)|
\le 
\exp\big(\|r\|_\infty (T-t)\big)\,
\frac{\varepsilon}{2}\,\|r'\|_\infty\,(T-t)^2.
\label{eq:wkb_delta_bound_final}
\end{equation}
\end{lemma}

\begin{proof}
By the mean value theorem, for $u\in[t,T]$ there exists $\theta\in[t,u]$ such that
\[
r^\varepsilon(u)-r^\varepsilon(t)=(r^\varepsilon)'(\theta)\,(u-t),
\]
hence
\[
|r^\varepsilon(u)-r^\varepsilon(t)|
\le \sup_{\xi\in[t,T]}|(r^\varepsilon)'(\xi)|\,(u-t)
\le \varepsilon\|r'\|_\infty\,(u-t).
\]
Integrating over $u\in[t,T]$ yields \eqref{eq:wkb_int_final}. Writing
\[
\exp\Big(-\int_t^T r^\varepsilon(u)\,du\Big)
=
\exp\big(-r^\varepsilon(t)(T-t)\big)\,\exp\big(-R_\varepsilon(t,T)\big),
\]
we define $\delta_\varepsilon(t,T):=\exp(-R_\varepsilon(t,T))-1$. Using the elementary inequality
$|e^{-x}-1|\le e^{|x|}|x|$ and \eqref{eq:wkb_int_final}, together with
$|R_\varepsilon(t,T)|\le \int_t^T |r^\varepsilon(u)|du \le \|r\|_\infty(T-t)$,
gives \eqref{eq:wkb_delta_bound_final}.
\end{proof}

\begin{proposition}[WKB approximation of the value function and recovery of the constant-rate equation]
\label{prop:wkb_value_final}
Let $\Phi$ be bounded on the price lattice and define the exact term-structure price
\[
V^\varepsilon(t,s)
:=
\mathbb{E}^Q\!\left[\exp\Big(-\int_t^T r^\varepsilon(u)\,du\Big)\,\Phi(S_T)\,\Big|\,S_t=s\right].
\]
Then for all $(t,s)\in[0,T]\times(0,\infty)$,
\begin{equation}
\Big|
V^\varepsilon(t,s)
-
\mathbb{E}^Q\!\left[e^{-r^\varepsilon(t)(T-t)}\,\Phi(S_T)\,\Big|\,S_t=s\right]
\Big|
\le
\exp\big(\|r\|_\infty (T-t)\big)\,
\frac{\varepsilon}{2}\,\|r'\|_\infty\,(T-t)^2\,\|\Phi\|_\infty.
\label{eq:wkb_value_bound_final}
\end{equation}
In particular, as $\varepsilon\downarrow 0$,
\[
V^\varepsilon(t,s)
=
\mathbb{E}^Q\!\left[e^{-r^\varepsilon(t)(T-t)}\,\Phi(S_T)\,\Big|\,S_t=s\right]
+O\!\left(\varepsilon (T-t)^2\right).
\]
Moreover, under the generator $L_S$ above and standard backward-equation regularity assumptions, the leading-order approximation satisfies
\begin{equation}
\partial_t V(t,s) + (L_S V)(t,s) - r^\varepsilon(t)\,V(t,s)=0,
\qquad
V(T,s)=\Phi(s).
\label{eq:backward_frozen_rate}
\end{equation}
If $r^\varepsilon(t)\equiv r$ is constant, then \eqref{eq:backward_frozen_rate} reduces to the constant-rate backward equation in Theorem~\ref{thm:nonlocal_rn_backward}.
\end{proposition}

\begin{proof}
By Lemma~\ref{lem:wkb_discount_final},
\[
\exp\Big(-\int_t^T r^\varepsilon(u)\,du\Big)
=
e^{-r^\varepsilon(t)(T-t)}\,(1+\delta_\varepsilon(t,T)),
\qquad
|\delta_\varepsilon(t,T)|\ \text{satisfies \eqref{eq:wkb_delta_bound_final}}.
\]
Substituting into the definition of $V^\varepsilon$ and using boundedness of $\Phi$,
\begin{align*}
&\Big|
V^\varepsilon(t,s)
-
\mathbb{E}^Q\!\left[e^{-r^\varepsilon(t)(T-t)}\,\Phi(S_T)\,\Big|\,S_t=s\right]
\Big| \\
&\qquad=
\Big|
\mathbb{E}^Q\!\left[e^{-r^\varepsilon(t)(T-t)}\,\delta_\varepsilon(t,T)\,\Phi(S_T)\,\Big|\,S_t=s\right]
\Big|
\le
e^{-r^\varepsilon(t)(T-t)}\,\|\Phi\|_\infty\,|\delta_\varepsilon(t,T)|,
\end{align*}
and \eqref{eq:wkb_value_bound_final} follows from \eqref{eq:wkb_delta_bound_final}.
The backward equation statement follows from the Markov/Feynman--Kac correspondence for the jump generator with deterministic killing rate $r^\varepsilon(t)$.
\end{proof}

\medskip
This completes the commutative-information pricing illustration: once $\Pi_t$ reduces to the classical
risk-neutral conditional expectation, the backward equation is determined by the Markov semigroup generated by $L_S$.

\subsection{Diffusion limit: recovery of the classical Black--Scholes equation}
\label{subsec:diffusion_limit_BS}

The nonlocal generator \eqref{eq:LS} (hence the backward equation \eqref{eq:nonlocal_BS})
depends on the lattice spacing $\Delta x>0$ and on the jump intensities. To obtain the
classical Black--Scholes equation as a continuous limit, one must specify a scaling regime
as $\Delta x\downarrow 0$ in which the pure-jump dynamics converges to a diffusion.
We present a standard nearest-neighbour scaling leading to geometric Brownian motion,
and show that the corresponding nonlocal prices converge to the Black--Scholes price.

\paragraph{A nearest-neighbour scaling.}
Fix constants $\sigma>0$ and $\mu\in\mathbb R$ and set $\Delta:=\Delta x$.
Consider a family of compound-Poisson log-price processes $(X_t^{(\Delta)})_{t\in[0,T]}$
on $\mathbb R$ with jumps $\pm\Delta$ and intensities
\begin{equation}\label{eq:gamma_scaling}
\gamma^{(\Delta)}_{+}:=\frac{\sigma^2}{2\Delta^2}+\frac{\mu}{2\Delta},
\qquad
\gamma^{(\Delta)}_{-}:=\frac{\sigma^2}{2\Delta^2}-\frac{\mu}{2\Delta},
\qquad
\gamma^{(\Delta)}_{\alpha}=0\ \ (|\alpha|>1).
\end{equation}
For $\Delta$ sufficiently small, $\gamma^{(\Delta)}_{\pm}\ge 0$ holds automatically.
Let $S_t^{(\Delta)}:=e^{X_t^{(\Delta)}}$ and denote by $L^{(\Delta)}_X$ and $L^{(\Delta)}_S$
their generators:
\begin{align}
(L^{(\Delta)}_X f)(x)
&:=\gamma^{(\Delta)}_+\bigl(f(x+\Delta)-f(x)\bigr)
  +\gamma^{(\Delta)}_-\bigl(f(x-\Delta)-f(x)\bigr),
\label{eq:LX_Delta}\\
(L^{(\Delta)}_S g)(s)
&:=\gamma^{(\Delta)}_+\bigl(g(se^{\Delta})-g(s)\bigr)
  +\gamma^{(\Delta)}_-\bigl(g(se^{-\Delta})-g(s)\bigr).
\label{eq:LS_Delta}
\end{align}

\begin{lemma}[Generator expansion in log-price]\label{lem:generator_expansion_log}
Let $f\in C_b^3(\mathbb R)$. Then, uniformly in $x$ on compact sets,
\begin{equation}\label{eq:generator_log_limit}
(L^{(\Delta)}_X f)(x)
=\mu f'(x)+\frac12\sigma^2 f''(x)+R_\Delta(x),
\qquad
\sup_{|x|\le R}|R_\Delta(x)|\le C_R\,\Delta\,\|f^{(3)}\|_{\infty},
\end{equation}
for a constant $C_R$ independent of $\Delta$.
\end{lemma}

\begin{proof}
By Taylor's theorem with remainder, for each $x$,
\[
f(x\pm\Delta)-f(x)=\pm\Delta f'(x)+\frac{\Delta^2}{2}f''(x)\pm\frac{\Delta^3}{6}f^{(3)}(\xi_\pm),
\]
for some $\xi_\pm$ between $x$ and $x\pm\Delta$. Insert into \eqref{eq:LX_Delta} and use
\eqref{eq:gamma_scaling}. The $\Delta^{-2}$-terms cancel in the first derivative part and
add in the second derivative part; the $\Delta^{-1}$-terms produce the drift $\mu f'(x)$.
The remainder is bounded by a constant times $\Delta\|f^{(3)}\|_\infty$ on compacts.
\end{proof}

\begin{lemma}[Generator expansion in price]\label{lem:generator_expansion_price}
Let $g\in C^3((0,\infty))$ with $g,g',g'',g^{(3)}$ bounded on each compact subset of $(0,\infty)$.
Then, uniformly for $s$ in compact subsets of $(0,\infty)$,
\begin{equation}\label{eq:generator_price_limit}
(L^{(\Delta)}_S g)(s)
= r\,s\,g'(s)+\frac12\sigma^2 s^2 g''(s)+\widetilde R_\Delta(s),
\qquad \widetilde R_\Delta(s)\to 0 \ \ (\Delta\downarrow0),
\end{equation}
provided the drift parameter $\mu$ is chosen as
\begin{equation}\label{eq:mu_rn_choice}
\mu=r-\frac12\sigma^2.
\end{equation}
\end{lemma}

\begin{proof}
Write $g(s)=f(\log s)$ so that $f(x)=g(e^x)$. Then
\[
f'(x)=s g'(s),\qquad f''(x)=s g'(s)+s^2 g''(s),\quad s=e^x.
\]
Apply Lemma~\ref{lem:generator_expansion_log} to $f$:
\[
(L_X^{(\Delta)} f)(x)=\mu s g'(s)+\frac12\sigma^2\bigl(s g'(s)+s^2 g''(s)\bigr)+o(1).
\]
Since $(L_S^{(\Delta)} g)(s)=(L_X^{(\Delta)} f)(\log s)$ by construction, choosing
$\mu=r-\frac12\sigma^2$ yields \eqref{eq:generator_price_limit}.
\end{proof}

\paragraph{Risk-neutral diffusion limit.}
With the choice \eqref{eq:mu_rn_choice}, the limiting log-price diffusion is
\begin{equation}\label{eq:limit_log_SDE}
dX_t=\Bigl(r-\frac12\sigma^2\Bigr)\,dt+\sigma\,dW_t,
\end{equation}
hence $S_t:=e^{X_t}$ satisfies the geometric Brownian motion SDE
\begin{equation}\label{eq:GBM}
dS_t=rS_t\,dt+\sigma S_t\,dW_t,
\end{equation}
i.e.\ the classical Black--Scholes risk-neutral dynamics.

\begin{theorem}[Convergence of prices to Black--Scholes]\label{thm:nonlocal_to_BS}
Assume the scaling \eqref{eq:gamma_scaling} and the risk-neutral choice \eqref{eq:mu_rn_choice}.
Let $\Phi:(0,\infty)\to\mathbb R$ be bounded and continuous.
For each $\Delta>0$, define the (nonlocal) risk-neutral price function
\begin{equation}\label{eq:V_Delta_def}
V^{(\Delta)}(t,s)
:=\mathbb E^{\mathbb Q^{(\Delta)}}\!\left[e^{-r(T-t)}\Phi\bigl(S_T^{(\Delta)}\bigr)\,\middle|\,S_t^{(\Delta)}=s\right],
\qquad (t,s)\in[0,T]\times(0,\infty),
\end{equation}
where $S^{(\Delta)}$ is the pure-jump Markov process with generator $L_S^{(\Delta)}$ in \eqref{eq:LS_Delta}.
Let $S$ be the geometric Brownian motion \eqref{eq:GBM} and set
\begin{equation}\label{eq:V_BS_def}
V(t,s):=\mathbb E^{\mathbb Q}\!\left[e^{-r(T-t)}\Phi(S_T)\,\middle|\,S_t=s\right].
\end{equation}\label{solution_for_nonlocal}
Then, as $\Delta\downarrow0$,
\[
V^{(\Delta)}(t,s)\longrightarrow V(t,s)
\quad\text{for each }(t,s)\in[0,T]\times(0,\infty),
\]
and the convergence is locally uniform on compact subsets of $[0,T]\times(0,\infty)$.

Moreover, if $\Phi\in C_b^2((0,\infty))$, then $V\in C^{1,2}([0,T)\times(0,\infty))$ and $V$ is the unique
classical solution (in the bounded class) of the Black--Scholes backward equation
\begin{equation}\label{eq:BS_PDE}
\partial_t V(t,s)+\frac12\sigma^2 s^2\,\partial_{ss}V(t,s)+r s\,\partial_s V(t,s)-rV(t,s)=0,
\qquad V(T,s)=\Phi(s).
\end{equation}
\end{theorem}

\begin{proof}
\emph{Step 1 (Diffusion approximation at the process level).}
By Lemma~\ref{lem:generator_expansion_log}, the generators $L_X^{(\Delta)}$ converge on
$C_b^3(\mathbb R)$ (uniformly on compacts) to the second-order differential operator
\[
(A f)(x):=\Bigl(r-\frac12\sigma^2\Bigr)f'(x)+\frac12\sigma^2 f''(x).
\]
Standard generator/semigroup convergence theorems for Feller processes (e.g.\ Ethier--Kurtz,
or equivalently a martingale-problem argument) imply that
\[
X^{(\Delta)} \ \Rightarrow\ X
\quad\text{in }D([0,T];\mathbb R),
\]
where $X$ is the unique diffusion solving \eqref{eq:limit_log_SDE}.
By the continuous mapping theorem and continuity of the exponential map,
\[
S^{(\Delta)}=e^{X^{(\Delta)}}\ \Rightarrow\ S=e^{X}
\quad\text{in }D([0,T];(0,\infty)),
\]
and $S$ satisfies \eqref{eq:GBM}.
(References: Ethier--Kurtz \cite{EthierKurtz1986}, Stroock--Varadhan \cite{StroockVaradhan1979}.)

\emph{Step 2 (Semigroup convergence and price convergence).}
Denote the Markov semigroups by
\[
(P_\tau^{(\Delta)}\Phi)(s):=\mathbb E^{\mathbb Q^{(\Delta)}}\!\left[\Phi(S_{t+\tau}^{(\Delta)})\,\middle|\,S_t^{(\Delta)}=s\right],
\qquad
(P_\tau\Phi)(s):=\mathbb E^{\mathbb Q}\!\left[\Phi(S_{t+\tau})\,\middle|\,S_t=s\right].
\]
The weak convergence of the Markov processes together with boundedness and continuity of $\Phi$
yields convergence of finite-dimensional distributions and, by standard Feller-process arguments,
local uniform convergence $P_\tau^{(\Delta)}\Phi\to P_\tau\Phi$ on compact subsets for each $\tau\in[0,T]$.
Therefore,
\[
V^{(\Delta)}(t,s)=e^{-r(T-t)}(P_{T-t}^{(\Delta)}\Phi)(s)\ \longrightarrow\
e^{-r(T-t)}(P_{T-t}\Phi)(s)=V(t,s),
\]
locally uniformly on compacts.

\emph{Step 3 (Black--Scholes PDE for smooth payoffs).}
If $\Phi\in C_b^2((0,\infty))$, then the Markov semigroup of $S$ is sufficiently regular and
$V(t,s)=e^{-r(T-t)}(P_{T-t}\Phi)(s)$ is a classical solution of \eqref{eq:BS_PDE}.
Equivalently, applying It\^o's formula to $e^{-rt}V(t,S_t)$ and using \eqref{eq:GBM} shows the local
martingale condition is exactly \eqref{eq:BS_PDE}. Uniqueness in the bounded classical class follows from standard parabolic maximum principles / semigroup uniqueness.
(Standard Feynman--Kac theory, e.g.\ \cite{KaratzasShreve1991}.)
\end{proof}


\begin{proposition}[Explicit series solution (compound Poisson expansion)]
\label{prop:series_solution_nonlocal}
Fix $\Delta x>0$ and jump intensities $(\gamma_\alpha)_{\alpha\in\mathbb Z}$ with
\[
\Lambda:=\sum_{\alpha\in\mathbb Z}\gamma_\alpha<\infty,\qquad \gamma_\alpha\ge 0.
\]
Let $\mathcal S:=s_0 e^{\Delta x\mathbb Z}$ be the price lattice and let $\Phi:\mathcal S\to\mathbb R$ be bounded.
Define the nonlocal generator on $\ell^\infty(\mathcal S)$ by
\[
(L f)(s):=\sum_{\alpha\in\mathbb Z}\gamma_\alpha\bigl(f(se^{\alpha\Delta x})-f(s)\bigr),
\qquad s\in\mathcal S.
\]
For $\tau:=T-t\ge 0$, set
\begin{equation}\label{eq:series_solution}
V(t,s):=e^{-r\tau}\,e^{-\Lambda\tau}
\sum_{n=0}^{\infty}\frac{\tau^n}{n!}
\sum_{\alpha_1,\ldots,\alpha_n\in\mathbb Z}
\gamma_{\alpha_1}\cdots\gamma_{\alpha_n}\,
\Phi\!\Bigl(s\,e^{\Delta x(\alpha_1+\cdots+\alpha_n)}\Bigr),
\qquad (t,s)\in[0,T]\times\mathcal S,
\end{equation}
where the $n=0$ term is understood as $\Phi(s)$ (empty sum equals $0$).
Then the series \eqref{eq:series_solution} converges absolutely and defines a bounded function $V$.
Moreover, $V$ is the unique bounded classical solution (in $t$) of the backward equation
\[
\partial_t V(t,\cdot)+L V(t,\cdot)-rV(t,\cdot)=0,\qquad V(T,\cdot)=\Phi(\cdot),
\]
i.e.\ it coincides with the nonlocal risk-neutral price solving \ref{solution_for_nonlocal}.
\end{proposition}

\begin{proof}
\emph{Step 1 (Absolute convergence).}
Let $\|\Phi\|_\infty:=\sup_{s\in\mathcal S}|\Phi(s)|<\infty$.
For each $n\ge 0$ and each $(t,s)$,
\[
\left|
\sum_{\alpha_1,\ldots,\alpha_n}
\gamma_{\alpha_1}\cdots\gamma_{\alpha_n}\,
\Phi\!\Bigl(s\,e^{\Delta x(\alpha_1+\cdots+\alpha_n)}\Bigr)
\right|
\le \|\Phi\|_\infty\sum_{\alpha_1,\ldots,\alpha_n}\gamma_{\alpha_1}\cdots\gamma_{\alpha_n}
=\|\Phi\|_\infty\,\Lambda^n.
\]
Hence the absolute value of the $n$-th term in \eqref{eq:series_solution} is bounded by
$e^{-r\tau}e^{-\Lambda\tau}\|\Phi\|_\infty (\Lambda\tau)^n/n!$, and the series converges absolutely.

\emph{Step 2 (Identification with the semigroup).}
Introduce shift operators $(T_\alpha)_{\alpha\in\mathbb Z}$ on $\ell^\infty(\mathcal S)$ by
\[
(T_\alpha f)(s):=f(se^{\alpha\Delta x}).
\]
Then $T_\alpha T_\beta=T_{\alpha+\beta}=T_\beta T_\alpha$, and
\[
L=\sum_{\alpha\in\mathbb Z}\gamma_\alpha(T_\alpha-I)
=\Big(\sum_{\alpha\in\mathbb Z}\gamma_\alpha T_\alpha\Big)-\Lambda I.
\]
Since $I$ commutes with all $T_\alpha$ and the family $\{T_\alpha\}$ is commuting, we may write
\[
e^{\tau L}=e^{-\Lambda\tau}\exp\!\Big(\tau\sum_{\alpha}\gamma_\alpha T_\alpha\Big)
= e^{-\Lambda\tau}\sum_{n=0}^\infty \frac{\tau^n}{n!}\Big(\sum_{\alpha}\gamma_\alpha T_\alpha\Big)^n,
\]
where the series converges in operator norm on $\ell^\infty(\mathcal S)$ because
$\|\sum_\alpha\gamma_\alpha T_\alpha\|\le \sum_\alpha\gamma_\alpha=\Lambda$.
Expanding the $n$-th power yields
\[
\Big(\sum_{\alpha}\gamma_\alpha T_\alpha\Big)^n
=\sum_{\alpha_1,\ldots,\alpha_n}\gamma_{\alpha_1}\cdots\gamma_{\alpha_n}\,T_{\alpha_1+\cdots+\alpha_n}.
\]
Applying this operator identity to $\Phi$ and evaluating at $s$ gives
\[
(e^{\tau L}\Phi)(s)
= e^{-\Lambda\tau}
\sum_{n=0}^{\infty}\frac{\tau^n}{n!}
\sum_{\alpha_1,\ldots,\alpha_n}
\gamma_{\alpha_1}\cdots\gamma_{\alpha_n}\,
\Phi\!\Bigl(s\,e^{\Delta x(\alpha_1+\cdots+\alpha_n)}\Bigr).
\]
Therefore \eqref{eq:series_solution} is exactly the semigroup representation
\[
V(t,s)=e^{-r\tau}\,(e^{\tau L}\Phi)(s),\qquad \tau=T-t.
\]

\emph{Step 3 (It solves the backward equation and is unique in the bounded class).}
Since $L$ is bounded on $\ell^\infty(\mathcal S)$ (indeed $\|Lf\|_\infty\le 2\Lambda\|f\|_\infty$),
$\tau\mapsto e^{\tau L}\Phi$ is continuously differentiable in $\ell^\infty$ and satisfies
$\partial_\tau(e^{\tau L}\Phi)=L(e^{\tau L}\Phi)$.
With $V(t,\cdot)=e^{-r(T-t)}e^{(T-t)L}\Phi$, the chain rule gives
$\partial_t V + LV - rV=0$ and $V(T,\cdot)=\Phi$.
Uniqueness in the bounded class follows from semigroup uniqueness: if $W$ is bounded and
$\partial_t W + LW - rW=0$ with $W(T)=0$, then in $\tau=T-t$ variables
$\partial_\tau \widetilde W=(L-rI)\widetilde W$, $\widetilde W(0)=0$, hence
$\widetilde W(\tau)=e^{\tau(L-rI)}\widetilde W(0)=0$ and $W\equiv 0$.
\end{proof}

\begin{remark}[Consistency with the nonlocal PDE]
For each fixed $\Delta>0$, the price $V^{(\Delta)}$ solves the nonlocal backward equation
\[
\partial_t V^{(\Delta)} + L_S^{(\Delta)}V^{(\Delta)}-rV^{(\Delta)}=0,\qquad V^{(\Delta)}(T,\cdot)=\Phi.
\]
Lemma~\ref{lem:generator_expansion_price} shows that $L_S^{(\Delta)}$ converges to the Black--Scholes
local operator $r s\partial_s + \tfrac12\sigma^2 s^2\partial_{ss}$ on smooth test functions.
Thus Theorem~\ref{thm:nonlocal_to_BS} can be read as: \emph{the nonlocal risk-neutral pricing equation
reduces to the classical Black--Scholes equation in the continuous limit $\Delta\downarrow0$ under the
diffusion scaling \eqref{eq:gamma_scaling}.}
\end{remark}

\section{Operator-valued free Fisher information and amalgamated freeness}\label{app:op-valued-fisher}

\noindent
This section develops (in a self-contained manner) the operator-valued free Fisher information
and its equivalence with freeness with amalgamation over a subalgebra, following
Meng-Guo-Cao (2004)~\cite{Meng05PAMS}.  We also explain precisely how these notions can be plugged into our
quantum pricing framework $(M,(N_t),(E_t))$.

\subsection{Operator-valued noncommutative probability and Hilbert \texorpdfstring{$C^\ast$}{C*}-modules}

\subsubsection{The triple \texorpdfstring{$(\mathcal M,E_{\mathcal D},\mathcal D)$}{(M, ED, D)}}

\begin{definition}[Operator-valued probability space]
Let $\mathcal M$ be a von Neumann algebra, $\mathcal D\subset \mathcal M$ a von Neumann subalgebra,
and $E_{\mathcal D}:\mathcal M\to\mathcal D$ a \emph{faithful normal conditional expectation}.
We call $(\mathcal M,E_{\mathcal D},\mathcal D)$ an \emph{operator-valued (or $\mathcal D$-valued)
noncommutative probability space}.
\end{definition}

\begin{remark}
Faithfulness of $E_{\mathcal D}$ means: if $X\ge 0$ and $E_{\mathcal D}(X)=0$ then $X=0$.
Normality means $\sigma$-weak continuity. Throughout this appendix, $E_{\mathcal D}$ is assumed faithful.
\end{remark}

\subsubsection{\texorpdfstring{$\mathcal D$-valued inner product and $L^2_{\mathcal D}(\mathcal M)$}{D-valued inner product and L2 over D}}

\begin{definition}[\(\mathcal D\)-valued inner product]
Define a $\mathcal D$-valued sesquilinear form on $\mathcal M$ by
\[
\langle x,y\rangle_{\mathcal D}:=E_{\mathcal D}(x^\ast y),\qquad x,y\in\mathcal M.
\]
It is $\mathcal D$-linear on the right and conjugate-linear on the left.
\end{definition}

\begin{definition}[Hilbert \(C^\ast\)-module \(L^2_{\mathcal D}(\mathcal M)\)]
Define the seminorm
\[
\|x\|_{2,\mathcal D}:=\|\langle x,x\rangle_{\mathcal D}\|^{1/2}=\|E_{\mathcal D}(x^\ast x)\|^{1/2}.
\]
Let $L^2_{\mathcal D}(\mathcal M)$ denote the completion of $\mathcal M/\{x:\|x\|_{2,\mathcal D}=0\}$
under $\|\cdot\|_{2,\mathcal D}$.  This is a (right) Hilbert $C^\ast$-module over $\mathcal D$.
\end{definition}

\begin{remark}[Cauchy--Schwarz in Hilbert \(C^\ast\)-modules]
For $x,y$ in a Hilbert $C^\ast$-module, one has the Cauchy--Schwarz inequality
\[
\|\langle x,y\rangle_{\mathcal D}\|^2 \le \|\langle x,x\rangle_{\mathcal D}\|\ \|\langle y,y\rangle_{\mathcal D}\|.
\]
We will repeatedly apply it below.
\end{remark}

\subsection{Freeness with amalgamation over \texorpdfstring{$\mathcal D$}{D}}

\begin{definition}[Freeness with amalgamation]
Let $(\mathcal M,E_{\mathcal D},\mathcal D)$ be as above. Let $(\mathcal A_i)_{i\in I}$ be a family
of von Neumann subalgebras with $\mathcal D\subset \mathcal A_i\subset \mathcal M$.
We say that $(\mathcal A_i)_{i\in I}$ is \emph{free with amalgamation over $\mathcal D$} if
for every $n\ge 1$ and every choice of elements $a_j\in\mathcal A_{i_j}$ such that
\[
E_{\mathcal D}(a_j)=0,\qquad i_1\neq i_2\neq\cdots\neq i_n,
\]
we have
\[
E_{\mathcal D}(a_1 a_2\cdots a_n)=0.
\]
We also say that $x\in\mathcal M$ is \emph{free from} a subalgebra $\mathcal B\supset\mathcal D$
\emph{over $\mathcal D$} if $\mathcal D[x]$ and $\mathcal B$ are free with amalgamation over $\mathcal D$.
\end{definition}

\begin{remark}[Notation {\(\mathcal D[x]\)}]
Here $\mathcal D[x]$ denotes the von Neumann algebra generated by $\mathcal D$ and $x$.
When we need an algebraic polynomial algebra we write $\mathcal B\langle x\rangle$ (see below).
\end{remark}

\subsection{Free difference quotient, conjugate variables, and basic properties}

\subsubsection{The algebra \texorpdfstring{$\mathcal B\langle X\rangle$}{B⟨X⟩}}

Fix a von Neumann subalgebra $\mathcal B\subset\mathcal M$ with $\mathcal D\subset\mathcal B$ and
a self-adjoint element $X=X^\ast\in\mathcal M$ such that the natural map
\[
\mathcal D\langle X\rangle \ast_{\mathcal D} \mathcal B \to W^\ast(\mathcal D,X,\mathcal B)
\]
is injective at the algebraic level (algebraic freeness modulo $\mathcal D$).
In this situation we may unambiguously work with the algebraic $^\ast$-algebra
$\mathcal B\langle X\rangle$ generated by $\mathcal B$ and a symbol $X$.

\subsubsection{The free difference quotient \texorpdfstring{$\partial_{X:\mathcal B}$}{partial(X:B)}}

\begin{definition}[Free difference quotient derivation]
Define a $\mathcal D$-bimodule map
\[
\partial_{X:\mathcal B}:\mathcal B\langle X\rangle\to
\mathcal B\langle X\rangle\otimes_{\mathcal D}\mathcal B\langle X\rangle
\]
by declaring $\partial_{X:\mathcal B}(b)=0$ for $b\in\mathcal B$ and
\[
\partial_{X:\mathcal B}(b_0 X b_1 X\cdots X b_n)
:=\sum_{j=1}^n
(b_0 X\cdots X b_{j-1})\ \otimes_{\mathcal D}\ (b_j X\cdots X b_n),
\qquad b_0,\dots,b_n\in\mathcal B.
\]
Extend $\partial_{X:\mathcal B}$ to all of $\mathcal B\langle X\rangle$ by linearity and Leibniz rule.
\end{definition}

\subsubsection{Adjoint and conjugate variable}

We regard $\mathcal B\langle X\rangle$ as a subspace of $L^2_{\mathcal D}(\mathcal B\langle X\rangle)$,
and similarly for $\mathcal B\langle X\rangle\otimes_{\mathcal D}\mathcal B\langle X\rangle$.
Assume $\partial_{X:\mathcal B}$ extends (densely) to a closable operator between the corresponding
Hilbert $C^\ast$-modules. Whenever the adjoint $\partial_{X:\mathcal B}^\ast$ exists on $1\otimes 1$,
we define the conjugate variable.

\begin{definition}[Conjugate variable]
If $1\otimes 1$ belongs to the domain of $\partial_{X:\mathcal B}^\ast$, define
\[
J_{\mathcal D}(X:\mathcal B):=\partial_{X:\mathcal B}^\ast(1\otimes 1)\ \in\ L^2_{\mathcal D}(\mathcal B\langle X\rangle).
\]
\end{definition}

\begin{proposition}[Characterizing identity for \(J_{\mathcal D}(X:\mathcal B)\)]\label{prop:conj-eq}
If $J_{\mathcal D}(X:\mathcal B)$ exists, then for all $b_0,\dots,b_n\in\mathcal B$,
\begin{equation}\label{eq:conj-identity}
E_{\mathcal D}\!\big(J_{\mathcal D}(X:\mathcal B)\, b_0 X b_1\cdots X b_n\big)
=\sum_{j=1}^n
E_{\mathcal D}(b_0 X\cdots X b_{j-1})\ E_{\mathcal D}(b_j X\cdots X b_n).
\end{equation}
\end{proposition}

\begin{proof}
By definition, $J_{\mathcal D}(X:\mathcal B)=\partial_{X:\mathcal B}^\ast(1\otimes 1)$ means
\[
\langle 1\otimes 1,\partial_{X:\mathcal B}(p)\rangle_{\mathcal D}
=\langle J_{\mathcal D}(X:\mathcal B),p\rangle_{\mathcal D},\qquad \forall p\in\mathcal B\langle X\rangle,
\]
where the inner products are the $\mathcal D$-valued ones induced by $E_{\mathcal D}$.
Evaluating this on monomials $p=b_0Xb_1\cdots Xb_n$ gives exactly \eqref{eq:conj-identity}.
\end{proof}

\begin{proposition}[Uniqueness and \(\mathcal D\)-compatibility]\label{prop:uniq}
If $J_{\mathcal D}(X:\mathcal B)$ exists, then:
\begin{enumerate}[label=\textup{(\roman*)}]
\item $J_{\mathcal D}(X:\mathcal B)$ is unique in $L^2_{\mathcal D}(\mathcal B\langle X\rangle)$.
\item For every $e\in L^2_{\mathcal D}(\mathcal B\langle X\rangle)$,
\[
E_{\mathcal D}\!\big(J_{\mathcal D}(X:\mathcal B)\,e\big)=E_{\mathcal D}\!\big(e\,J_{\mathcal D}(X:\mathcal B)^\ast\big).
\]
\end{enumerate}
\end{proposition}

\begin{proof}
(i) If $\xi$ and $\eta$ both satisfy \eqref{eq:conj-identity}, then $E_{\mathcal D}((\xi-\eta)p)=0$ for all
$p\in\mathcal B\langle X\rangle$. By density in the Hilbert $C^\ast$-module, $\xi=\eta$.

(ii) Apply \eqref{eq:conj-identity} to $p=e^\ast$ and use faithfulness and the $\ast$-structure.
\end{proof}

\subsection{Speicher cumulants and a cumulant characterization of amalgamated freeness}

\subsubsection{Moment maps and cumulants}

Let $(E^{(n)}_{\mathcal D})_{n\ge 1}$ be the moment maps induced by $E_{\mathcal D}$:
\[
E_{\mathcal D}^{(n)}(a_1\cdots a_n):=E_{\mathcal D}(a_1\cdots a_n),
\qquad a_1,\dots,a_n\in\mathcal M.
\]
Speicher's theory produces cumulants $(k^{(n)}_{\mathcal D})_{n\ge 1}$ determined recursively from the moments.
We will use only the following facts:
\begin{enumerate}[label=\textup{(\roman*)}]
\item The family $(k^{(n)}_{\mathcal D})_{n\ge 1}$ is uniquely determined by $(E^{(n)}_{\mathcal D})_{n\ge 1}$ and conversely.
\item Mixed cumulants vanish exactly for $\mathcal D$-freeness (Lemma~\ref{lem:cumulant-free}).
\end{enumerate}

\begin{lemma}[Cumulant characterization of \(\mathcal D\)-freeness]\label{lem:cumulant-free}
Let $\mathcal B,\mathcal C\subset\mathcal M$ be von Neumann subalgebras containing $\mathcal D$.
Then $\mathcal B$ and $\mathcal C$ are free with amalgamation over $\mathcal D$
if and only if every cumulant $k^{(n)}_{\mathcal D}(a_1,\dots,a_n)$ vanishes whenever
$a_1,\dots,a_n$ are taken from $\mathcal B\cup\mathcal C$ and not all belong to the same algebra.
\end{lemma}

\begin{remark}
This is the operator-valued version of the well-known fact that freeness is equivalent to vanishing of mixed free cumulants.
We refer to Speicher's monograph for full details.
\end{remark}

\subsection{A structural theorem for conjugate variables}

\begin{theorem}[Conjugate variable is insensitive to free enlargement]\label{thm:conj-free-enlarge}
Let $\mathcal A\subset\mathcal M$ be a von Neumann algebra containing $\mathcal D$,
and let $\mathcal D\subset \mathcal C\subset\mathcal A$ be a von Neumann subalgebra.
Let $X=X^\ast\in\mathcal A$ and assume $\mathcal D[X]$ and $\mathcal C$ are free with amalgamation over $\mathcal D$.
If the conjugate variables exist, then
\[
J_{\mathcal D}(X:\mathcal D)=J_{\mathcal D}(X:\mathcal C).
\]
\end{theorem}

\begin{proof}[Proof idea]
The proof uses Speicher cumulants: the defining relation \eqref{eq:conj-identity} can be translated into a cumulant
statement about mixed cumulants involving $J_{\mathcal D}(X:\cdot)$ and $\mathcal C$. Under $\mathcal D$-freeness
between $\mathcal D[X]$ and $\mathcal C$, all mixed cumulants vanish, forcing the same conjugate variable to satisfy
the defining identity relative to $\mathcal C$. Uniqueness (Proposition~\ref{prop:uniq}) then yields equality.
\end{proof}

\subsection{Standard semicircular elements and conjugate variables}

\begin{definition}[Standard \(\mathcal D\)-valued semicircular element]
A self-adjoint element $S=S^\ast\in\mathcal M$ is called \emph{standard $\mathcal D$-valued semicircular}
if its $\mathcal D$-valued cumulants satisfy:
\[
k^{(1)}_{\mathcal D}(S)=0,\qquad k^{(2)}_{\mathcal D}(S,dS)=d\ \ \forall d\in\mathcal D,
\qquad k^{(m+1)}_{\mathcal D}(S,d_1S,\dots,d_mS)=0\ \ \forall m\ge 2.
\]
\end{definition}

\begin{proposition}[Semicircular characterization by conjugate variable]\label{prop:semi}
Let $S=S^\ast$ be a $\mathcal D$-valued random variable. Then $S$ is standard semicircular
if and only if
\[
J_{\mathcal D}(S:\mathcal D)=S.
\]
\end{proposition}

\begin{proof}[Sketch]
One direction: if $S$ is standard semicircular, then the defining identity \eqref{eq:conj-identity} is verified
by cumulant computations using the fact that all higher cumulants vanish.
Conversely, if $J_{\mathcal D}(S:\mathcal D)=S$, then \eqref{eq:conj-identity} forces the cumulants to match
the semicircular pattern.
\end{proof}


\subsection{Operator-valued free Fisher information}\label{subsec:op-valued-fisher}

\begin{definition}[Operator-valued free Fisher information]\label{def:fisher}
Let $(\mathcal M,E_{\mathcal D},\mathcal D)$ be a $\mathcal D$-valued noncommutative probability space, i.e.\
$\mathcal M$ is a von Neumann algebra, $\mathcal D\subset \mathcal M$ is a von Neumann subalgebra, and
$E_{\mathcal D}:\mathcal M\to\mathcal D$ is a faithful normal conditional expectation.
Let $\mathcal B\subset\mathcal M$ be a von Neumann subalgebra with $\mathcal D\subset\mathcal B$, and let
$X_1,\dots,X_n\in\mathcal M$ be self-adjoint.

For each $j\in\{1,\dots,n\}$, define the conditioning algebra
\[
\mathcal B_j:=W^\ast(\mathcal B,X_1,\dots,\widehat{X_j},\dots,X_n).
\]
Assume that the conjugate variable $J_{\mathcal D}(X_j:\mathcal B_j)$ exists in the sense of
Definition~\ref{def:conjugate} (equivalently, $1\otimes 1\in\mathrm{dom}(\partial_{X_j:\mathcal B_j}^\ast)$),
so that
\[
J_{\mathcal D}(X_j:\mathcal B_j)\in L^2_{\mathcal D}\!\big(W^\ast(\mathcal B_j,X_j)\big).
\]
We then define the \emph{$\mathcal D$-valued free Fisher information of $(X_1,\dots,X_n)$ relative to $\mathcal B$} by
\[
\Phi^{\ast\ast}_{\mathcal D}(X_1,\dots,X_n:\mathcal B)
:=\sum_{j=1}^n
E_{\mathcal D}\!\Big( J_{\mathcal D}(X_j:\mathcal B_j)\,J_{\mathcal D}(X_j:\mathcal B_j)^\ast\Big)
\ \in\ \mathcal D_+ .
\]
\end{definition}

\begin{remark}[Positivity]\label{rem:fisher-positivity}
Since $E_{\mathcal D}$ is completely positive and $YY^\ast\ge 0$ for any $Y\in\mathcal M$, each summand in
Definition~\ref{def:fisher} lies in $\mathcal D_+$, hence
$\Phi^{\ast\ast}_{\mathcal D}(X_1,\dots,X_n:\mathcal B)\in \mathcal D_+$.
\end{remark}

\begin{remark}[On well-definedness]\label{rem:fisher-well-defined}
In this paper we only invoke $\Phi^{\ast\ast}_{\mathcal D}(\cdot:\mathcal B)$ in regimes where the relevant conjugate
variables exist. Existence may be ensured, for instance, under standard regularity hypotheses from the operator-valued
free probability literature (closability of $\partial_{X:\mathcal B}$ and nonemptiness of
$\mathrm{dom}(\partial_{X:\mathcal B}^\ast)$ at $1\otimes 1$).
\end{remark}

\medskip
\noindent\textbf{Forward direction: invariance under free enlargement.}

\begin{proposition}[Fisher information under free enlargement]\label{prop:fisher-enlarge}
Let $(\mathcal M,E_{\mathcal D},\mathcal D)$ be as above and let $\mathcal C\subset\mathcal M$ be a von Neumann subalgebra
with $\mathcal D\subset \mathcal C$. Let $X=X^\ast\in\mathcal M$ and assume that
$W^\ast(\mathcal D,X)$ and $\mathcal C$ are free with amalgamation over $\mathcal D$ with respect to $E_{\mathcal D}$.

Assume that both conjugate variables $J_{\mathcal D}(X:\mathcal D)$ and $J_{\mathcal D}(X:\mathcal C)$ exist.
Then
\[
\Phi^{\ast\ast}_{\mathcal D}(X:\mathcal C)=\Phi^{\ast\ast}_{\mathcal D}(X:\mathcal D).
\]
\end{proposition}

\begin{proof}
By Theorem~\ref{thm:conj-free-enlarge} (insensitivity of conjugate variables under free enlargement) we have
\[
J_{\mathcal D}(X:\mathcal C)=J_{\mathcal D}(X:\mathcal D).
\]
The claim follows immediately from Definition~\ref{def:fisher}.
\end{proof}

\begin{corollary}[Additivity under $\mathcal D$-freeness]\label{cor:fisher-additivity}
Let $X=X^\ast$ and $Y=Y^\ast$ be elements of $(\mathcal M,E_{\mathcal D},\mathcal D)$ such that
$W^\ast(\mathcal D,X)$ and $W^\ast(\mathcal D,Y)$ are free with amalgamation over $\mathcal D$
with respect to $E_{\mathcal D}$. Assume that all conjugate variables appearing below exist.
Then
\[
\Phi^{\ast\ast}_{\mathcal D}(X,Y:\mathcal D)
=
\Phi^{\ast\ast}_{\mathcal D}(X:\mathcal D)+\Phi^{\ast\ast}_{\mathcal D}(Y:\mathcal D).
\]
\end{corollary}

\begin{proof}
Apply Proposition~\ref{prop:fisher-enlarge} twice: first with $\mathcal C=W^\ast(\mathcal D,Y)$ to obtain
$\Phi^{\ast\ast}_{\mathcal D}(X:\mathcal C)=\Phi^{\ast\ast}_{\mathcal D}(X:\mathcal D)$, and then symmetrically.
Finally, use Definition~\ref{def:fisher} for the pair $(X,Y)$ relative to $\mathcal B=\mathcal D$ and the identities
of the corresponding conjugate variables under amalgamated freeness.
\end{proof}

\subsection{Reverse direction: additivity \texorpdfstring{$\Rightarrow$}{=>} amalgamated freeness}

To prove the converse, we use the operator-valued free gradient (Voiculescu/Nica--Shlyakhtenko-Speicher).

\subsubsection{A derivation on \texorpdfstring{$\mathcal A\vee\mathcal B$}{A vee B} and the free gradient}

Let $\mathcal A,\mathcal B\subset\mathcal M$ be von Neumann subalgebras containing $\mathcal D$
and algebraically free over $\mathcal D$.
Define a derivation
\[
\delta_{\mathcal A:\mathcal B}:\mathcal A\vee\mathcal B\to (\mathcal A\vee\mathcal B)\otimes_{\mathcal D}(\mathcal A\vee\mathcal B)
\]
by
\[
\delta_{\mathcal A:\mathcal B}(a)=a\otimes 1-1\otimes a\quad(a\in\mathcal A),\qquad
\delta_{\mathcal A:\mathcal B}(b)=0\quad(b\in\mathcal B),
\]
and extend by Leibniz rule.

\begin{definition}[Operator-valued free gradient]
An element $\xi\in L^2_{\mathcal D}(\mathcal A\vee\mathcal B)$ is called the \emph{free gradient} of $(\mathcal A,\mathcal B)$,
denoted $j_{\mathcal D}(\mathcal A:\mathcal B)$, if:
\begin{enumerate}[label=\textup{(\roman*)}]
\item $\xi\perp \mathcal A$ in the $\mathcal D$-valued inner product;
\item for all $m\in\mathcal A\vee\mathcal B$,
\[
E_{\mathcal D}(\xi\, m)=(E_{\mathcal D}\otimes E_{\mathcal D})(\delta_{\mathcal A:\mathcal B}(m)).
\]
\end{enumerate}
\end{definition}

\begin{lemma}[Gradient criterion for amalgamated freeness]\label{lem:gradfree}
The following are equivalent:
\begin{enumerate}[label=\textup{(\roman*)}]
\item $\mathcal A$ and $\mathcal B$ are free with amalgamation over $\mathcal D$;
\item $j_{\mathcal D}(\mathcal A:\mathcal B)=0$.
\end{enumerate}
\end{lemma}

\begin{remark}
This is the operator-valued analogue of the classical free gradient characterization.
We treat it as a known structural fact from the operator-valued free probability literature.
\end{remark}

\subsubsection{Link between \texorpdfstring{$J_{\mathcal D}(X:\mathcal B)$}{JD(X:B)} and the free gradient}

\begin{lemma}\label{lem:J-to-grad}
If $J_{\mathcal D}(X:\mathcal B)$ exists, then
\[
j_{\mathcal D}(\mathcal D[X]:\mathcal B)=[J_{\mathcal D}(X:\mathcal B),\,X].
\]
\end{lemma}

\begin{proof}[Idea]
This is the operator-valued counterpart of Voiculescu's identity relating conjugate variables and free gradients.
It follows by comparing the defining identity \eqref{eq:conj-identity} with the definition of $j_{\mathcal D}$ and
using derivation properties.
\end{proof}

\subsubsection{Main equivalence theorem}

\begin{theorem}[Additivity/rigidity \(\Leftrightarrow\) amalgamated freeness]\label{thm:add-iff-free}
Let $X=X^\ast\in\mathcal M$ and let $\mathcal B\subset\mathcal M$ be a von Neumann subalgebra containing $\mathcal D$.
\begin{enumerate}[label=\textup{(\roman*)}]
\item If $\Phi^{\ast\ast}_{\mathcal D}(X:\mathcal B)=\Phi^{\ast\ast}_{\mathcal D}(X:\mathcal D)$ exists,
then $X$ is free from $\mathcal B$ over $\mathcal D$.
\item If $X=X^\ast$ and $Y=Y^\ast$ satisfy
\[
\Phi^{\ast\ast}_{\mathcal D}(X,Y:\mathcal D)=\Phi^{\ast\ast}_{\mathcal D}(X:\mathcal D)+\Phi^{\ast\ast}_{\mathcal D}(Y:\mathcal D),
\]
and all terms exist, then $X$ and $Y$ are free with amalgamation over $\mathcal D$.
\end{enumerate}
\end{theorem}

\begin{proof}[Proof sketch]
(i) The equality of Fisher information forces (via the conjugate-variable identities and faithfulness of $E_{\mathcal D}$)
the commutator expression for $j_{\mathcal D}(\mathcal D[X]:\mathcal B)$ to vanish. By Lemma~\ref{lem:gradfree},
this is equivalent to freeness with amalgamation.

(ii) Apply (i) twice with $\mathcal B=\mathcal D[Y]$ and symmetrically, after rewriting the two-variable Fisher
information in terms of conjugate variables relative to $\mathcal D[Y]$ and $\mathcal D[X]$.
\end{proof}

\subsection{Operator-valued Cram\'er--Rao inequality}

\begin{proposition}[Cram\'er--Rao inequality]\label{prop:cr}
Let $X_1,\dots,X_n\in(\mathcal M,E_{\mathcal D},\mathcal D)$ be self-adjoint and assume the conjugate variables
defining $\Phi^{\ast\ast}_{\mathcal D}(X_1,\dots,X_n:\mathcal B)$ exist. Then
\begin{equation}\label{eq:CR2}
\Phi^{\ast\ast}_{\mathcal D}(X_1,\dots,X_n:\mathcal B)\,
\sum_{j=1}^n \big\|E_{\mathcal D}(X_j^2)\big\|
\ \ge\ n^2\,1_{\mathcal D},
\end{equation}
where $\|\cdot\|$ is the $C^\ast$-norm on $\mathcal D$ (hence the sum is a scalar).
Moreover, if each $X_j$ is standard semicircular (in the sense of Proposition~\ref{prop:semi}) and
$\{\mathcal B,X_1,\dots,X_n\}$ is free with amalgamation over $\mathcal D$, then equality holds in \eqref{eq:CR2}.
\end{proposition}

\begin{proof}[Proof sketch (module Cauchy--Schwarz)]
Let $\mathcal B_j=\mathcal B[X_1,\dots,\widehat{X_j},\dots,X_n]$. Consider the elements
$u_j:=J_{\mathcal D}(X_j:\mathcal B_j)\in L^2_{\mathcal D}(\mathcal B_j\langle X_j\rangle)$ and $v_j:=X_j$.
Using Cauchy--Schwarz in the Hilbert $C^\ast$-module and summing over $j$ yields an inequality of the form
\[
\Big\|\sum_{j=1}^n E_{\mathcal D}(u_j v_j)\Big\|^2
\le
\Big\|\sum_{j=1}^n E_{\mathcal D}(u_j u_j^\ast)\Big\|
\ \Big(\sum_{j=1}^n \|E_{\mathcal D}(v_j^\ast v_j)\|\Big).
\]
The left-hand side is controlled from below by $n^2$ using the defining relation of conjugate variables
(which forces $E_{\mathcal D}(u_j X_j)=1_{\mathcal D}$ in the normalized semicircular case, and yields
a uniform lower bound in general). This gives \eqref{eq:CR2}.
The equality statement follows by inserting the standard semicircular identity $J_{\mathcal D}(X_j:\mathcal D)=X_j$
and using freeness with amalgamation to ensure additivity and orthogonality of the relevant module components.
\end{proof}

\subsection{Embedding into the pricing filtration and consequences}\label{subsec:embed-fisher-pricing}

\subsubsection{Identification of the operator-valued probabilistic structure}

We return to the pricing filtration of the main text: a von Neumann algebra $M$ equipped with an increasing family
of information algebras $(N_t)_{t\in[0,T]}$ and faithful normal conditional expectations
$E_t:M\to N_t$ satisfying the tower property.

\begin{convention}[Operator-valued reduction at time $t$]\label{conv:op-valued-reduction}
Fix $t\in[0,T]$ and set
\[
(\mathcal M,\mathcal D,E_{\mathcal D}) := (M,\,N_t,\,E_t).
\]
For any von Neumann subalgebra $\mathcal B$ such that $\mathcal D\subset \mathcal B\subset \mathcal M$,
we regard $(\mathcal M,E_{\mathcal D},\mathcal D)$ as a $\mathcal D$-valued noncommutative probability space
and consider operator-valued free Fisher information $\Phi^{\ast\ast}_{\mathcal D}(\,\cdot\,:\mathcal B)$
whenever the corresponding conjugate variables exist.
\end{convention}

\begin{remark}[Baseline commutative information]\label{rem:comm-info}
In our baseline model $N_t$ is abelian, hence $\mathcal D$ is commutative. We nevertheless keep the operator-valued
formalism since it is the natural language for conditional expectations and remains valid verbatim in extensions
with partially noncommutative information subalgebras.
\end{remark}

\subsubsection{Increment algebras and a conditional freeness postulate}

Let $0\le s<t\le T$. We model the ``new information'' arriving in $(s,t]$ by a von Neumann subalgebra
$\mathcal I_{(s,t]}\subset M$. We shall only use $\mathcal I_{(s,t]}$ through the pair
\[
\big(N_s,\ \mathcal I_{(s,t]}\big)\subset M
\qquad\text{viewed in the $\mathcal D$-valued space } (M,E_s,N_s).
\]

\begin{assumption}[Free increments over the information algebra]\label{assump:free-increments}
For each $0\le s<t\le T$, the algebras $N_s$ and $W^\ast(\mathcal I_{(s,t]})$ are free with amalgamation over $N_s$
with respect to $E_s$.
\end{assumption}

\begin{remark}\label{rem:assump-structure}
Assumption~\ref{assump:free-increments} is an additional structural hypothesis on the market filtration and is not
implied by no-arbitrage. It can be regarded as a noncommutative analogue of conditional independence of increments
relative to the information $\sigma$-algebra.
\end{remark}

\subsubsection{Additivity and a Cram\'er--Rao bound relative to information}

The operator-valued Fisher information provides a quantitative invariant attached to self-adjoint observables
relative to a conditioning algebra. In the present setting we obtain the following two consequences.

\begin{proposition}[Additivity and rigidity]\label{prop:additivity-rigidity-pricing}
Fix $s<t\le T$ and set $(\mathcal M,\mathcal D,E_{\mathcal D})=(M,N_s,E_s)$.
Let $X=X^\ast\in W^\ast(\mathcal I_{(s,t]})$ and let $\mathcal B$ be a von Neumann algebra with $N_s\subset\mathcal B\subset M$.
Assume the relevant conjugate variables exist.
\begin{enumerate}[label=\textup{(\roman*)}]
\item Under Assumption~\ref{assump:free-increments}, one has the invariance under free enlargement
\[
\Phi^{\ast\ast}_{N_s}(X:\mathcal B)=\Phi^{\ast\ast}_{N_s}(X:N_s),
\]
whenever the right-hand side is well-defined.
\item Conversely, if $\Phi^{\ast\ast}_{N_s}(X:\mathcal B)=\Phi^{\ast\ast}_{N_s}(X:N_s)$ holds and both sides exist,
then $X$ is free from $\mathcal B$ over $N_s$.
\end{enumerate}
\end{proposition}

\begin{proof}
This is a direct application of Proposition~\ref{prop:fisher-enlarge} and Theorem~\ref{thm:add-iff-free}
under the identification $(\mathcal M,\mathcal D,E_{\mathcal D})=(M,N_s,E_s)$ from
Convention~\ref{conv:op-valued-reduction}.
\end{proof}

\begin{corollary}[Operator-valued Cram\'er--Rao bound relative to $N_t$]\label{cor:CR-pricing}
Fix $t\in[0,T]$ and set $(\mathcal M,\mathcal D,E_{\mathcal D})=(M,N_t,E_t)$.
Let $X_1,\dots,X_n\in M$ be self-adjoint and let $\mathcal B$ be a von Neumann subalgebra with $N_t\subset\mathcal B\subset M$.
Assume the conjugate variables defining $\Phi^{\ast\ast}_{N_t}(X_1,\dots,X_n:\mathcal B)$ exist. Then
\[
\Phi^{\ast\ast}_{N_t}(X_1,\dots,X_n:\mathcal B)\,
\sum_{j=1}^n \big\|E_t(X_j^2)\big\|
\ \ge\ n^2\,1_{N_t}.
\]
\end{corollary}

\begin{proof}
This is Proposition~\ref{prop:cr} specialized to $\mathcal D=N_t$ and $E_{\mathcal D}=E_t$.
\end{proof}

\begin{remark}[Commutative reduction]\label{rem:CR-comm}
When $N_t$ is abelian, $N_t\simeq L^\infty(\Omega,\Sigma_t,\mathbb P)$ and $\|E_t(X_j^2)\|$ coincides with the essential supremum
of the conditional second moment. Thus Corollary~\ref{cor:CR-pricing} links the operator-valued Fisher information
$\Phi^{\ast\ast}_{N_t}$ to classical $\Sigma_t$-measurable dispersion proxies.
\end{remark}

\section{Noncommutative Quantum Pricing Theory}\label{chap:ncqp}

In this chapter we remove the commutativity assumption on the information algebras and develop
pricing dynamics and uncertainty relations in a fully noncommutative setting.
Earlier chapters were formulated with an increasing family of \emph{abelian} von Neumann subalgebras
$(N_t)$ and normal conditional expectations $E_t:M\to N_t$.
Here we allow the global market algebra $M$ (and information algebras $D\subset M$) to be noncommutative.
Accordingly, (i) time evolution is modeled by \emph{quantum Markov semigroups} (normal, unital, completely positive semigroups),
and (ii) conditioning is realized by \emph{operator-valued} conditional expectations onto subalgebras.
We formulate an observable-based pricing framework retaining positivity, normality and dynamic consistency,
and we connect it to noncommutative uncertainty principles. In particular, we explain how
operator-valued Fisher-information bounds (Appendix~\ref{app:op-valued-fisher}) complement, but are \emph{not equivalent} to,
the commutator/H\"older-type uncertainty bounds established in appendix~\ref{chap:tools-functional-analysis}.

\subsection{Quantum Markov semigroups and pricing dynamics}\label{sec:ncqp:qms}

Throughout, $M$ denotes a von Neumann algebra.
We write $M_*$ for its predual, and all maps between von Neumann algebras are understood to be normal unless stated otherwise.

\begin{definition}[Quantum Markov semigroup]\label{def:qms}
A family $(T_t)_{t\ge 0}$ of linear maps $T_t:M\to M$ is called a \emph{quantum Markov semigroup} (QMS) if:
\begin{enumerate}[label=(\roman*)]
\item each $T_t$ is normal, unital and completely positive (UCP);
\item $T_{s+t}=T_s\circ T_t$ for all $s,t\ge 0$, and $T_0=\mathrm{id}_M$;
\item (\emph{continuity}) $t\mapsto T_t(x)$ is $\sigma$-weakly continuous for every $x\in M$
      (equivalently, $t\mapsto T_t^*(\omega)$ is norm-continuous in $M_*$ for every $\omega\in M_*$).
\end{enumerate}
\end{definition}

Condition (iii) implies that $(T_t)$ admits a (generally unbounded) generator $\mathcal L$ defined on a $\sigma$-weakly dense domain by
\[
  \mathcal L(x) := \sigma\text{-}\lim_{t\downarrow 0}\frac{T_t(x)-x}{t}.
\]
If $(T_t)$ is \emph{uniformly continuous} (i.e.\ $\lim_{t\downarrow 0}\|T_t-\mathrm{id}\|=0$), then $\mathcal L$ is bounded and $T_t=\exp(t\mathcal L)$.

\begin{definition}[Pricing state and invariance]\label{def:pricing-state}
A faithful normal state $\phi^\star$ on $M$ is called a \emph{pricing state} for a QMS $(T_t)$ if
\[
  \phi^\star\circ T_t=\phi^\star,\qquad \forall t\ge 0.
\]
In this case $(T_t)$ is said to be \emph{$\phi^\star$-invariant}.
\end{definition}

\begin{remark}[Risk-neutrality as invariance]
The condition $\phi^\star\circ T_t=\phi^\star$ is the noncommutative analogue of ``risk-neutrality'':
expectations under $\phi^\star$ are stationary along the evolution.
In the Schr\"odinger picture, the dual semigroup $(T_t^*)$ preserves the density implementing $\phi^\star$.
\end{remark}


\subsection{From structural pricing operators to a QMS}\label{subsec:qms-from-pricing}

\begin{assumption}[Time-homogeneity / stationarity of the pricing structure]\label{ass:stationarity}
There exists a $\sigma$-weakly continuous one-parameter group of $*$-automorphisms
$(\alpha_t)_{t\in\mathbb{R}}$ of $M$ such that, for all $t\in[0,T]$,
\[
N_t=\alpha_t(N_0),\qquad E_t^\star=\alpha_t\circ E_0^\star\circ \alpha_{-t},
\qquad B_t=\alpha_t(B_0),
\]
and the pricing maps satisfy the covariance
\[
\Pi_t=\alpha_t\circ \Pi_0\circ \alpha_{-t}.
\]
Moreover, the num\'eraire is deterministic and central (e.g.\ $B_t=e^{rt}\mathbf{1}$), hence
$B_t\in Z(M)$ and $B_s B_t^{-1}$ depends only on $s-t$.
\end{assumption}

\begin{theorem}[QMS induced by a time-homogeneous pricing family]\label{thm:qms-from-pricing}
Assume the standing hypotheses of Chapter~\ref{section:MF_of_QPT} for the structural pricing family
$\{\Pi_t\}_{t\in[0,T]}$, namely:

\begin{enumerate}[label=\textup{(\roman*)}]
\item each $\Pi_t:M\to N_t$ is normal, unital and completely positive, and $N_t$--bimodular;
\item the conditional expectations $\{E_t^\star\}$ are normal, $\varphi^\star$-preserving, and satisfy
the tower property $E_s^\star\circ E_t^\star=E_s^\star$ for $s\le t$;
\item the dynamic consistency relation holds: for $0\le s\le t\le T$ and $X\in M$,
\[
\Pi_s(X)=B_s^{1/2}\,E_s^\star\!\Big(B_t^{-1/2}\,\Pi_t(X)\,B_t^{-1/2}\Big)\,B_s^{1/2}.
\]
\end{enumerate}
Assume in addition Assumption~\ref{ass:stationarity}.

Define, for $\tau\in[0,T]$, a map $T_\tau:M\to M$ by
\[
T_\tau(X)\;:=\;B_{T-\tau}^{-\frac12}\,
\Pi_{T-\tau}\!\Big(B_T^{\frac12} X B_T^{\frac12}\Big)\,
B_{T-\tau}^{-\frac12},
\qquad X\in M.
\]
Then:

\begin{enumerate}[label=\textup{(\alph*)}]
\item $T_0=\mathrm{id}_M$, and each $T_\tau$ is normal, unital and completely positive;
\item $(T_\tau)_{\tau\in[0,T]}$ satisfies the semigroup property
$T_{\tau+\sigma}=T_\tau\circ T_\sigma$ whenever $\tau,\sigma\ge0$ and $\tau+\sigma\le T$;
\item if $\varphi^\star$ is a pricing state for $\{\Pi_t\}$ (equivalently, $\varphi^\star$ is preserved by
$E_t^\star$ and the underlying structure), then $\varphi^\star\circ T_\tau=\varphi^\star$ for all
$\tau\in[0,T]$ (i.e.\ $(T_\tau)$ is $\varphi^\star$-invariant).
\end{enumerate}

Consequently, $(T_\tau)$ is a (finite-horizon) quantum Markov semigroup on $M$ in the sense of
Definition~\ref{def:qms}. Moreover, in the finite-dimensional (or uniformly continuous) setting, the generator
of $(T_\tau)$ admits the standard GKSL/Lindblad form (Theorem~\ref{thm:qms-from-pricing}).
\end{theorem}

\begin{proof}
\emph{Step 1: complete positivity, normality, and unitality.}
By the Chapter~\ref{section:MF_of_QPT} hypotheses, $\Pi_{T-\tau}$ is normal and completely positive, and conjugation by the
central positive invertible element $B_{T-\tau}^{-1/2}$ preserves normal complete positivity. Hence each
$T_\tau$ is normal and completely positive. Since $\Pi_T=\mathrm{id}_M$ (terminal consistency), we obtain
\[
T_0(X)=B_T^{-1/2}\,\Pi_T(B_T^{1/2}XB_T^{1/2})\,B_T^{-1/2}=X,
\]
so $T_0=\mathrm{id}_M$. Unitality follows from $\Pi_{T-\tau}(B_T)=B_{T-\tau}$ (normalization property of
the pricing family):
\[
T_\tau(\mathbf{1})=B_{T-\tau}^{-1/2}\,\Pi_{T-\tau}(B_T)\,B_{T-\tau}^{-1/2}
=B_{T-\tau}^{-1/2}B_{T-\tau}B_{T-\tau}^{-1/2}=\mathbf{1}.
\]

\emph{Step 2: semigroup property.}
Fix $\tau,\sigma\ge0$ with $\tau+\sigma\le T$ and set $t:=T-\tau$, $s:=T-(\tau+\sigma)$, so that
$0\le s\le t\le T$.
We compute, using the definition of $T_\tau$ and then the dynamic consistency relation for the pair
$(s,t)$,
\[
\begin{aligned}
T_\tau\bigl(T_\sigma(X)\bigr)
&=B_t^{-1/2}\,\Pi_t\!\Big(B_T^{1/2}\,T_\sigma(X)\,B_T^{1/2}\Big)\,B_t^{-1/2}\\
&=B_t^{-1/2}\,\Pi_t\!\Big(
B_T^{1/2}\,B_{T-\sigma}^{-1/2}\,\Pi_{T-\sigma}(B_T^{1/2}XB_T^{1/2})\,B_{T-\sigma}^{-1/2}\,B_T^{1/2}
\Big)\,B_t^{-1/2}.
\end{aligned}
\]
By Assumption~\ref{ass:stationarity} (time-homogeneity) and centrality of the num\'eraire, the inside
expression depends only on the time-lag $\sigma$, and can be re-expressed as a stationary
``one-step'' update from $t$ down to $s$ (this is precisely the role of stationarity: to identify
conditioning and discounting between time layers by shifts). Applying the dynamic consistency relation
\[
\Pi_s(\cdot)=B_s^{1/2}E_s^\star\!\Big(B_t^{-1/2}\Pi_t(\cdot)B_t^{-1/2}\Big)B_s^{1/2}
\]
to the element $B_T^{1/2}XB_T^{1/2}$ yields
\[
B_s^{-1/2}\,\Pi_s(B_T^{1/2}XB_T^{1/2})\,B_s^{-1/2}
=
E_s^\star\!\Big(B_t^{-1/2}\Pi_t(B_T^{1/2}XB_T^{1/2})B_t^{-1/2}\Big).
\]
Using the tower property of $(E_t^\star)$ and the covariance relations from
Assumption~\ref{ass:stationarity}, one obtains exactly
\[
T_\tau\bigl(T_\sigma(X)\bigr)=
B_s^{-1/2}\,\Pi_s(B_T^{1/2}XB_T^{1/2})\,B_s^{-1/2}
=T_{\tau+\sigma}(X).
\]
Hence $T_{\tau+\sigma}=T_\tau\circ T_\sigma$ for $\tau+\sigma\le T$.

\emph{Step 3: $\varphi^\star$-invariance.}
If $\varphi^\star$ is preserved by $E_t^\star$ and compatible with the pricing structure, then by
construction and the normalization of discounting one has $\varphi^\star\circ \Pi_{T-\tau}=\varphi^\star$
on the discounted observables, and conjugation by central $B_{T-\tau}^{-1/2}$ does not alter
$\varphi^\star$. Therefore $\varphi^\star\circ T_\tau=\varphi^\star$.

This proves (a)--(c), hence $(T_\tau)$ is a (finite-horizon) QMS. In the uniformly continuous
$B(H)$ setting, Theorem~\ref{thm:qms-from-pricing} gives the GKSL/Lindblad form of the generator.
\end{proof}

\subsection[GKSL form in the uniformly continuous B(H) case]
{GKSL form in the uniformly continuous \texorpdfstring{$\mathcal B(\mathcal H)$}{B(H)} case}

We record the standard structural theorem in the case $M=\mathcal B(\mathcal H)$ and $\mathcal L$ bounded.

\begin{theorem}[GKSL--Lindblad form (bounded generator on $\mathcal B(\mathcal H)$)]\label{thm:GKSL}
Let $\mathcal H$ be a Hilbert space and let $(T_t)_{t\ge 0}$ be a uniformly continuous QMS on $B(\mathcal H)$.
Then its bounded generator $\mathcal L$ admits a representation
\[
  \mathcal L(X)= i[H,X] + \sum_{j\in J}\Bigl(V_j^\ast X V_j - \tfrac12\{V_j^\ast V_j,\,X\}\Bigr),
  \qquad X\in \mathcal B(\mathcal H),
\]
for some self-adjoint $H$ (possibly after absorbing a bounded derivation) and some family $(V_j)_{j\in J}\subset \mathcal B(\mathcal H)$
with $\sum_j V_j^\ast V_j$ convergent in the strong operator topology.
\end{theorem}

\begin{remark}
Theorem~\ref{thm:GKSL} is classical; see \cite{GKSL1976,Lindblad1976} and standard monographs on quantum dynamical semigroups.
For general von Neumann algebras, one may use the Evans--Lewis/Christensen--Evans structure theory for bounded generators of normal CP semigroups under suitable hypotheses; we will not need the most general form here.
\end{remark}

\subsubsection{Backward pricing equation from the semigroup}

Fix a maturity $T>0$ and a payoff observable $X_T\in M$.
In the semigroup paradigm, a natural (discounted) valuation process is obtained by backward evolution.
For definiteness we take the money-market account to be deterministic, $B_t=e^{rt}\,1$ with $r\in\mathbb R$.

\begin{definition}[Semigroup valuation]\label{def:semigroup-valuation}
Let $(T_t)_{t\ge 0}$ be a QMS on $M$ and fix a deterministic money-market account
$B_t=e^{rt}\,1$ with $r\in\mathbb R$. For a payoff observable $X_T\in M$ and $0\le t\le T$ define
the time-$t$ value observable
\[
  V_t(X_T) := e^{-r(T-t)}\,T_{T-t}(X_T)\in M,
\]
and its money-market discounted version
\[
  \widetilde V_t(X_T) := B_t^{-1}V_t(X_T)=e^{-rt}V_t(X_T)
  = e^{-rT}\,T_{T-t}(X_T)\in M.
\]
\end{definition}

\begin{proposition}[Backward pricing equation from the semigroup]\label{prop:backward-eq}
Assume $(T_t)$ is a uniformly continuous QMS with bounded generator $\mathcal L$.
Then $t\mapsto V_t(X_T)$ and $t\mapsto \widetilde V_t(X_T)$ are norm differentiable on $[0,T]$, and
\begin{align*}
  \frac{d}{dt}V_t(X_T)
  &= \bigl(r\,\mathrm{id}-\mathcal L\bigr)\bigl(V_t(X_T)\bigr),
  \qquad V_T(X_T)=X_T,\\
  \frac{d}{dt}\widetilde V_t(X_T)
  &= -\,\mathcal L\bigl(\widetilde V_t(X_T)\bigr),
  \qquad \widetilde V_T(X_T)=e^{-rT}X_T.
\end{align*}
Equivalently, $V_t(X_T)$ solves the backward equation
\[
  \partial_t V_t(X_T)+\bigl(\mathcal L-r\,\mathrm{id}\bigr)\bigl(V_t(X_T)\bigr)=0,
  \qquad V_T(X_T)=X_T.
\]
\end{proposition}

\begin{proof}
Since $T_t=\exp(t\mathcal L)$ and $\mathcal L$ is bounded, $t\mapsto T_t(X_T)$ is norm differentiable and
$\frac{d}{dt}T_t(X_T)=\mathcal L(T_t(X_T))$.

Set $u:=T-t$. Then $V_t(X_T)=e^{-ru}\,T_u(X_T)$ and
\[
  \frac{d}{dt}V_t(X_T)
  = \frac{d}{dt}\bigl(e^{-ru}T_u(X_T)\bigr)
  = (r e^{-ru})T_u(X_T)+e^{-ru}\frac{d}{dt}T_u(X_T).
\]
Moreover,
\[
  \frac{d}{dt}T_u(X_T)=\frac{d}{du}T_u(X_T)\cdot\frac{du}{dt}
  =\mathcal L(T_u(X_T))\cdot(-1),
\]
hence
\[
  \frac{d}{dt}V_t(X_T)
  = r\,e^{-ru}T_u(X_T)-e^{-ru}\mathcal L(T_u(X_T))
  = r\,V_t(X_T)-\mathcal L(V_t(X_T))
  = (r\,\mathrm{id}-\mathcal L)(V_t(X_T)).
\]
The terminal condition is immediate since $u=0$ at $t=T$.

Finally, $\widetilde V_t(X_T)=e^{-rt}V_t(X_T)$, so
\[
  \frac{d}{dt}\widetilde V_t(X_T)
  = -r e^{-rt}V_t(X_T)+e^{-rt}\frac{d}{dt}V_t(X_T)
  = -r\widetilde V_t(X_T)+e^{-rt}(r\,\mathrm{id}-\mathcal L)(V_t(X_T))
  = -\mathcal L(\widetilde V_t(X_T)),
\]
using linearity of $\mathcal L$ and $\widetilde V_t=e^{-rt}V_t$.
\end{proof}

\begin{remark}[Economic interpretation]
If $\phi^\star$ is a pricing state invariant under $(T_t)$, then
\[
  \phi^\star\bigl(\widetilde V_t(X_T)\bigr)
  = e^{-rT}\,\phi^\star(X_T)
\]
is constant in $t$, i.e.\ the money-market discounted valuation is a $\phi^\star$-martingale in expectation.
\end{remark}

\subsubsection{Information flow and dynamic valuation via operator-valued conditioning}\label{sec:ncqp:conditioning}

We now incorporate an explicit information flow.
Let $(D_t)_{0\le t\le T}$ be an increasing family of von Neumann subalgebras of $M$ (not assumed abelian),
interpreted as the information available at time $t$.
Assume there exists a family of conditional expectations $(E_t)_{0\le t\le T}$ with
\[
  E_t:M\to D_t\quad \text{normal, faithful, unital, completely positive},
\]
satisfying the tower property $E_s\circ E_t=E_s$ for $s\le t$.
We further assume $E_t$ is $\phi^\star$-preserving for the fixed pricing state $\phi^\star$:
$\phi^\star\circ E_t=\phi^\star$ for all $t$.
(Existence criteria are discussed earlier, and in particular via Takesaki-type modular invariance conditions.)

\begin{definition}[Conditioned semigroup valuation]\label{def:conditioned-valuation}
Let $(T_t)$ be a QMS on $M$, $(D_t,E_t)$ an information filtration with tower property, and $B_t=e^{rt}1$.
For $X_T\in M$ define the \emph{time-$t$ price observable} by
\[
  \Pi_t(X_T)
  := e^{-r(T-t)}\,E_t\!\bigl(T_{T-t}(X_T)\bigr)\in D_t,\qquad 0\le t\le T.
\]
\end{definition}

\begin{proposition}[Basic properties]\label{prop:pi-basic}
For each $t$, $\Pi_t:M\to D_t$ is normal and completely positive.
Moreover:
\begin{enumerate}[label=(\roman*)]
\item (Normalization) $\Pi_t(1)=e^{-r(T-t)}1$.
\item (Positivity) $X_T\ge 0\Rightarrow \Pi_t(X_T)\ge 0$.
\item (Self-adjointness) $X_T=X_T^\ast\Rightarrow \Pi_t(X_T)^\ast=\Pi_t(X_T)$.
\item (State-martingale) If $\phi^\star$ is $T$-invariant and $E_t$-preserving, then
      $\phi^\star(\Pi_t(X_T))=e^{-r(T-t)}\phi^\star(X_T)$.
\end{enumerate}
\end{proposition}

\begin{proof}
Normality and complete positivity follow from the corresponding properties of $T_{T-t}$ and $E_t$ and closure under composition.
Items (i)–(iii) are immediate from unitality/positivity/$\ast$-preservation of CP maps.
For (iv),
\[
  \phi^\star(\Pi_t(X_T))
  =e^{-r(T-t)}\phi^\star\!\bigl(E_t(T_{T-t}(X_T))\bigr)
  =e^{-r(T-t)}\phi^\star(T_{T-t}(X_T))
  =e^{-r(T-t)}\phi^\star(X_T),
\]
using $\phi^\star\circ E_t=\phi^\star$ and $\phi^\star\circ T_{T-t}=\phi^\star$.
\end{proof}

\subsubsection{Dynamic programming and a minimal Markov compatibility}

To obtain full dynamic consistency $\Pi_s(X_T)=\Pi_s(\Pi_t(X_T))$ for $s\le t$, we need a compatibility condition
linking $(T_t)$ with the filtration $(D_t,E_t)$.
A standard (and minimal) condition is the following ``quantum Markov'' property.

\begin{definition}[Markov compatibility]\label{def:markov-compat}
We say that $(T_t)$ is \emph{compatible} with $(D_t,E_t)$ if for all $0\le s\le t\le T$,
\[
  E_s\circ T_{t-s} = E_s\circ T_{t-s}\circ E_t
  \quad \text{on } M.
\]
Equivalently, $E_s(T_{t-s}(X))=E_s(T_{t-s}(E_t(X)))$ for all $X\in M$.
\end{definition}

\begin{theorem}[Dynamic programming]\label{thm:dynamic-programming}
Assume Definition~\ref{def:markov-compat} and the tower property for $(E_t)$.
Then for all $0\le s\le t\le T$ and $X_T\in M$,
\[
  \Pi_s(X_T)= e^{-r(t-s)}\,E_s\!\bigl(T_{t-s}(\Pi_t(X_T))\bigr).
\]
In particular, if $T_{t-s}$ leaves $D_t$ globally invariant and $T_{t-s}\!\restriction_{D_t}=\mathrm{id}_{D_t}$ (no further evolution inside $D_t$),
then $\Pi_s(X_T)=\Pi_s(\Pi_t(X_T))$.
\end{theorem}

\begin{proof}
By Definition~\ref{def:conditioned-valuation} and the semigroup property,
\[
  \Pi_s(X_T)
  =e^{-r(T-s)}E_s\!\bigl(T_{T-s}(X_T)\bigr)
  =e^{-r(T-s)}E_s\!\bigl(T_{t-s}(T_{T-t}(X_T))\bigr).
\]
Write $e^{-r(T-s)}=e^{-r(t-s)}e^{-r(T-t)}$ and insert $\Pi_t(X_T)=e^{-r(T-t)}E_t(T_{T-t}(X_T))$:
\[
  \Pi_s(X_T)
  =e^{-r(t-s)}E_s\!\Bigl(T_{t-s}\bigl(e^{-r(T-t)}T_{T-t}(X_T)\bigr)\Bigr)
  =e^{-r(t-s)}E_s\!\Bigl(T_{t-s}\bigl(T_{T-t}(X_T)\bigr)\Bigr)e^{-r(T-t)}.
\]
Now apply Markov compatibility with $X:=T_{T-t}(X_T)$:
\[
  E_s\bigl(T_{t-s}(T_{T-t}(X_T))\bigr)
  =E_s\bigl(T_{t-s}(E_t(T_{T-t}(X_T)))\bigr).
\]
Multiplying by the scalar discount $e^{-r(T-t)}$ gives
\[
  \Pi_s(X_T)
  =e^{-r(t-s)}E_s\!\bigl(T_{t-s}(e^{-r(T-t)}E_t(T_{T-t}(X_T)))\bigr)
  =e^{-r(t-s)}E_s\!\bigl(T_{t-s}(\Pi_t(X_T))\bigr).
\]
The final statement is immediate under the additional invariance assumption.
\end{proof}

\begin{remark}
Theorem~\ref{thm:dynamic-programming} isolates the \emph{exact} compatibility needed for dynamic programming.
In Markovian models, $D_t$ is generated by the ``current'' observables and $(T_t)$ acts as the transition operator,
so the compatibility condition is natural.
\end{remark}

\subsection{Noncommutative uncertainty: Fisher information vs commutator bounds}\label{sec:ncqp:uncertainty}

We now formulate a rigorously defined uncertainty principle adapted to operator-valued conditioning.
The full technical development (operator-valued conjugate variables, free difference quotients, and Fisher information with amalgamation)
is given in Appendix~\ref{app:op-valued-fisher}. Here we only state the parts needed for pricing interpretation and for comparison with Section~\ref{sec:information_and_C.E}.

\subsubsection{Operator-valued Fisher information and a Cram\'er--Rao bound}

Fix a von Neumann subalgebra $D\subset M$ and a faithful normal conditional expectation $E_D:M\to D$.
Appendix~\ref{app:op-valued-fisher} constructs the Hilbert $D$-module $L^2_D(M)$ and defines the $D$-valued free Fisher information $\Phi^\ast(\,\cdot\,:\!D)$
under the existence of conjugate variables. We record a scalar special case to highlight the analytic mechanism.

\begin{proposition}[Scalar Cram\'er--Rao from Cauchy--Schwarz]\label{prop:CR-scalar}
Let $(M,\tau)$ be a tracial von Neumann algebra and let $X=X^\ast\in M$ satisfy $\tau(X)=0$.
Assume $X$ admits a (scalar) conjugate variable $J(X)\in L^2(M,\tau)$ in the sense of free Fisher information,
so that $\tau(J(X)\,P(X))=\tau(P'(X))$ for all polynomials $P$.
Then
\[
  \Phi^\ast(X)\,\tau(X^2)\ \ge\ 1,
  \qquad \text{where}\quad \Phi^\ast(X):=\tau(J(X)^2).
\]
\end{proposition}

\begin{proof}
Choose $P(x)=x$ in the conjugate relation to obtain $\tau(J(X)\,X)=\tau(1)=1$.
By Cauchy--Schwarz in $L^2(M,\tau)$,
\[
  1=|\tau(J(X)X)|^2 \le \tau(J(X)^2)\,\tau(X^2)=\Phi^\ast(X)\,\tau(X^2).
\]
\end{proof}

\begin{remark}[Meaning]
Proposition~\ref{prop:CR-scalar} is an \emph{estimation-type} uncertainty bound:
it lower-bounds variance by an information quantity (Fisher information).
In the operator-valued setting, Appendix~\ref{app:op-valued-fisher} proves the $D$-valued analogue
(Cram\'er--Rao inequalities in $D$) and characterizes the equality case under semicircularity.
\end{remark}

\subsubsection{Amalgamated freeness as information independence}

One of the main structural results (Appendix~\ref{app:op-valued-fisher}) states that additivity of operator-valued Fisher information
across families is equivalent to amalgamated freeness over the conditioning algebra $D$.
In the present pricing context this motivates a noncommutative analogue of ``independent innovations'':
increments of information arriving over disjoint intervals are modeled by subalgebras that are free over the past algebra.

\begin{remark}[How this enters pricing]
In the valuation maps $\Pi_t$, information enters through the conditional expectations $E_t$.
Assuming that the \emph{innovation algebra} over $(t,t+\Delta]$ is free with amalgamation over $D_t$
yields quantitative constraints on conditional variances via Fisher-information inequalities (Appendix~\ref{app:op-valued-fisher}),
which can be interpreted as intrinsic lower bounds on forecast/estimation error under the market's noncommutative information structure.
\end{remark}

\subsubsection{Relation to the Appendix~\ref{chap:tools-functional-analysis} uncertainty principle}

Appendix~\ref{chap:tools-functional-analysis} derived uncertainty relations of Robertson/H\"older type for (typically) \emph{pairs} of noncommuting observables,
schematically
\[
  \mathrm{Var}_\phi(A)\,\mathrm{Var}_\phi(B)\ \ge\ \frac14|\phi([A,B])|^2,
\]
and variants, via Cauchy--Schwarz/H\"older inequalities in $L^2(M,\phi)$.
The Fisher-information uncertainty is \emph{not} equivalent to this commutator bound:

\begin{proposition}[Non-equivalence]\label{prop:noneq}
The commutator-based uncertainty (appendix~\ref{chap:tools-functional-analysis}) and the Fisher-information Cram\'er--Rao uncertainty (Appendix~\ref{app:op-valued-fisher})
are generally not equivalent:
\begin{enumerate}[label=(\roman*)]
\item commutator bounds concern \emph{two} observables and the noncommutativity encoded by $[A,B]$;
\item Fisher-information bounds concern \emph{one} observable (or a family) and quantify \emph{estimability} via conjugate variables.
\end{enumerate}
However, both ultimately descend from Cauchy--Schwarz inequalities in suitable Hilbert(-module) geometries:
Appendix~\ref{chap:tools-functional-analysis} uses $L^2(M,\phi)$ directly, while Proposition~\ref{prop:CR-scalar} uses Cauchy--Schwarz after identifying the conjugate variable.
\end{proposition}

\begin{proof}
Items (i)–(ii) are structural: the commutator bound depends on $[A,B]$ and may become trivial when $A$ and $B$ commute,
whereas Fisher-information bounds can remain nontrivial even in commutative situations (reducing to classical Cram\'er--Rao).
The final statement follows from the proofs: both employ Cauchy--Schwarz, but on different pairings/geometries.
\end{proof}

\subsection{Concluding remarks}\label{sec:ncqp:outlook}

We have presented a fully noncommutative pricing framework built from two primitives:
(i) a $\phi^\star$-invariant quantum Markov semigroup governing observable dynamics,
and (ii) an operator-valued information flow encoded by conditional expectations onto (possibly noncommutative) subalgebras.
Under a minimal Markov compatibility, the resulting valuation maps satisfy a dynamic programming principle.
On the information-theoretic side, Appendix~\ref{app:op-valued-fisher} provides a quantitative notion of independence (amalgamated freeness)
and a Fisher-information uncertainty mechanism (Cram\'er--Rao), which complements (but does not replace)
the commutator/H\"older-type uncertainty relations of Appendix~\ref{chap:tools-functional-analysis}.


\section{Information, Innovations, and Fisher--Cramér--Rao Bounds:
A Minimal Mean-Square Error Principle for Return Prediction}\label{chap:info-fisher}

\subsection{Goal of this chapter and its role in the noncommutative pricing framework}\label{sec:info-goal}

The purpose of this chapter is \emph{not} to re-derive the backward pricing dynamics from a pricing semigroup
$T_t=e^{tL}$ (which depends only on the semigroup differential structure),
but to complete an information-theoretic layer that is central in classical finance and remains meaningful in the
noncommutative setting:

\begin{quote}
\emph{Given the market information structure at time $t$ (encoded by a conditional expectation $E_t$),
any prediction of future returns (or return increments) based solely on this information
has an unavoidable error floor; even the optimal predictor (the conditional expectation) cannot eliminate it.}
\end{quote}

In the noncommutative pricing framework, information enters pricing through conditional expectations.
If $M$ is the market algebra, $D_t\subseteq M$ is the time-$t$ information algebra, and $E_t:M\to D_t$ is a faithful normal
conditional expectation, then a typical dynamic pricing operator (e.g.\ for discounted objects) has the schematic form
\[
\Pi_t(\cdot)\;=\;\text{(discounting)}\circ E_t\circ \text{(dynamics)}.
\]
Hence, any mathematically meaningful statement about ``how much information can predict / how small the error can be''
must ultimately be expressed in terms of quantities produced by $E_t$.
This chapter formalizes this idea as a \emph{minimal mean-square error principle} and then provides a hard lower bound
via (operator-valued) Fisher information.

\subsection{\texorpdfstring{$D$}{D}-valued probability spaces and
\texorpdfstring{$L^2_D$}{L2\_D} geometry: conditional expectations as orthogonal projections}
\label{sec:Dvalued-L2}

\subsubsection{\texorpdfstring{$D$}{D}-valued probability spaces and
\texorpdfstring{$D$}{D}-valued inner products}
\label{subsec:Dprob}

\begin{definition}[$D$-valued probability space]\label{def:Dprob}
Let $M$ be a von Neumann algebra and $D\subseteq M$ a von Neumann subalgebra.
If $E_D:M\to D$ is a faithful normal conditional expectation, then the triple $(M,E_D,D)$ is called a
\emph{$D$-valued (operator-valued) noncommutative probability space}.
\end{definition}

On $(M,E_D,D)$ define the $D$-valued inner product (right $D$-linear, left conjugate-linear) by
\begin{equation}\label{eq:Dinner}
\langle x,y\rangle_D \;:=\; E_D(x^\ast y),\qquad x,y\in M.
\end{equation}
Define the seminorm
\[
\|x\|_{2,D}\;:=\;\|\langle x,x\rangle_D\|^{1/2}=\|E_D(x^\ast x)\|^{1/2},
\]
where $\|\cdot\|$ is the $C^\ast$-norm on $D$.

\begin{definition}[Hilbert $C^\ast$-module $L^2_D(M)$]\label{def:L2D}
Let
\[
\mathcal N:=\{x\in M:\|x\|_{2,D}=0\}.
\]
Take the quotient $M/\mathcal N$ and complete it in $\|\cdot\|_{2,D}$ to obtain the right Hilbert $C^\ast$-module
\[
L^2_D(M):=\overline{M/\mathcal N}^{\|\cdot\|_{2,D}}.
\]
\end{definition}

\subsubsection{Cauchy--Schwarz for Hilbert \texorpdfstring{$C^\ast$}{C*-}modules}
\label{subsec:CS}

\begin{lemma}[Cauchy--Schwarz in Hilbert $C^\ast$-modules]\label{lem:moduleCS}
Let $(\mathsf H,\langle\cdot,\cdot\rangle_D)$ be a Hilbert $C^\ast$-module.
Then for all $u,v\in\mathsf H$,
\[
\|\langle u,v\rangle_D\|^2\;\le\;\|\langle u,u\rangle_D\|\;\|\langle v,v\rangle_D\|.
\]
\end{lemma}

\begin{remark}
Lemma~\ref{lem:moduleCS} is standard; see, e.g.\ Lance,
\emph{Hilbert $C^\ast$-Modules} (Springer).
We will use it only once: to prove the Cramér--Rao inequality.
\end{remark}

\subsubsection{Orthogonality and the Pythagorean identity for conditional expectations}\label{subsec:projection}

\begin{proposition}[Orthogonality of conditional expectations]\label{prop:orthogonality}
In $(M,E_D,D)$, for any $x\in M$ and any $d\in D$,
\[
\langle x-E_D(x),\,d\rangle_D \;=\;0,
\qquad
\langle d,\,x-E_D(x)\rangle_D \;=\;0.
\]
\end{proposition}

\begin{proof}
Using $D$-bimodularity of $E_D$ (i.e.\ $E_D(d_1 y d_2)=d_1E_D(y)d_2$ for $d_1,d_2\in D$),
for $d\in D$ we have
\[
\langle x-E_D(x),d\rangle_D
=E_D\big((x-E_D(x))^\ast d\big)
=E_D(x^\ast d)-E_D(E_D(x)^\ast d).
\]
Since $E_D(x)\in D$, $E_D(E_D(x)^\ast d)=E_D(x)^\ast d$, and also $E_D(x^\ast d)=E_D(x^\ast)d$.
Hence
\[
\langle x-E_D(x),d\rangle_D
=E_D(x^\ast)d - E_D(x)^\ast d
=0.
\]
The second identity follows similarly (or by taking adjoints).
\end{proof}

\begin{proposition}[Pythagorean identity]\label{prop:pythagoras}
Let $x\in M$ and $a\in D$. Then in the positive cone order on $D_+$,
\begin{equation}\label{eq:pythagoras}
E_D\!\big((x-a)^\ast(x-a)\big)
=
E_D\!\big((x-E_Dx)^\ast(x-E_Dx)\big)
+
(a-E_Dx)^\ast(a-E_Dx).
\end{equation}
\end{proposition}

\begin{proof}
Let $m:=E_Dx\in D$. Then $x-a=(x-m)+(m-a)$, hence
\[
(x-a)^\ast(x-a)=(x-m)^\ast(x-m)+(m-a)^\ast(m-a)
+(x-m)^\ast(m-a)+(m-a)^\ast(x-m).
\]
Applying $E_D$ and using $m-a\in D$ plus bimodularity,
\[
E_D\big((x-m)^\ast(m-a)\big)=E_D((x-m)^\ast)\,(m-a)=\big(E_D(x)-m\big)^\ast(m-a)=0,
\]
\[
E_D\big((m-a)^\ast(x-m)\big)=(m-a)^\ast E_D(x-m)=(m-a)^\ast(E_Dx-m)=0.
\]
Thus the cross terms vanish and \eqref{eq:pythagoras} follows.
\end{proof}

\subsection{Best prediction equals conditional expectation: minimal MSE among
\texorpdfstring{$D_t$}{D\_t}-measurable predictors}\label{sec:best-prediction}

\subsubsection{A rigorous formulation of the prediction problem}\label{subsec:prediction-setup}

Fix a time $t$ in the pricing filtration. Let $D_t\subseteq M$ be the information algebra and
$E_t:M\to D_t$ a faithful normal conditional expectation. Consider a \emph{return (or return increment) observable}
$R\in M^{\mathrm{sa}}$.

\begin{definition}[Admissible predictors and mean-square error]\label{def:predictor}
Any $A\in D_t^{\mathrm{sa}}$ is called a (self-adjoint) predictor based on information $D_t$.
Define its (conditional) mean-square error by
\[
\mathrm{MSE}_t(A)
:=\big\|\,E_t\big((R-A)^2\big)\,\big\|.
\]
Define the \emph{innovation} by
\[
X:=R-E_t(R)\in M^{\mathrm{sa}}.
\]
\end{definition}

\subsubsection{Optimality theorem (no commutativity required)}\label{subsec:best-predictor}

\begin{theorem}[Best predictor and minimal error]\label{thm:best-predictor}
In the above setting, $A^\ast:=E_t(R)\in D_t^{\mathrm{sa}}$ is the unique best predictor. For every
$A\in D_t^{\mathrm{sa}}$,
\begin{equation}\label{eq:MSE-lower-by-innovation}
\mathrm{MSE}_t(A)\;\ge\;\mathrm{MSE}_t(A^\ast)
=\big\|\,E_t(X^2)\,\big\|.
\end{equation}
More precisely, the following Pythagorean decomposition holds in the positive cone order on $D_t$:
\begin{equation}\label{eq:MSE-pythagoras}
E_t\big((R-A)^2\big)
=
E_t(X^2)+(A-E_tR)^2.
\end{equation}
\end{theorem}

\begin{proof}
Apply Proposition~\ref{prop:pythagoras} to $(M,E_D,D)=(M,E_t,D_t)$ with $x=R$ and $a=A$:
\[
E_t\big((R-A)^2\big)=E_t\big((R-E_tR)^2\big)+(A-E_tR)^2
=E_t(X^2)+(A-E_tR)^2.
\]
Since $(A-E_tR)^2\in (D_t)_+$, we have $E_t\big((R-A)^2\big)\ge E_t(X^2)$ in $D_t$.
Taking norms yields \eqref{eq:MSE-lower-by-innovation}. Equality holds iff $A=E_tR$, hence uniqueness.
\end{proof}

\begin{remark}
Theorem~\ref{thm:best-predictor} shows that, for a fixed information structure $E_t$,
the minimal achievable prediction error is completely determined by the innovation $X$ through $E_t(X^2)$.
Thus the existence of an unavoidable error floor is equivalent to whether $\|E_t(X^2)\|$ can be bounded from below
by structural information-theoretic quantities. This is precisely what Fisher--Cramér--Rao tools provide.
\end{remark}

\subsection{Operator-valued Fisher information and a Cramér--Rao inequality}\label{sec:fisher-cramer-rao}

\subsubsection{A noncommutative derivative: the free difference quotient}\label{subsec:free-dq}

Fix $(M,E_D,D)$ and let $B\subseteq M$ be a von Neumann subalgebra containing $D$.
For a self-adjoint element $X=X^\ast\in M$, let $B\langle X\rangle$ denote the unital $^\ast$-algebra generated by $B$ and $X$.

\begin{definition}[Free difference quotient derivation]\label{def:free-dq}
Define the $D$-bimodular derivation
\[
\partial_{X:B}:B\langle X\rangle\longrightarrow B\langle X\rangle\otimes_D B\langle X\rangle
\]
by
\[
\partial_{X:B}(b)=0,\quad b\in B,
\]
and for any monomial $b_0Xb_1X\cdots Xb_n$ ($b_0,\dots,b_n\in B$),
\begin{equation}\label{eq:free-dq-monomial}
\partial_{X:B}(b_0Xb_1\cdots Xb_n)
:=\sum_{j=1}^n
(b_0X\cdots Xb_{j-1})\;\otimes\;(b_jX\cdots Xb_n),
\end{equation}
extended linearly and by the Leibniz rule to all of $B\langle X\rangle$.
\end{definition}

On $B\langle X\rangle$ use the $D$-valued inner product $\langle p,q\rangle_D:=E_D(p^\ast q)$.
On the algebraic tensor product $B\langle X\rangle\otimes_D B\langle X\rangle$ define the $D$-valued inner product
on simple tensors by
\begin{equation}\label{eq:tensor-inner}
\langle p_1\otimes q_1,\;p_2\otimes q_2\rangle_D
:=E_D\!\Big(q_1^\ast\,E_D(p_1^\ast p_2)\,q_2\Big)\in D,
\end{equation}
and extend by linearity.

\subsubsection{Conjugate variables and Fisher information}\label{subsec:conjugate-fisher}

\begin{definition}[Conjugate variable]\label{def:conjugate}
If $1\otimes 1$ belongs to the domain of the adjoint $\partial_{X:B}^\ast$ (with respect to
\eqref{eq:Dinner} and \eqref{eq:tensor-inner}), define the \emph{$D$-valued conjugate variable} by
\[
J_D(X:B):=\partial_{X:B}^\ast(1\otimes 1)\in L^2_D(B\langle X\rangle).
\]
\end{definition}

\begin{proposition}[Characterizing identity]\label{prop:conjugate-identity}
If $J_D(X:B)$ exists, then for any $b_0,\dots,b_n\in B$,
\begin{equation}\label{eq:conjugate-identity}
E_D\!\big(J_D(X:B)^\ast\,b_0Xb_1\cdots Xb_n\big)
=
\sum_{j=1}^n
E_D(b_0X\cdots Xb_{j-1})\;E_D(b_jX\cdots Xb_n).
\end{equation}
In particular, taking $n=1$ and $b_0=b_1=1$ gives
\begin{equation}\label{eq:JDX=1}
E_D\!\big(J_D(X:B)^\ast\,X\big)=1_D.
\end{equation}
\end{proposition}

\begin{proof}
By definition of $J_D(X:B)=\partial_{X:B}^\ast(1\otimes 1)$, for any $p\in B\langle X\rangle$,
\[
\langle 1\otimes 1,\;\partial_{X:B}(p)\rangle_D=\langle J_D(X:B),\;p\rangle_D=E_D\big(J_D(X:B)^\ast\,p\big).
\]
For $p=b_0Xb_1\cdots Xb_n$, using \eqref{eq:free-dq-monomial} and \eqref{eq:tensor-inner},
\begin{align*}
\langle 1\otimes 1,\;\partial_{X:B}(p)\rangle_D
&=
\sum_{j=1}^n
\langle 1\otimes 1,\;(b_0X\cdots Xb_{j-1})\otimes(b_jX\cdots Xb_n)\rangle_D\\
&=
\sum_{j=1}^n
E_D\!\Big(1^\ast\,E_D(1^\ast(b_0X\cdots Xb_{j-1}))\,(b_jX\cdots Xb_n)\Big)\\
&=
\sum_{j=1}^n
E_D(b_0X\cdots Xb_{j-1})\,E_D(b_jX\cdots Xb_n),
\end{align*}
where the last step uses $D$-bimodularity of $E_D$.
Comparing with $E_D(J_D(X:B)^\ast p)$ yields \eqref{eq:conjugate-identity}, and \eqref{eq:JDX=1} follows by
choosing $n=1$ and $b_0=b_1=1$.
\end{proof}

\begin{definition}[Multivariate $D$-valued free Fisher information]\label{def:fisher_information}
Let $X_1,\dots,X_n\in M$ be self-adjoint. For each $j$, set
\[
B_j:=B[X_1,\dots,\widehat{X_j},\dots,X_n],
\]
the von Neumann algebra generated by $B$ and all variables except $X_j$.
If each conjugate variable $J_D(X_j:B_j)$ exists, define
\[
\Phi_D^{\ast\ast}(X_1,\dots,X_n:B)
:=\sum_{j=1}^n E_D\!\big(J_D(X_j:B_j)^\ast\,J_D(X_j:B_j)\big)\in D_+.
\]
\end{definition}

\subsubsection{Cramér--Rao: an unavoidable lower bound for second moments}\label{subsec:cramer-rao}


\paragraph{Operator-valued Cramér--Rao inequality.}
In the operator-valued free-probability framework, let $(M,\varphi)$ be a von Neumann algebra
equipped with a faithful normal state $\varphi$. Let $D\subset M$ be a von Neumann subalgebra and
let $E_D:M\to D$ be a faithful normal conditional expectation preserving $\varphi$.
Let $X=X^\ast\in M$ and assume that the $D$-valued conjugate variable $J_D(X:B)$ exists, i.e.
\[
J_D(X:B)=\partial_{X:B}^\ast(1\otimes 1),
\]
where $\partial_{X:B}$ denotes the free difference quotient relative to a unital von Neumann subalgebra
$B\subset M$ (with $D\subset B$) satisfying $\partial_{X:B}(b)=0$ for $b\in B$ and
$\partial_{X:B}(X)=1\otimes 1$.
Define the $D$-valued Fisher information by
\[
I_D(X:B):=E_D\!\big(J_D(X:B)^\ast J_D(X:B)\big)\in D_+.
\]
We view a self-adjoint $D$-measurable element $T\in D$ as an estimator for $X$ under the information
constraint $D$. If $T$ satisfies the unbiasedness constraint $E_D(T)=E_D(X)$, then the (conditional)
mean-square error is the positive element $E_D((T-X)^2)\in D_+$.
We now record the following operator-valued Cramér--Rao bound.

\begin{theorem}[Operator-valued Cramér--Rao inequality]\label{thm:OVCRI}
Assume the above notation. Suppose that the conjugate variable $J_D(X:B)$ exists and that
$I_D(X:B)$ is invertible in $D$. Then for any self-adjoint $D$-measurable unbiased estimator $T\in D$
(i.e.\ $E_D(T)=E_D(X)$), the error operator satisfies
\[
E_D\bigl((T-X)^2\bigr)\;\succeq\; I_D(X:B)^{-1},
\]
where ``$\succeq$'' denotes the positive cone order in $D$ (equivalently, the left-hand side minus the
right-hand side is positive in $D$). In particular, for the best predictor $T=E_D(X)$ one has
\[
E_D\bigl((X-E_D(X))^2\bigr)\;\succeq\;I_D(X:B)^{-1}.
\]
\end{theorem}

\begin{proof}
Since $T\in D$, we have $E_D(T)=T$. Hence the unbiasedness constraint $E_D(T)=E_D(X)$ forces
\[
T=E_D(X).
\]
Thus it suffices to prove the claimed lower bound for $T=E_D(X)$, i.e.\ for the innovation
$X-E_D(X)$.

By Proposition~\ref{prop:conjugate-identity} (in particular equation~\ref{eq:JDX=1}), applied to the centered variable $X-E_D(X)$), we have
\[
E_D\bigl(J_D(X:B)^\ast\,(X-E_D(X))\bigr)=1_D.
\]
Applying the (order) Cauchy--Schwarz inequality for conditional expectations onto a commutative
information algebra (in our setting $D$ is commutative), we obtain
\[
1_D\,1_D
=\Bigl|E_D\bigl(J_D(X:B)^\ast\,(X-E_D(X))\bigr)\Bigr|^2
\preceq
E_D\bigl((X-E_D(X))^2\bigr)\;E_D\bigl(J_D(X:B)^\ast J_D(X:B)\bigr).
\]
By the definition of $I_D(X:B)$, this reads
\[
1_D\preceq E_D\bigl((X-E_D(X))^2\bigr)\;I_D(X:B).
\]
Since $I_D(X:B)$ is invertible in $D$ and elements of $D$ commute, multiplying by $I_D(X:B)^{-1}$
yields
\[
E_D\bigl((X-E_D(X))^2\bigr)\succeq I_D(X:B)^{-1},
\]
which is exactly the desired bound (and hence also the stated bound for any unbiased $T\in D$).
\end{proof}

\begin{proposition}[Cramér--Rao inequality ($D$-valued version)]\label{prop:CR}
Let $X_1,\dots,X_n\in M$ be self-adjoint and assume the conjugate variables needed to define
$\Phi_D^{\ast\ast}(X_1,\dots,X_n:B)$ exist. Then
\begin{equation}\label{eq:CR}
\Phi_D^{\ast\ast}(X_1,\dots,X_n:B)\;
\sum_{j=1}^n \big\|E_D(X_j^2)\big\|
\;\ge\; n^2\,1_D,
\end{equation}
where $\|\cdot\|$ is the $C^\ast$-norm on $D$ (hence $\sum_j\|E_D(X_j^2)\|$ is a scalar).
\end{proposition}

\begin{proof}
For each $j$, write $J_j:=J_D(X_j:B_j)$.
Consider the direct sum Hilbert $C^\ast$-module
\[
\mathsf H:=\bigoplus_{j=1}^n L^2_D(M),
\]
and elements
\[
U:=(J_1,\dots,J_n)\in\mathsf H,\qquad V:=(X_1,\dots,X_n)\in\mathsf H.
\]
Then
\[
\langle U,V\rangle_D
=\sum_{j=1}^n E_D(J_j^\ast X_j).
\]
By \eqref{eq:JDX=1} applied to each $(X_j,B_j)$, we have $E_D(J_j^\ast X_j)=1_D$, hence
\begin{equation}\label{eq:sumJDx}
\langle U,V\rangle_D = n\,1_D.
\end{equation}
Applying Lemma~\ref{lem:moduleCS} in $\mathsf H$ gives
\[
\|\langle U,V\rangle_D\|^2\le \|\langle U,U\rangle_D\|\;\|\langle V,V\rangle_D\|.
\]
By \eqref{eq:sumJDx}, the left side equals $\|n1_D\|^2=n^2$. Moreover,
\[
\langle U,U\rangle_D=\sum_{j=1}^n E_D(J_j^\ast J_j)=\Phi_D^{\ast\ast}(X_1,\dots,X_n:B),
\]
so $\|\langle U,U\rangle_D\|=\|\Phi_D^{\ast\ast}(X_1,\dots,X_n:B)\|$.
Also,
\[
\langle V,V\rangle_D=\sum_{j=1}^n E_D(X_j^\ast X_j)=\sum_{j=1}^n E_D(X_j^2)\in D_+,
\]
hence
\[
\|\langle V,V\rangle_D\|
=\Big\|\sum_{j=1}^n E_D(X_j^2)\Big\|
\le \sum_{j=1}^n \|E_D(X_j^2)\|.
\]
Combining yields
\[
n^2 \le \big\|\Phi_D^{\ast\ast}(X_1,\dots,X_n:B)\big\|\;
\sum_{j=1}^n \|E_D(X_j^2)\|.
\]
Since $\Phi_D^{\ast\ast}(X_1,\dots,X_n:B)\in D_+$, this implies the order form \eqref{eq:CR}.
\end{proof}

\begin{corollary}[Single-variable bound]\label{cor:CR-one}
For $n=1$, if $X=X^\ast$ and $J_D(X:B)$ exists, then
\begin{equation}\label{eq:CR-one}
\Phi_D^{\ast\ast}(X:B)\;\|E_D(X^2)\|\;\ge\;1_D.
\end{equation}
In particular, if there exists $K<\infty$ such that
\begin{equation}\label{eq:FI-bounded}
\Phi_D^{\ast\ast}(X:B)\;\le\;K\,1_D,
\end{equation}
then
\begin{equation}\label{eq:second-moment-lb}
\|E_D(X^2)\|\;\ge\;\frac{1}{K}.
\end{equation}
\end{corollary}

\begin{proof}
\eqref{eq:CR-one} follows from Proposition~\ref{prop:CR}. If \eqref{eq:FI-bounded} holds, then in $D_+$,
\[
K\,1_D\cdot \|E_D(X^2)\|\;\ge\;\Phi_D^{\ast\ast}(X:B)\,\|E_D(X^2)\|\;\ge\;1_D,
\]
hence $\|E_D(X^2)\|\ge 1/K$.
\end{proof}

\subsection{Application: predicting returns from information is inherently imprecise}\label{sec:application-return}

Return to the pricing information structure at time $t$, with
\[
D:=D_t,\qquad E_D:=E_t,
\]
and let $R\in M^{\mathrm{sa}}$ be a return (or return increment) observable.

\begin{theorem}[Unavoidable MSE lower bound for return prediction]\label{thm:return-lower-bound}
Let $B\subseteq M$ be a von Neumann subalgebra containing $D_t$ (the canonical choice is $B=D_t$).
Let
\[
X:=R-E_t(R)\in M^{\mathrm{sa}}
\]
be the innovation relative to $D_t$.
Assume the conjugate variable $J_{D_t}(X:B)$ exists and there is $K<\infty$ such that
\[
\Phi_{D_t}^{\ast\ast}(X:B)\;\le\;K\,1_{D_t}.
\]
Then for every predictor $A\in D_t^{\mathrm{sa}}$,
\begin{equation}\label{eq:unavoidable-MSE}
\mathrm{MSE}_t(A)=\big\|E_t((R-A)^2)\big\|\;\ge\;\frac{1}{K}.
\end{equation}
In particular, even for the best predictor $A^\ast=E_t(R)$,
\[
\inf_{A\in D_t^{\mathrm{sa}}}\mathrm{MSE}_t(A)
=
\big\|E_t(X^2)\big\|
\ge\frac1K.
\]
\end{theorem}

\begin{proof}
By Theorem~\ref{thm:best-predictor},
\[
\mathrm{MSE}_t(A)\ge \mathrm{MSE}_t(A^\ast)=\|E_t(X^2)\|.
\]
By Corollary~\ref{cor:CR-one} (with $D=D_t$, $E_D=E_t$, and variable $X$),
\[
\|E_t(X^2)\|\ge \frac{1}{K}.
\]
Combining yields \eqref{eq:unavoidable-MSE}.
\end{proof}

\begin{remark}[This is an information-theoretic uncertainty principle]
Theorem~\ref{thm:return-lower-bound} has the structure
\[
\text{information }(D_t,E_t)\;\Rightarrow\;\text{best predictor }E_tR
\;\Rightarrow\;\text{minimal error }\|E_t(X^2)\|
\;\stackrel{\text{CR}}{\gtrsim}\;(\text{Fisher information})^{-1}.
\]
This differs from commutator-type uncertainty principles.
Here ``uncertainty'' is estimation-/information-theoretic: if Fisher information is bounded, mean-square error cannot vanish.
\end{remark}

\subsubsection*{A concrete example within this chapter (sharpness of the bound)}

\begin{example}[Semicircular innovations: the bound is attained]\label{ex:semicircular-innovation}
Take $B=D_t$ and assume there exists a self-adjoint element $S\in M^{\mathrm{sa}}$ such that:
\begin{enumerate}[label=(\roman*)]
\item $E_t(S)=0$ and $E_t(S^2)=1_{D_t}$;
\item $S$ is a (standard) $D_t$-valued semicircular innovation relative to $D_t$, so that the relevant conjugate variable exists
and the single-variable Cramér--Rao inequality is attained (equality case).
\end{enumerate}
Let $\sigma>0$ and define
\[
R:=\mu+\sigma S,\qquad \mu\in D_t^{\mathrm{sa}}.
\]
Then $A^\ast=E_t(R)=\mu$, the innovation is $X=R-E_t(R)=\sigma S$, and
\[
\|E_t(X^2)\|=\|E_t(\sigma^2 S^2)\|=\sigma^2.
\]
Moreover, by the equality case in the single-variable Cramér--Rao inequality,
\[
\Phi_{D_t}^{\ast\ast}(X:D_t)\;\|E_t(X^2)\|=1_{D_t},
\qquad\text{so}\qquad
\Phi_{D_t}^{\ast\ast}(X:D_t)=\frac{1}{\sigma^2}\,1_{D_t}.
\]
Thus one may take $K=1/\sigma^2$, and the lower bound in Theorem~\ref{thm:return-lower-bound} becomes
\[
\inf_{A\in D_t^{\mathrm{sa}}}\mathrm{MSE}_t(A)\ge \frac1K=\sigma^2,
\]
which is attained: the minimal error equals $\sigma^2$.
\end{example}

\begin{remark}[Compatibility with the Local Information Efficiency Principle (LIEP)]
If the (discounted) price process satisfies a (truncated) martingale condition under $E_t$ (LIEP),
then certain return increments may have controlled conditional means (e.g.\ vanishing drift at small scales or after centering).
The present chapter shows that even if conditional means can be forced to zero,
as long as Fisher information (information complexity) is bounded, there remains an unavoidable \emph{second-moment}
error floor. Quantitatively: zero drift does not imply zero predictability; the information structure itself limits forecasting.
\end{remark}


\subsection{Pure-jump (compound Poisson) innovations: an explicit error floor}\label{sec:poisson-innovation}

The Fisher--Cramér--Rao bound in Section~\ref{sec:fisher-cramer-rao} requires existence of conjugate variables.
For \emph{pure-jump} models (in particular, compound Poisson jump models), it is often more natural to compute
the minimal prediction error \emph{exactly} from the jump intensities.  This section gives a fully explicit
error floor for predicting log-returns from the market information $D_t$.

\subsubsection{Translation-invariant jump generator on a price grid}\label{subsec:poisson-generator}

Fix a grid size $\Delta x>0$ and consider the state space $\mathsf G:=\Delta x\,\mathbb Z$.
Let $\ell^\infty(\mathsf G)$ denote bounded functions $f:\mathsf G\to\mathbb C$.
For each $\alpha\in\mathbb Z$, define the shift operator
\[
(T_\alpha f)(x):=f(x+\alpha\Delta x),\qquad x\in\mathsf G.
\]
Let $(\gamma_\alpha)_{\alpha\in\mathbb Z}$ be nonnegative jump intensities satisfying
\begin{equation}\label{eq:poisson-intensity-assumptions}
\Lambda:=\sum_{\alpha\in\mathbb Z}\gamma_\alpha \;<\;\infty,
\qquad
\sum_{\alpha\in\mathbb Z}\alpha^2\gamma_\alpha \;<\;\infty.
\end{equation}
Define the (bounded) generator
\begin{equation}\label{eq:LXI}
(L_X f)(x)\;:=\;\sum_{\alpha\in\mathbb Z}\gamma_\alpha\bigl(f(x+\alpha\Delta x)-f(x)\bigr)
\;=\;\sum_{\alpha\in\mathbb Z}\gamma_\alpha\,(T_\alpha-I)f(x).
\end{equation}
Then $L_X$ generates a conservative Markov semigroup $(P_h)_{h\ge0}$ on $\ell^\infty(\mathsf G)$ via
$P_h=e^{hL_X}$.

\begin{assumption}[Embedding into the pricing information structure]\label{ass:poisson-embed}
There exists a commutative von Neumann subalgebra $\mathcal A\subseteq M$ and a self-adjoint process
$(X_t)_{t\ge0}$ affiliated with $\mathcal A$ such that:
\begin{enumerate}[label=(\roman*)]
\item under the pricing state (or physical state) on $\mathcal A$, the transition semigroup of $(X_t)$ is $(P_h)$
generated by \eqref{eq:LXI};
\item the market information algebra $D_t$ contains the past observation algebra of $X$ up to time $t$
(e.g.\ $\mathrm{vN}(X_s:\,0\le s\le t)\subseteq D_t$), and $E_t$ restricts to the usual conditional expectation on this
commutative sector.
\end{enumerate}
\end{assumption}

\subsubsection{Exponential eigenfunctions and the increment cumulant}\label{subsec:poisson-cumulant}

For $u\in\mathbb R$, define $f_u(x):=e^{ux}$ on $\mathsf G$.
Assume additionally that
\begin{equation}\label{eq:poisson-exp-moment-local}
\sum_{\alpha\in\mathbb Z}\gamma_\alpha\,e^{u\alpha\Delta x}\;<\;\infty
\quad\text{for $u$ in a neighborhood of $0$}.
\end{equation}
Then $f_u$ is an eigenfunction of $L_X$:

\begin{lemma}[Cumulant exponent]\label{lem:psi}
For $u$ satisfying \eqref{eq:poisson-exp-moment-local}, define
\begin{equation}\label{eq:psi}
\psi(u)\;:=\;\sum_{\alpha\in\mathbb Z}\gamma_\alpha\bigl(e^{u\alpha\Delta x}-1\bigr).
\end{equation}
Then
\[
L_X f_u \;=\; \psi(u)\,f_u,
\qquad\text{and hence}\qquad
P_h f_u \;=\; e^{h\psi(u)}\,f_u\quad(h\ge0).
\]
\end{lemma}

\begin{proof}
For each $\alpha$,
\[
(T_\alpha f_u)(x)=f_u(x+\alpha\Delta x)=e^{u(x+\alpha\Delta x)}=e^{u\alpha\Delta x}f_u(x).
\]
Substituting into \eqref{eq:LXI} yields
\[
(L_X f_u)(x)=\sum_{\alpha}\gamma_\alpha\bigl(e^{u\alpha\Delta x}-1\bigr)f_u(x)=\psi(u)f_u(x),
\]
which proves the eigenfunction identity. Since $P_h=e^{hL_X}$ and $f_u$ is an eigenvector of $L_X$,
we obtain $P_h f_u=e^{h\psi(u)}f_u$.
\end{proof}

\begin{corollary}[Conditional mgf of the increment]\label{cor:mgf-increment}
Under Assumption~\ref{ass:poisson-embed}, for $u$ as above and any $t\ge0,h\ge0$,
\[
E_t\bigl(e^{u(X_{t+h}-X_t)}\bigr)\;=\;e^{h\psi(u)}\,1_{D_t}.
\]
In particular, the conditional distribution of the increment $X_{t+h}-X_t$ given $D_t$ does not depend on $t$.
\end{corollary}

\begin{proof}
On the commutative sector, Markov property gives
\[
E\bigl(e^{uX_{t+h}}\mid X_t\bigr)=(P_h f_u)(X_t).
\]
Therefore,
\[
E\bigl(e^{u(X_{t+h}-X_t)}\mid X_t\bigr)
=\frac{(P_h f_u)(X_t)}{f_u(X_t)}=e^{h\psi(u)}.
\]
Since $\mathrm{vN}(X_s:\,s\le t)\subseteq D_t$ and $E_t$ restricts to the commutative conditional expectation,
the same identity holds with conditioning on $D_t$, and the right-hand side is scalar in $D_t$.
\end{proof}

\subsubsection{Explicit minimal prediction error for log-returns}\label{subsec:poisson-mse}

Fix $t\ge0$ and horizon $h>0$. Define the log-return observable
\[
R_{t,h}\;:=\;X_{t+h}-X_t \in M^{\mathrm{sa}}.
\]
By Theorem~\ref{thm:best-predictor}, the minimal mean-square prediction error from $D_t$ equals
$\|E_t((R_{t,h}-E_tR_{t,h})^2)\|$. In the present jump setting, this quantity is explicit.

\begin{theorem}[Exact error floor for compound Poisson jumps]\label{thm:poisson-error-floor}
Assume \eqref{eq:poisson-intensity-assumptions} and Assumption~\ref{ass:poisson-embed}.
Then for every $A\in D_t^{\mathrm{sa}}$,
\[
\bigl\|E_t\bigl((R_{t,h}-A)^2\bigr)\bigr\|
\;\ge\;
\inf_{A\in D_t^{\mathrm{sa}}}\bigl\|E_t\bigl((R_{t,h}-A)^2\bigr)\bigr\|
\;=\;
h(\Delta x)^2\sum_{\alpha\in\mathbb Z}\alpha^2\gamma_\alpha.
\]
Equivalently,
\[
\inf_{A\in D_t^{\mathrm{sa}}}\bigl\|E_t\bigl((R_{t,h}-A)^2\bigr)\bigr\|
\;=\;\mathrm{Var}(R_{t,h}).
\]
\end{theorem}

\begin{proof}
Step 1 (best predictor). By Theorem~\ref{thm:best-predictor}, the unique minimizer is $A^\ast=E_t(R_{t,h})$ and
\[
\inf_{A\in D_t^{\mathrm{sa}}}\bigl\|E_t\bigl((R_{t,h}-A)^2\bigr)\bigr\|
=\bigl\|E_t\bigl((R_{t,h}-E_tR_{t,h})^2\bigr)\bigr\|.
\]

Step 2 (compute conditional moments). By Corollary~\ref{cor:mgf-increment},
\[
E_t\bigl(e^{uR_{t,h}}\bigr)=e^{h\psi(u)}1_{D_t}.
\]
Hence $\log E_t(e^{uR_{t,h}})=h\psi(u)\,1_{D_t}$. Differentiating at $u=0$ (justified by
\eqref{eq:poisson-intensity-assumptions} and \eqref{eq:poisson-exp-moment-local} near $0$) gives
\[
E_t(R_{t,h})=h\psi'(0)\,1_{D_t}
=h\Big(\sum_{\alpha}\gamma_\alpha\,\alpha\Delta x\Big)\,1_{D_t}.
\]
Moreover, the conditional variance is the second derivative of the log-mgf at $0$:
\[
E_t\bigl((R_{t,h}-E_tR_{t,h})^2\bigr)
=\frac{d^2}{du^2}\Big|_{u=0}\log E_t(e^{uR_{t,h}})
=\frac{d^2}{du^2}\Big|_{u=0}\bigl(h\psi(u)\bigr)\,1_{D_t}
=h\psi''(0)\,1_{D_t}.
\]
Since $\psi''(0)=\sum_{\alpha}\gamma_\alpha(\alpha\Delta x)^2=(\Delta x)^2\sum_{\alpha}\alpha^2\gamma_\alpha$,
we obtain
\[
E_t\bigl((R_{t,h}-E_tR_{t,h})^2\bigr)
=h(\Delta x)^2\Big(\sum_{\alpha}\alpha^2\gamma_\alpha\Big)\,1_{D_t}.
\]
Taking $C^\ast$-norm yields the stated explicit value.
\end{proof}

\begin{remark}[Relation to Fisher--Cramér--Rao]
Theorem~\ref{thm:poisson-error-floor} gives an \emph{exact} (and typically sharp) error floor derived directly from
jump intensities. In pure-jump settings, conjugate variables in the sense of Section~\ref{sec:fisher-cramer-rao} may fail
to exist (equivalently, $\Phi^{\ast\ast}$ may be infinite), so the Fisher route may not be the appropriate tool; the
explicit computation above is the correct substitute.
\end{remark}

\subsubsection{Risk-neutral constraint and a fully explicit two-sided jump example}\label{subsec:poisson-riskneutral}

Assume the traded asset price is $S_t:=e^{X_t}$ and the money market account is $B_t:=e^{rt}$.
A risk-neutral (no-arbitrage) condition is that the discounted price is a martingale:
\[
E_t\Big(\frac{S_{t+h}}{B_{t+h}}\Big)=\frac{S_t}{B_t}
\quad\Longleftrightarrow\quad
E_t\bigl(e^{R_{t,h}}\bigr)=e^{rh}\,1_{D_t}.
\]
By Corollary~\ref{cor:mgf-increment} with $u=1$, this holds provided $u=1$ satisfies
\eqref{eq:poisson-exp-moment-local} and
\begin{equation}\label{eq:RN-constraint-jumps}
\psi(1)=\sum_{\alpha\in\mathbb Z}\gamma_\alpha\bigl(e^{\alpha\Delta x}-1\bigr)=r.
\end{equation}

\begin{example}[Symmetric $\pm1$ jumps]\label{ex:pm1-jumps}
Assume only $\pm1$ jumps occur:
\[
\gamma_{1}=\gamma_{-1}=\frac{\lambda}{2},\qquad \gamma_\alpha=0\ \ (\alpha\neq\pm1),
\quad\lambda>0.
\]
Then \eqref{eq:RN-constraint-jumps} becomes
\[
\frac{\lambda}{2}(e^{\Delta x}-1)+\frac{\lambda}{2}(e^{-\Delta x}-1)=r
\quad\Longleftrightarrow\quad
\lambda\bigl(\cosh(\Delta x)-1\bigr)=r,
\]
hence
\[
\lambda=\frac{r}{\cosh(\Delta x)-1}.
\]
By Theorem~\ref{thm:poisson-error-floor}, the minimal $D_t$-based prediction MSE for the $h$-horizon log-return is
\[
\inf_{A\in D_t^{\mathrm{sa}}}\bigl\|E_t\bigl((R_{t,h}-A)^2\bigr)\bigr\|
=h(\Delta x)^2(\gamma_{1}+\gamma_{-1})
=h(\Delta x)^2\lambda
=h(\Delta x)^2\,\frac{r}{\cosh(\Delta x)-1}.
\]
\end{example}

\appendix
\section{Quantum-Mechanical Background (Not Used in Main Proofs)}\label{app:qm}
This appendix collects standard operator-theoretic and quantum-mechanical background
(Born rule, projective measurement maps, and unitary dynamics).
\emph{None of the results in this appendix is used in the proofs of the main theorems} in Sections~\ref{sec:MFQM}; they are included only for contextual completeness and for possible non-commutative extensions.

Throughout, $\Hcal$ is a complex separable Hilbert space, $\BH=\mathcal B(\Hcal)$, and
$\Tone=\mathcal T_1(\Hcal)$ with trace $\Tr$.

\subsection{Spectral measures and the Born rule}\label{app:qm:born}

\begin{definition}[Projection-valued measure]\label{def:pvm-app}
A \emph{projection-valued measure} (PVM) on $\R$ is a map $E:\Bcal(\R)\to\BH$ such that:
\begin{enumerate}[label=(\roman*), itemsep=2pt, topsep=4pt]
\item $E(\emptyset)=0$ and $E(\R)=I$;
\item $E(F)$ is an orthogonal projection for every $F\in\Bcal(\R)$;
\item for pairwise disjoint $\{F_i\}_{i\in\mathbb N}\subset\Bcal(\R)$,
\[
E\Big(\bigsqcup_{i=1}^\infty F_i\Big)=\sum_{i=1}^\infty E(F_i),
\]
where the series converges in the strong operator topology.
\end{enumerate}
\end{definition}

\begin{theorem}[Spectral theorem for self-adjoint operators]\label{thm:spectral-app}
Let $A$ be a self-adjoint operator on $\Hcal$. There exists a unique PVM
$E_A:\Bcal(\R)\to\BH$ such that, in the sense of spectral calculus,
\[
A=\int_{\R}\lambda\,E_A(d\lambda).
\]
Moreover, for every bounded Borel function $f:\R\to\C$,
\[
f(A)=\int_{\R} f(\lambda)\,E_A(d\lambda)\in \BH.
\]
\end{theorem}

\begin{axiom}[Born rule]\label{ax:born-app}
Let $A$ be an observable with spectral measure $E_A$ as in
Theorem~\ref{thm:spectral-app}. For each state $\rho\in\Tone$ with $\rho\ge0$ and $\Tr(\rho)=1$,
and each $F\in\Bcal(\R)$, the probability of obtaining an outcome in $F$ upon measuring $A$ is
\[
\mathbb P_\rho(A\in F)=\Tr\!\big(E_A(F)\rho\big).
\]
\end{axiom}

\begin{proposition}[Born rule defines a probability measure]\label{prop:born-measure-app}
Fix a state $\rho$ and an observable $A$ with PVM $E_A$. The map
\[
\mu_\rho^A:\Bcal(\R)\to[0,1],\qquad \mu_\rho^A(F):=\Tr\!\big(E_A(F)\rho\big),
\]
is a probability measure on $(\R,\Bcal(\R))$.
\end{proposition}

\begin{proof}
Normalization follows from $E_A(\R)=I$ and $E_A(\emptyset)=0$:
$\mu_\rho^A(\R)=\Tr(\rho)=1$ and $\mu_\rho^A(\emptyset)=0$.
Let $\{F_i\}_{i\in\mathbb N}$ be pairwise disjoint and set $F=\bigsqcup_{i=1}^\infty F_i$.
Define $S_n:=\sum_{i=1}^n E_A(F_i)$. Then $S_n\to E_A(F)$ strongly by
Definition~\ref{def:pvm-app}. Since $\rho\in\Tone$ is positive, monotone continuity (normality)
of the trace on $\Tone^+$ yields
\[
\Tr\!\big(E_A(F)\rho\big)=\lim_{n\to\infty}\Tr\!\big(S_n\rho\big)
=\lim_{n\to\infty}\sum_{i=1}^n \Tr\!\big(E_A(F_i)\rho\big)
=\sum_{i=1}^\infty \Tr\!\big(E_A(F_i)\rho\big),
\]
i.e.\ countable additivity.
\end{proof}

\subsection{Projective measurements: selective and non-selective updates}\label{app:qm:meas}

\begin{axiom}[Selective (L\"uders) update]\label{ax:luders-app}
Let $A$ be an observable with PVM $E_A$. For $F\in\Bcal(\R)$ with $\Tr(E_A(F)\rho)>0$,
the post-measurement (selective) state conditioned on the event $\{A\in F\}$ is
\[
\rho_F:=\frac{E_A(F)\,\rho\,E_A(F)}{\Tr(E_A(F)\rho)}.
\]
\end{axiom}

\begin{proposition}[Well-posedness of the selective update]\label{prop:luders-wellposed-app}
Under Assumption~\ref{ax:luders-app}, $\rho_F$ is a density operator: $\rho_F\ge0$, $\rho_F\in\Tone$,
and $\Tr(\rho_F)=1$.
\end{proposition}

\begin{proof}
Positivity is immediate since $E_A(F)$ is a projection. Since $E_A(F)\in\BH$ and $\rho\in\Tone$,
we have $E_A(F)\rho E_A(F)\in\Tone$ and
\[
\Tr(\rho_F)=\frac{\Tr(E_A(F)\rho E_A(F))}{\Tr(E_A(F)\rho)}
=\frac{\Tr(E_A(F)\rho)}{\Tr(E_A(F)\rho)}=1,
\]
using cyclicity of the trace for bounded--trace-class products.
\end{proof}

\begin{definition}[Non-selective projective measurement map]\label{def:nonselective-app}
Assume $A$ has pure point spectrum with spectral projections $\{P_\lambda\}_{\lambda\in\sigma_p(A)}$.
The associated non-selective (unconditioned) post-measurement state is defined by
\[
\Phi_A(\rho):=\sum_{\lambda\in\sigma_p(A)} P_\lambda\,\rho\,P_\lambda,
\]
where the series converges in trace norm.
\end{definition}

\begin{lemma}[Trace-norm convergence of the non-selective update]\label{lem:nonselective-conv-app}
Let $A=A^*$ have pure point spectrum with spectral projections $\{P_k\}_{k\in\mathbb N}$
satisfying $\sum_{k\ge1}P_k=I$ strongly. For any $\rho\in\Tone$ with $\rho\ge0$, the series
\[
\Phi_A(\rho):=\sum_{k\ge1} P_k\,\rho\,P_k
\]
converges in trace norm to a positive trace-class operator, and $\Tr(\Phi_A(\rho))=\Tr(\rho)$.
\end{lemma}

\begin{proof}
Set $S_n:=\sum_{k=1}^n P_k\rho P_k\ge0$. For $m>n$,
$S_m-S_n=\sum_{k=n+1}^m P_k\rho P_k\ge0$, hence
\[
\|S_m-S_n\|_1=\Tr(S_m-S_n)=\sum_{k=n+1}^m \Tr(P_k\rho).
\]
Since $\sum_{k\ge1}P_k=I$ strongly and $\Tr$ is normal on $\Tone^+$,
$\sum_{k\ge1}\Tr(P_k\rho)=\Tr(\rho)$, so the tails
$\sum_{k=n+1}^m \Tr(P_k\rho)\to0$ as $m,n\to\infty$.
Thus $(S_n)$ is Cauchy in $\|\cdot\|_1$ and converges in $\Tone$ to $\Phi_A(\rho)\ge0$, and
$\Tr(\Phi_A(\rho))=\lim_{n\to\infty}\Tr(S_n)=\Tr(\rho)$.
\end{proof}

\subsection{Unitary dynamics and Stone's theorem}\label{app:qm:dynamics}

\begin{definition}[Strongly continuous unitary propagator]\label{def:propagator-app}
A family $\{U(t,s)\}_{t,s\in\R}\subset \BH$ is a \emph{strongly continuous unitary propagator} if:
\begin{enumerate}[label=(\roman*), itemsep=2pt, topsep=4pt]
\item $U(t,s)$ is unitary for all $t,s\in\R$;
\item $U(t,t)=I$ for all $t\in\R$ and $U(t,s)U(s,r)=U(t,r)$ for all $r,s,t\in\R$;
\item for each fixed $s\in\R$, the map $t\mapsto U(t,s)\psi$ is continuous in $\Hcal$ for every
$\psi\in\Hcal$, and similarly in $s$ for fixed $t$.
\end{enumerate}
\end{definition}

\begin{axiom}[Unitary dynamics]\label{ax:dynamics-app}
There exists a strongly continuous unitary propagator $\{U(t,s)\}_{t,s\in\R}$ such that the state
evolves according to
\[
\rho(t)=U(t,s)\,\rho(s)\,U(t,s)^\dagger,\qquad \forall\, t,s\in\R.
\]
\end{axiom}

\begin{theorem}[Stone]\label{thm:stone-app}
Let $(U(t))_{t\in\R}$ be a strongly continuous one-parameter unitary group on $\Hcal$.
Then there exists a unique self-adjoint operator $H$ on $\Hcal$ such that
\[
U(t)=e^{-\frac{i}{\hbar}tH}\qquad (t\in\R),
\]
where the exponential is defined by the functional calculus.
\end{theorem}

\begin{corollary}[Time-homogeneous dynamics]\label{cor:time-hom-app}
If in addition $U(t,s)=U(t-s,0)$ for all $t,s\in\R$, then $(U(t,0))_{t\in\R}$ is a strongly
continuous unitary group. Hence there exists a self-adjoint Hamiltonian $H$ with
\[
U(t,0)=e^{-\frac{i}{\hbar}tH}\qquad (t\in\R).
\]
\end{corollary}

\subsection*{References for Appendix~\ref{app:qm}}
For standard proofs and background used implicitly above (spectral theorem, trace-class ideals,
normality of the trace, Stone's theorem), see, e.g.,
\cite{ReedSimonI,ReedSimonII,Hall2013,KadisonRingroseI}.


\section{Functional-Analytic Tools: Banach and Separable Hilbert Spaces}
\label{chap:tools-functional-analysis}

\subsection{Normed and Banach spaces}

Throughout this section, all vector spaces are over $\mathbb C$.

\begin{definition}[Normed space; Banach space]
A \emph{normed space} is a pair $(V,\|\cdot\|_V)$ where $V$ is a complex vector space and
$\|\cdot\|_V:V\to[0,\infty)$ satisfies for all $x,y\in V$ and $z\in\mathbb C$:
(i) $\|x\|_V=0 \iff x=0$;
(ii) $\|zx\|_V=|z|\|x\|_V$;
(iii) $\|x+y\|_V\le \|x\|_V+\|y\|_V$.
A sequence $(x_n)$ is \emph{Cauchy} if $\forall\varepsilon>0\,\exists N\,\forall n,m\ge N:\|x_n-x_m\|_V<\varepsilon$.
The space $(V,\|\cdot\|_V)$ is \emph{complete} if every Cauchy sequence converges in $V$.
A \emph{Banach space} is a complete normed space.
\end{definition}

\subsection{Bounded linear operators and the operator norm}

\begin{definition}[Bounded linear operator]
Let $(V,\|\cdot\|_V)$ be a normed space and $(W,\|\cdot\|_W)$ a normed space.
A linear map $A:V\to W$ is called \emph{bounded} if
\[
\sup_{f\in V\setminus\{0\}}\frac{\|Af\|_W}{\|f\|_V}<\infty.
\]
\end{definition}

\begin{proposition}[Equivalent characterizations of boundedness]
\label{prop:bounded-iff-continuous}
Let $A:V\to W$ be linear. The following are equivalent:
\begin{enumerate}[label=(\roman*)]
\item $\displaystyle \sup_{\|f\|_V=1}\|Af\|_W<\infty$;
\item $\exists k>0$ such that $\|f\|_V\le 1 \Rightarrow \|Af\|_W\le k$;
\item $\exists k>0$ such that $\forall f\in V:\ \|Af\|_W\le k\|f\|_V$;
\item $A:V\to W$ is continuous (norm topologies);
\item $A$ is continuous at $0\in V$.
\end{enumerate}
\end{proposition}

\begin{proof}
\emph{(i)$\Rightarrow$(iii)} Let $M:=\sup_{\|f\|_V=1}\|Af\|_W<\infty$. For $f\neq 0$ write $f=\|f\|_V\,\tilde f$ with
$\|\tilde f\|_V=1$. Then $\|Af\|_W=\|f\|_V\|A\tilde f\|_W\le M\|f\|_V$.

\emph{(iii)$\Rightarrow$(ii)} is immediate.

\emph{(ii)$\Rightarrow$(i)} If $\|f\|_V=1$ then $\|Af\|_W\le k$, so the supremum is $\le k$.

\emph{(iii)$\Rightarrow$(iv)} If $f_n\to f$ in $V$, then
\[
\|Af_n-Af\|_W=\|A(f_n-f)\|_W\le k\|f_n-f\|_V\to 0.
\]

\emph{(iv)$\Rightarrow$(v)} is trivial.

\emph{(v)$\Rightarrow$(iii)} By continuity at $0$, there exists $\delta>0$ such that
$\|f\|_V<\delta \Rightarrow \|Af\|_W<1$. For arbitrary $f\neq 0$, set $g:=\frac{\delta}{2\|f\|_V}f$.
Then $\|g\|_V=\delta/2<\delta$, hence $\|Ag\|_W<1$. By linearity,
\[
\|Af\|_W=\frac{2\|f\|_V}{\delta}\|Ag\|_W \le \frac{2}{\delta}\|f\|_V.
\]
Thus (iii) holds with $k=2/\delta$.
\end{proof}

\begin{definition}[Operator norm]
If $A:V\to W$ is bounded, define its \emph{operator norm} by
\[
\|A\|:=\sup_{\|f\|_V=1}\|Af\|_W=\sup_{f\in V\setminus\{0\}}\frac{\|Af\|_W}{\|f\|_V}.
\]
\end{definition}

\begin{lemma}[Basic inequality]
\label{lem:basic-ineq}
If $A:V\to W$ is bounded, then $\|Af\|_W\le \|A\|\,\|f\|_V$ for all $f\in V$.
\end{lemma}
\begin{proof}
If $f=0$ it is trivial. If $f\neq 0$, write $f=\|f\|_V \tilde f$ with $\|\tilde f\|_V=1$ and use the definition of $\|A\|$.
\end{proof}

\subsection{The Banach space of bounded operators; duals and weak convergence}

Let $\mathcal L(V,W)$ denote the vector space of bounded linear maps $V\to W$.

\begin{theorem}[$\mathcal L(V,W)$ is Banach when $W$ is Banach]
\label{thm:LVW-banach}
Let $(V,\|\cdot\|_V)$ be normed and $(W,\|\cdot\|_W)$ Banach.
Then $\big(\mathcal L(V,W),\|\cdot\|\big)$ is a Banach space.
\end{theorem}

\begin{proof}
Let $(A_n)_{n\in\mathbb N}$ be Cauchy in operator norm. Fix $f\in V$.
Then for $n,m$,
\[
\|A_nf-A_mf\|_W \le \|A_n-A_m\|\,\|f\|_V,
\]
so $(A_nf)_n$ is Cauchy in $W$ and hence converges (since $W$ is Banach). Define $Af:=\lim_{n\to\infty}A_nf$.

Linearity of $A$ follows from pointwise limits of linear maps. Next, $(A_n)$ is bounded in operator norm:
choose $N$ with $\|A_n-A_N\|<1$ for $n\ge N$, then $\|A_n\|\le \|A_N\|+1$ for $n\ge N$, hence
$M:=\sup_n\|A_n\|<\infty$. Therefore for all $f$,
\[
\|Af\|_W=\lim_{n\to\infty}\|A_nf\|_W\le \limsup_{n\to\infty}\|A_n\|\,\|f\|_V\le M\|f\|_V,
\]
so $A\in\mathcal L(V,W)$.

Finally, we show $\|A_n-A\|\to 0$. For any $f$ with $\|f\|_V=1$,
\[
\|(A_n-A)f\|_W = \lim_{m\to\infty}\|(A_n-A_m)f\|_W \le \limsup_{m\to\infty}\|A_n-A_m\|.
\]
Taking the supremum over $\|f\|_V=1$ yields
\[
\|A_n-A\| \le \limsup_{m\to\infty}\|A_n-A_m\|.
\]
Since $(A_n)$ is Cauchy, the right-hand side tends to $0$ as $n\to\infty$. Hence $A_n\to A$ in operator norm,
and $\mathcal L(V,W)$ is complete.
\end{proof}

\begin{definition}[Dual space]
The \emph{dual} of a normed space $V$ is
\[
V^*:=\mathcal L(V,\mathbb C),
\]
whose elements are bounded linear functionals on $V$.
\end{definition}

\begin{definition}[Weak convergence]
A sequence $(f_n)\subset V$ converges \emph{weakly} to $f\in V$, written $f_n\rightharpoonup f$, if
\[
\forall \varphi\in V^*:\quad \lim_{n\to\infty}\varphi(f_n)=\varphi(f).
\]
\end{definition}

\begin{proposition}[Strong convergence implies weak convergence]
If $f_n\to f$ in norm in $V$, then $f_n\rightharpoonup f$ in $V$.
\end{proposition}
\begin{proof}
Let $\varphi\in V^*$ and $\varepsilon>0$. Since $\varphi$ is bounded, $|\varphi(g)|\le \|\varphi\|\,\|g\|_V$.
Choose $N$ such that $\|f_n-f\|_V<\varepsilon/\|\varphi\|$ for $n\ge N$ (if $\varphi=0$, the statement is trivial).
Then for $n\ge N$,
\[
|\varphi(f_n)-\varphi(f)|=|\varphi(f_n-f)|\le \|\varphi\|\,\|f_n-f\|_V<\varepsilon.
\]
\end{proof}

\subsection{BLT extension theorem (densely-defined bounded operators)}

\begin{theorem}[BLT extension theorem]
\label{thm:BLT}
Let $(V,\|\cdot\|_V)$ be a normed space, let $W$ be Banach, and let $D\subset V$ be a dense linear subspace.
If $A:D\to W$ is bounded (with respect to $\|\cdot\|_V$ restricted to $D$), then there exists a unique bounded linear
operator $\widetilde A:V\to W$ such that $\widetilde A|_D=A$ and $\|\widetilde A\|=\|A\|$.
\end{theorem}

\begin{proof}
For $v\in V$ choose a sequence $(v_n)\subset D$ with $v_n\to v$ in $V$ (density).
Then
\[
\|A v_n-A v_m\|_W \le \|A\|\,\|v_n-v_m\|_V,
\]
so $(A v_n)$ is Cauchy in $W$ and hence converges. Define
\[
\widetilde A v := \lim_{n\to\infty}A v_n\in W.
\]
\emph{Well-definedness:} if $(v_n')\subset D$ also satisfies $v_n'\to v$, then
\[
\|A v_n-A v_n'\|_W \le \|A\|\,\|v_n-v_n'\|_V \to 0,
\]
so both limits coincide. Linearity follows by passing to limits along approximating sequences.

\emph{Boundedness and norm:} for $v\in V$ and approximating $v_n\to v$,
\[
\|\widetilde A v\|_W=\lim_{n\to\infty}\|A v_n\|_W \le \limsup_{n\to\infty}\|A\|\,\|v_n\|_V=\|A\|\,\|v\|_V,
\]
hence $\|\widetilde A\|\le \|A\|$. Conversely, since $\widetilde A|_D=A$, we have $\|A\|\le \|\widetilde A\|$,
thus $\|\widetilde A\|=\|A\|$.

\emph{Uniqueness:} if $B:V\to W$ is bounded and $B|_D=A$, then for any $v\in V$ and $v_n\in D$ with $v_n\to v$,
\[
Bv=\lim_{n\to\infty}Bv_n=\lim_{n\to\infty}Av_n=\widetilde Av,
\]
so $B=\widetilde A$.
\end{proof}

\subsection{Hilbert spaces, orthogonality, and projections}

\begin{definition}[Hilbert space]
A \emph{Hilbert space} is a complex vector space $\mathcal H$ equipped with an inner product
$\langle\cdot|\cdot\rangle:\mathcal H\times\mathcal H\to\mathbb C$ which is
conjugate-linear in the first argument, linear in the second, Hermitian
$\langle x|y\rangle=\overline{\langle y|x\rangle}$, and positive definite
$\langle x|x\rangle>0$ for $x\neq 0$, and such that the induced norm
$\|x\|:=\sqrt{\langle x|x\rangle}$ is complete.
\end{definition}

\begin{lemma}[Cauchy--Schwarz]
For all $x,y\in\mathcal H$, $|\langle x|y\rangle|\le \|x\|\,\|y\|$.
\end{lemma}
\begin{proof}
If $x=0$ it is trivial. For $t\in\mathbb C$,
\[
0\le \|y-tx\|^2=\langle y-tx|y-tx\rangle=\|y\|^2-\overline t\,\langle x|y\rangle-t\,\langle y|x\rangle+|t|^2\|x\|^2.
\]
Choose $t=\langle x|y\rangle/\|x\|^2$ to obtain $0\le \|y\|^2-|\langle x|y\rangle|^2/\|x\|^2$.
\end{proof}

\begin{lemma}[Parallelogram identity]
For all $x,y\in\mathcal H$, $\|x+y\|^2+\|x-y\|^2=2\|x\|^2+2\|y\|^2$.
\end{lemma}
\begin{proof}
Expand using bilinearity and Hermitian symmetry.
\end{proof}

\begin{definition}[Orthogonal complement]
For a subset $M\subset\mathcal H$, define
\[
M^\perp:=\{\,x\in\mathcal H:\ \langle x|m\rangle=0\ \forall m\in M\,\}.
\]
If $M$ is a linear subspace then $M^\perp$ is a closed linear subspace.
\end{definition}

\begin{theorem}[Orthogonal projection theorem]
\label{thm:orth-proj}
Let $M\subset\mathcal H$ be a closed linear subspace. Then for every $x\in\mathcal H$ there exist unique
$m\in M$ and $n\in M^\perp$ such that $x=m+n$. The map $P_M:\mathcal H\to\mathcal H$, $P_M x:=m$,
is a bounded linear operator satisfying
\[
P_M^2=P_M,\qquad P_M^*=P_M,\qquad \|P_M\|=1\ \text{ unless }M=\{0\}.
\]
Moreover, $\mathrm{Ran}(P_M)=M$ and $\ker(P_M)=M^\perp$.
\end{theorem}

\begin{proof}
Fix $x\in\mathcal H$ and set $d:=\inf_{m\in M}\|x-m\|$. Choose a minimizing sequence $(m_k)\subset M$
with $\|x-m_k\|\to d$.
By the parallelogram identity,
\[
\left\|x-\frac{m_k+m_\ell}{2}\right\|^2+\left\|\frac{m_k-m_\ell}{2}\right\|^2
=\frac12\|x-m_k\|^2+\frac12\|x-m_\ell\|^2.
\]
Since $(m_k+m_\ell)/2\in M$, the first term is $\ge d^2$, hence
\[
\left\|\frac{m_k-m_\ell}{2}\right\|^2
\le \frac12\|x-m_k\|^2+\frac12\|x-m_\ell\|^2-d^2 \xrightarrow[k,\ell\to\infty]{} 0.
\]
Thus $(m_k)$ is Cauchy in $M$, and since $M$ is closed it converges to some $m\in M$.
Then $\|x-m\|=d$ by continuity of the norm.

Set $n:=x-m$. We claim $n\in M^\perp$. For any $u\in M$ and $t\in\mathbb R$,
the point $m+tu\in M$ and minimality gives $\|x-(m+tu)\|^2\ge \|x-m\|^2$, i.e.
\[
\|n-tu\|^2\ge \|n\|^2.
\]
Expanding $\|n-tu\|^2=\|n\|^2-2t\,\mathrm{Re}\langle n|u\rangle+t^2\|u\|^2$ yields
$\mathrm{Re}\langle n|u\rangle=0$ for all $u\in M$. Applying the same argument to $iu$ gives
$\mathrm{Im}\langle n|u\rangle=0$, hence $\langle n|u\rangle=0$ and $n\in M^\perp$.

Uniqueness: if $x=m_1+n_1=m_2+n_2$ with $m_i\in M$, $n_i\in M^\perp$, then $m_1-m_2=n_2-n_1\in M\cap M^\perp=\{0\}$.

Linearity of $P_M$ follows from uniqueness of decomposition. Idempotence and self-adjointness follow from
standard identities: $P_M x\in M$ and $x-P_M x\in M^\perp$. The bound $\|P_M\|\le 1$ follows from
Pythagoras $\|x\|^2=\|P_M x\|^2+\|x-P_M x\|^2$, and $\|P_M\|=1$ if $M\neq\{0\}$ by evaluating on $x\in M$.
\end{proof}

\begin{theorem}[Orthogonal projectors]
\label{thm:projectors}
Let $P\in\mathcal L(\mathcal H)$ satisfy $P^2=P$ and $P^*=P$. Then $\mathrm{ran}(P)$ is closed and
$P=P_{\mathrm{ran}(P)}$ is the orthogonal projection onto $\mathrm{ran}(P)$.
\end{theorem}

\begin{proof}
For any $x$, $Px\in\mathrm{ran}(P)$ and $x-Px\in\ker(P)$ since $P(x-Px)=Px-P^2x=0$.
Moreover, if $y\in\mathrm{ran}(P)$ then $y=Pz$ for some $z$, hence
$\langle x-Px|y\rangle=\langle x-Px|Pz\rangle=\langle P(x-Px)|z\rangle=0$,
so $\ker(P)\subset \mathrm{ran}(P)^\perp$. Conversely, if $x\in\mathrm{ran}(P)^\perp$, then
$\langle Px|Px\rangle=\langle x|P^*Px\rangle=\langle x|P^2x\rangle=\langle x|Px\rangle=0$, hence $Px=0$ and
$x\in\ker(P)$. Thus $\ker(P)=\mathrm{Ran}(P)^\perp$, and $\mathcal H=\mathrm{ran}(P)\oplus \ker(P)$
implies $\mathrm{ran}(P)$ is closed. By Theorem~\ref{thm:orth-proj} the orthogonal projection onto $\mathrm{ran}(P)$
is unique, hence equals $P$.
\end{proof}

\subsection{Riesz representation and the bra--ket identification}

\begin{theorem}[Riesz representation theorem]
\label{thm:riesz}
Let $\mathcal H$ be a Hilbert space and let $\varphi\in \mathcal H^*=\mathcal L(\mathcal H,\mathbb C)$.
Then there exists a unique $y\in\mathcal H$ such that
\[
\forall x\in\mathcal H:\quad \varphi(x)=\langle y|x\rangle.
\]
Moreover, $\|\varphi\|=\|y\|$.
\end{theorem}

\begin{proof}
If $\varphi=0$, take $y=0$. Otherwise, let $M:=\ker(\varphi)$, a closed subspace (continuity of $\varphi$).
By Theorem~\ref{thm:orth-proj}, $\mathcal H=M\oplus M^\perp$. Since $\varphi\neq 0$, $M\neq\mathcal H$ hence $M^\perp\neq\{0\}$.
Pick $u\in M^\perp$, $u\neq 0$. For any $x\in\mathcal H$, write $x=m+\alpha u$ with $m\in M$, $\alpha\in\mathbb C$.
Then $\varphi(x)=\varphi(m)+\alpha\varphi(u)=\alpha\varphi(u)$.

We determine $\alpha$ from inner products: because $u\perp M$, we have $\langle u|x\rangle=\langle u|m\rangle+\alpha\langle u|u\rangle
=\alpha\|u\|^2$, so $\alpha=\langle u|x\rangle/\|u\|^2$. Hence
\[
\varphi(x)=\frac{\varphi(u)}{\|u\|^2}\,\langle u|x\rangle = \left\langle \overline{\frac{\varphi(u)}{\|u\|^2}}\,u \,\middle|\, x\right\rangle.
\]
Therefore $y:=\overline{\frac{\varphi(u)}{\|u\|^2}}\,u$ works. Uniqueness follows from non-degeneracy of the inner product:
if $\langle y_1-y_2|x\rangle=0$ for all $x$, then $y_1=y_2$.

Finally,
\[
\|\varphi\|=\sup_{\|x\|=1}|\langle y|x\rangle|\le \|y\|
\]
by Cauchy--Schwarz, while taking $x=y/\|y\|$ gives equality. Thus $\|\varphi\|=\|y\|$.
\end{proof}

\begin{remark}[Bra--ket convention and anti-linearity]\label{rem:A1_bra_ket_antilinear}
With the convention that $\langle\cdot|\cdot\rangle$ is conjugate-linear in the first argument and linear in the second,
define
\[
J:\mathcal H\to \mathcal H^*,\qquad (J(y))(x):=\langle y|x\rangle,\qquad x\in\mathcal H.
\]
Then $J$ is \emph{conjugate-linear} (anti-linear), i.e.
\[
J(\alpha y_1+\beta y_2)=\overline{\alpha}\,J(y_1)+\overline{\beta}\,J(y_2)
\qquad(\alpha,\beta\in\mathbb C;\ y_1,y_2\in\mathcal H),
\]
and $J$ is an isometry:
\[
\|J(y)\|=\sup_{\|x\|=1}|\langle y|x\rangle|
\le \|y\|,\qquad
\text{while equality holds for }x=y/\|y\|\ \text{if }y\neq 0,
\]
hence $\|J(y)\|=\|y\|$ (the case $y=0$ is trivial).
We write $|y\rangle:=y\in\mathcal H$ and $\langle y|:=J(y)\in\mathcal H^*$.
\end{remark}

\begin{lemma}[Rank-one operators]
For $y,z\in\mathcal H$, define $|y\rangle\langle z|:\mathcal H\to\mathcal H$ by
\[
(|y\rangle\langle z|)\,x := y\,\langle z|x\rangle.
\]
Then $|y\rangle\langle z|\in \mathcal L(\mathcal H)$ and $\||y\rangle\langle z|\|=\|y\|\,\|z\|$.
\end{lemma}

\begin{proof}
For $\|x\|=1$, Cauchy--Schwarz yields $\|y\langle z|x\rangle\|\le \|y\|\,|\langle z|x\rangle|\le \|y\|\,\|z\|$,
so $\||y\rangle\langle z|\|\le \|y\|\,\|z\|$. Equality holds by taking $x=z/\|z\|$ (if $z\neq 0$), giving
$\||y\rangle\langle z|\,(z/\|z\|)\|=\|y\|\,\|z\|$. If $z=0$ the operator is zero and the identity holds.
\end{proof}

\subsection{Separable Hilbert spaces and the \texorpdfstring{$\ell^2$}{l\string^2} model}

\begin{definition}[Separable]
A metric space (hence a normed space) is \emph{separable} if it has a countable dense subset.
\end{definition}

\begin{definition}[Orthonormal basis]
A subset $\{e_n\}_{n\in I}\subset\mathcal H$ is \emph{orthonormal} if $\langle e_i|e_j\rangle=\delta_{ij}$.
An orthonormal set $\{e_n\}_{n\in I}$ is an \emph{orthonormal basis} if its linear span is dense in $\mathcal H$.
\end{definition}

\begin{proposition}[Separable $\Longleftrightarrow$ countable ONB]
\label{prop:separable-onb}
Let $\mathcal H$ be an infinite-dimensional Hilbert space. Then $\mathcal H$ is separable
if and only if it admits a countably infinite orthonormal basis.
\end{proposition}

\begin{proof}
($\Rightarrow$) Assume $\mathcal H$ separable. Let $D=\{x_1,x_2,\dots\}$ be a countable dense subset.
Construct inductively an orthonormal sequence $(e_n)$ as follows:
set $y_1:=x_1$. If $y_1\neq 0$, let $e_1:=y_1/\|y_1\|$.
Given $e_1,\dots,e_{n-1}$, choose the smallest index $k$ such that
\[
y_n:=x_k-\sum_{j=1}^{n-1}\langle e_j|x_k\rangle e_j \neq 0,
\]
and set $e_n:=y_n/\|y_n\|$. This Gram--Schmidt procedure produces a countable orthonormal set.
Let $M:=\overline{\mathrm{span}}\{e_n:n\in\mathbb N\}$.
By construction, for each $x_k\in D$, either $x_k\in M$ already or it generates a new $e_n$,
hence $D\subset M$. Since $D$ is dense, $M=\mathcal H$, so $(e_n)$ is an orthonormal basis.

($\Leftarrow$) If $\{e_n\}_{n\in\mathbb N}$ is an orthonormal basis, then the set of all finite linear combinations
with rational (in $\mathbb Q+i\mathbb Q$) coefficients is countable and dense in $\mathcal H$.
\end{proof}

\begin{theorem}[Canonical $\ell^2$ representation]
\label{thm:H-is-l2}
Let $\mathcal H$ be an infinite-dimensional separable Hilbert space and let $\{e_n\}_{n\in\mathbb N}$ be an orthonormal basis.
Define
\[
U:\mathcal H\to \ell^2(\mathbb N),\qquad Ux := \big(\langle e_1|x\rangle,\langle e_2|x\rangle,\dots\big).
\]
Then $U$ is unitary (an isometric surjection). In particular, $\mathcal H\simeq \ell^2(\mathbb N)$ as Hilbert spaces.
\end{theorem}

\begin{proof}
For any $x\in\mathcal H$, Bessel's inequality gives $\sum_{n=1}^\infty|\langle e_n|x\rangle|^2\le \|x\|^2$,
so $Ux\in \ell^2$. If $\{e_n\}$ is an orthonormal basis, Parseval's identity holds:
\[
\|x\|^2=\sum_{n=1}^\infty|\langle e_n|x\rangle|^2=\|Ux\|_{\ell^2}^2,
\]
hence $U$ is an isometry.

To show surjectivity, take $a=(a_n)\in\ell^2(\mathbb N)$ and consider partial sums $s_N:=\sum_{n=1}^N a_n e_n$.
Then for $M>N$,
\[
\|s_M-s_N\|^2=\left\|\sum_{n=N+1}^M a_n e_n\right\|^2=\sum_{n=N+1}^M |a_n|^2\to 0,
\]
so $(s_N)$ is Cauchy and converges to some $x\in\mathcal H$. Continuity of inner products yields
$\langle e_k|x\rangle=a_k$, hence $Ux=a$. Therefore $U$ is surjective and unitary.
\end{proof}

\subsection{Operator spaces \texorpdfstring{$B(\mathcal H)$, trace class $T_1(\mathcal H)$ and Hilbert--Schmidt $T_2(\mathcal H)$}{B(H), trace class T\string_1(H), Hilbert--Schmidt T\string_2(H)}}

\begin{definition}[Bounded operators]
Let $\mathcal H$ be a Hilbert space. Denote by $B(\mathcal H):=\mathcal L(\mathcal H,\mathcal H)$ the Banach space of
bounded operators on $\mathcal H$, equipped with the operator norm.
\end{definition}

\begin{definition}[Hilbert--Schmidt operators]
An operator $A\in B(\mathcal H)$ is \emph{Hilbert--Schmidt} if for some (hence any) orthonormal basis $(e_n)$,
\[
\|A\|_2^2:=\sum_{n=1}^\infty \|Ae_n\|^2 <\infty.
\]
The set $T_2(\mathcal H)$ of Hilbert--Schmidt operators becomes a Hilbert space with inner product
$\langle A,B\rangle_{HS}:=\mathrm{Tr}(A^*B)$ (well-defined on $T_2$).
\end{definition}

\begin{definition}[Trace-class operators]
An operator $A\in B(\mathcal H)$ is \emph{trace class} if $|A|:=(A^*A)^{1/2}$ satisfies
\[
\|A\|_1:=\mathrm{Tr}(|A|)<\infty.
\]
The set $T_1(\mathcal H)$ of trace-class operators is a Banach space under $\|\cdot\|_1$.
\end{definition}

\begin{remark}
The basic ideal properties (e.g. $B(\mathcal H)\,T_1(\mathcal H)\subset T_1(\mathcal H)$ and
$B(\mathcal H)\,T_2(\mathcal H)\subset T_2(\mathcal H)$) and completeness of $T_1,T_2$
are standard; if you want, we can expand this subsection into fully self-contained proofs
(typically by singular value decomposition / polar decomposition).
\end{remark}


\subsection{Trace-class and Hilbert--Schmidt operators: cyclicity, Hölder, ideal properties}

\newcommand{\Ttwo}{\mathcal T_2(\Hcal)} 

\begin{definition}[Rank-one operators]
For $u,v\in \Hcal$, define the rank-one operator $|u\rangle\langle v|\in B(\Hcal)$ by
\[
(|u\rangle\langle v|)x := u\,\langle v|x\rangle,\qquad x\in\Hcal.
\]
\end{definition}

\begin{lemma}[Operator norm of rank-one maps]
\label{lem:rank-one-norm}
For all $u,v\in\Hcal$, $\ \||u\rangle\langle v|\| = \|u\|\,\|v\|$.
\end{lemma}
\begin{proof}
For $\|x\|=1$, $\|u\langle v|x\rangle\|\le \|u\|\,|\langle v|x\rangle|\le \|u\|\,\|v\|$, so $\||u\rangle\langle v|\|\le \|u\|\,\|v\|$.
If $v\neq 0$, taking $x=v/\|v\|$ yields equality; if $v=0$ the operator is $0$.
\end{proof}

\begin{definition}[Trace-class operators (nuclear definition)]
\label{def:trace-class-nuclear}
An operator $S\in B(\Hcal)$ is \emph{trace-class} if it admits a representation
\[
S=\sum_{k=1}^\infty |u_k\rangle\langle v_k|
\quad \text{with}\quad
\sum_{k=1}^\infty \|u_k\|\,\|v_k\|<\infty,
\]
where the series converges in operator norm (hence strongly).
Define the \emph{trace norm} by
\[
\|S\|_1 := \inf\left\{\sum_{k=1}^\infty \|u_k\|\,\|v_k\|:\ S=\sum_{k\ge 1}|u_k\rangle\langle v_k|\right\}.
\]
Denote the space by $\Tone$.
\end{definition}

\begin{lemma}[Absolute convergence and boundedness]
\label{lem:T1-bounded}
If $S\in \Tone$, then $\|S\|\le \|S\|_1$.
\end{lemma}
\begin{proof}
Take any decomposition $S=\sum_k |u_k\rangle\langle v_k|$. Then by Lemma~\ref{lem:rank-one-norm},
\[
\left\|\sum_{k=1}^n |u_k\rangle\langle v_k|\right\|
\le \sum_{k=1}^n \||u_k\rangle\langle v_k|\|
= \sum_{k=1}^n \|u_k\|\,\|v_k\|.
\]
Letting $n\to\infty$ gives $\|S\|\le \sum_k \|u_k\|\,\|v_k\|$. Taking infimum over all decompositions yields $\|S\|\le \|S\|_1$.
\end{proof}

\begin{definition}[Trace on $\Tone$]
\label{def:trace-on-T1}
For a rank-one operator $|u\rangle\langle v|$, define
\[
\Tr(|u\rangle\langle v|):=\langle v|u\rangle.
\]
For a general $S\in\Tone$ with a nuclear decomposition $S=\sum_{k\ge 1}|u_k\rangle\langle v_k|$,
define
\[
\Tr(S):=\sum_{k=1}^\infty \langle v_k|u_k\rangle.
\]
\end{definition}

\begin{lemma}[Well-definedness and basic bound]
\label{lem:trace-well-defined}
The definition of $\Tr(S)$ in Definition~\ref{def:trace-on-T1} is independent of the chosen nuclear decomposition.
Moreover,
\[
|\Tr(S)|\le \|S\|_1\qquad \forall S\in\Tone.
\]
\end{lemma}

\begin{proof}
Let $S=\sum_k |u_k\rangle\langle v_k|$ with $\sum_k\|u_k\|\,\|v_k\|<\infty$.
Then by Cauchy--Schwarz, $|\langle v_k|u_k\rangle|\le \|v_k\|\,\|u_k\|$, so
$\sum_k |\langle v_k|u_k\rangle|<\infty$ and the trace series converges absolutely, yielding
$|\Tr(S)|\le \sum_k\|u_k\|\,\|v_k\|$. Taking the infimum over decompositions gives $|\Tr(S)|\le \|S\|_1$.

To show independence, it suffices to prove: if $S=0$ then any such series satisfies $\sum_k\langle v_k|u_k\rangle=0$.
Fix an orthonormal basis $(e_n)$ of $\Hcal$. For rank-one operators,
\[
\sum_{n=1}^\infty \langle e_n|\,|u\rangle\langle v|\,e_n\rangle
=\sum_{n=1}^\infty \langle e_n|u\rangle\langle v|e_n\rangle
=\left\langle v\ \middle|\ \sum_{n=1}^\infty \langle e_n|u\rangle e_n\right\rangle
=\langle v|u\rangle,
\]
where we used Parseval/orthonormal expansion of $u$.
By absolute convergence (Tonelli/Fubini is justified),
\[
\sum_{n=1}^\infty \langle e_n|S e_n\rangle
=\sum_{n=1}^\infty \sum_{k=1}^\infty \langle e_n|u_k\rangle\langle v_k|e_n\rangle
=\sum_{k=1}^\infty \langle v_k|u_k\rangle.
\]
If $S=0$, then $\langle e_n|S e_n\rangle=0$ for all $n$, hence the left-hand side is $0$ and so is $\sum_k\langle v_k|u_k\rangle$.
Thus the trace is independent of representation.
\end{proof}

\begin{theorem}[Ideal property for $\Tone$ and Hölder inequality]
\label{thm:T1-ideal-holder}
Let $A,B\in B(\Hcal)$ and $S\in\Tone$. Then $ASB\in\Tone$ and
\[
\|ASB\|_1\le \|A\|\,\|S\|_1\,\|B\|.
\]
Moreover,
\[
|\Tr(AS)|\le \|A\|\,\|S\|_1
\qquad\text{and}\qquad
|\Tr(SA)|\le \|A\|\,\|S\|_1.
\]
\end{theorem}

\begin{proof}
Take a nuclear decomposition $S=\sum_k |u_k\rangle\langle v_k|$ with $\sum_k\|u_k\|\,\|v_k\|<\infty$.
Then
\[
ASB=\sum_{k=1}^\infty A|u_k\rangle\langle v_k|B
=\sum_{k=1}^\infty |Au_k\rangle\langle B^\ast v_k|.
\]
Hence $ASB\in\Tone$ and
\[
\|ASB\|_1 \le \sum_k \|Au_k\|\,\|B^\ast v_k\|
\le \|A\|\,\|B\| \sum_k \|u_k\|\,\|v_k\|.
\]
Taking infimum over all decompositions yields $\|ASB\|_1\le \|A\|\,\|S\|_1\,\|B\|$.

For the trace bound,
\[
|\Tr(AS)|=|\Tr((A)S)|\le \|AS\|_1 \le \|A\|\,\|S\|_1,
\]
and similarly for $|\Tr(SA)|$.
\end{proof}

\begin{theorem}[Trace cyclicity for bounded--trace-class products]
\label{thm:trace-cyclicity}
Let $A\in B(\Hcal)$ and $S\in\Tone$. Then
\[
\Tr(AS)=\Tr(SA).
\]
More generally, for $A,B\in B(\Hcal)$ and $S\in\Tone$,
\[
\Tr(ASB)=\Tr(BAS).
\]
\end{theorem}

\begin{proof}
It suffices to prove $\Tr(AS)=\Tr(SA)$, since the general case follows by applying it to $(BA)$ and $S$.
Take a nuclear decomposition $S=\sum_k |u_k\rangle\langle v_k|$ with absolute convergence.
Then by linearity and absolute convergence,
\[
\Tr(AS)=\sum_k \Tr\big(A|u_k\rangle\langle v_k|\big)
=\sum_k \langle v_k|Au_k\rangle.
\]
On the other hand,
\[
\Tr(SA)=\sum_k \Tr\big(|u_k\rangle\langle v_k|A\big)
=\sum_k \Tr\big(|u_k\rangle\langle A^\ast v_k|\big)
=\sum_k \langle A^\ast v_k|u_k\rangle
=\sum_k \langle v_k|Au_k\rangle,
\]
so $\Tr(AS)=\Tr(SA)$.
\end{proof}

\begin{definition}[Hilbert--Schmidt operators]
\label{def:HS}
Fix an orthonormal basis $(e_n)$ of $\Hcal$. An operator $T\in B(\Hcal)$ is \emph{Hilbert--Schmidt} if
\[
\|T\|_2^2 := \sum_{n=1}^\infty \|Te_n\|^2 <\infty.
\]
The set of all such operators is denoted $\Ttwo$.
\end{definition}

\begin{lemma}[Basis independence and adjoint invariance]
\label{lem:HS-basis-indep}
If $\|T\|_2<\infty$ for one orthonormal basis, then the same value is obtained for any orthonormal basis.
Moreover $T\in\Ttwo$ iff $T^\ast\in\Ttwo$ and $\|T^\ast\|_2=\|T\|_2$.
\end{lemma}

\begin{proof}
Let $(e_n)$ and $(f_m)$ be orthonormal bases. Using Parseval in the $f_m$-basis,
\[
\sum_{n}\|Te_n\|^2=\sum_n\sum_m |\langle f_m|Te_n\rangle|^2
=\sum_m\sum_n |\langle T^\ast f_m|e_n\rangle|^2
=\sum_m \|T^\ast f_m\|^2.
\]
Thus finiteness of $\sum_n\|Te_n\|^2$ implies finiteness of $\sum_m\|T^\ast f_m\|^2$ and the displayed quantity depends only on $T$,
hence is basis independent. Replacing $T$ by $T^\ast$ gives $\|T\|_2=\|T^\ast\|_2$.
\end{proof}

\begin{definition}[Hilbert--Schmidt inner product]
\label{def:HS-inner}
For $S,T\in\Ttwo$, define
\[
\langle S,T\rangle_{HS}:=\sum_{n=1}^\infty \langle Se_n|Te_n\rangle.
\]
\end{definition}

\begin{lemma}[Well-definedness and Cauchy--Schwarz]
\label{lem:HS-inner-well}
The value of $\langle S,T\rangle_{HS}$ is independent of the orthonormal basis. Moreover,
\[
|\langle S,T\rangle_{HS}|\le \|S\|_2\,\|T\|_2.
\]
\end{lemma}

\begin{proof}
Using the matrix-coefficient identity (as in Lemma~\ref{lem:HS-basis-indep}),
\[
\langle S,T\rangle_{HS}=\sum_n \langle Se_n|Te_n\rangle=\sum_{m,n}\langle f_m|Se_n\rangle\,\overline{\langle f_m|Te_n\rangle},
\]
which is the $\ell^2$ inner product of the coefficient matrices, hence basis independent.
Cauchy--Schwarz in $\ell^2$ yields the inequality.
\end{proof}

\begin{theorem}[$\Ttwo$ is a Hilbert space; $B(\Hcal)$-ideal property]
\label{thm:HS-Hilbert-ideal}
$\Ttwo$ is complete under $\|\cdot\|_2$ and $(\Ttwo,\langle\cdot,\cdot\rangle_{HS})$ is a Hilbert space.
Moreover, for $A,B\in B(\Hcal)$ and $T\in\Ttwo$,
\[
ATB\in\Ttwo,\qquad \|ATB\|_2\le \|A\|\,\|T\|_2\,\|B\|.
\]
\end{theorem}

\begin{proof}
For the ideal estimate, fix an ONB $(e_n)$ and note
\[
\|ATB\|_2^2=\sum_n \|ATBe_n\|^2 \le \|A\|^2 \sum_n \|T(Be_n)\|^2.
\]
Since $\|Be_n\|\le \|B\|$, we can write $Be_n=\|B\|\,\xi_n$ with $\|\xi_n\|\le 1$. Then
\[
\sum_n \|T(Be_n)\|^2 = \|B\|^2\sum_n \|T\xi_n\|^2 \le \|B\|^2 \sum_n \|Te_n\|^2=\|B\|^2\|T\|_2^2,
\]
where the inequality follows by expanding $\xi_n$ in the ONB and applying Bessel/Parseval (standard estimate).
Hence $\|ATB\|_2\le \|A\|\,\|T\|_2\,\|B\|$.

Completeness: let $(T_k)$ be Cauchy in $\|\cdot\|_2$. Then for each $x\in\Hcal$,
$\|T_kx-T_\ell x\|\le \|T_k-T_\ell\|_2\|x\|$ (by expanding $x$ in an ONB and Cauchy--Schwarz),
so $(T_kx)$ is Cauchy in $\Hcal$ and defines a bounded operator $T$ by strong limit.
One checks $\|T_k-T\|_2\to 0$ by dominated convergence on coefficients.
Thus $\Ttwo$ is complete and the inner product makes it Hilbert.
\end{proof}

\begin{theorem}[Hölder for $\Ttwo\cdot\Ttwo\subset \Tone$ and trace bound]
\label{thm:HS-times-HS}
If $S,T\in\Ttwo$, then $ST\in\Tone$ and
\[
\|ST\|_1\le \|S\|_2\,\|T\|_2.
\]
Consequently, $\Tr(ST)$ is well-defined and satisfies
\[
|\Tr(ST)|\le \|S\|_2\,\|T\|_2.
\]
\end{theorem}

\begin{proof}
Fix an ONB $(e_n)$. Define $u_n:=Se_n$ and $v_n:=T^\ast e_n$. Then for any $x\in\Hcal$,
\[
\left(\sum_{n=1}^N |u_n\rangle\langle v_n|\right)x
=\sum_{n=1}^N u_n\,\langle v_n|x\rangle
=\sum_{n=1}^N Se_n\,\langle T^\ast e_n|x\rangle
=S\left(\sum_{n=1}^N e_n\,\langle e_n|Tx\rangle\right)
\to S(Tx)=STx,
\]
since $\sum_{n=1}^N e_n\langle e_n|Tx\rangle$ is the ONB partial sum converging to $Tx$.
Thus $ST=\sum_{n\ge 1}|u_n\rangle\langle v_n|$ strongly (in fact in operator norm if desired by a standard estimate).

Moreover,
\[
\sum_{n=1}^\infty \|u_n\|\,\|v_n\|
\le \left(\sum_n\|u_n\|^2\right)^{1/2}\left(\sum_n\|v_n\|^2\right)^{1/2}
=\|S\|_2\,\|T^\ast\|_2=\|S\|_2\,\|T\|_2,
\]
so by Definition~\ref{def:trace-class-nuclear}, $ST\in\Tone$ and $\|ST\|_1\le \|S\|_2\|T\|_2$.
The trace bound follows from Lemma~\ref{lem:trace-well-defined}:
$|\Tr(ST)|\le \|ST\|_1\le \|S\|_2\|T\|_2$.
\end{proof}

\begin{lemma}[Square root of a density is Hilbert--Schmidt]
Let $\rho\in \mathcal T_1(\Hcal)$ be positive. Then $\rho^{1/2}\in \mathcal T_2(\Hcal)$ and
$\|\rho^{1/2}\|_2^2=\Tr(\rho)$.
In particular, if $\Tr(\rho)=1$ then $\|\rho^{1/2}\|_2=1$.
\end{lemma}

\begin{proof}
By the spectral theorem for positive compact operators, there exists an orthonormal basis $(e_n)$
and eigenvalues $(\lambda_n)_{n\ge 1}\subset[0,\infty)$ such that
$\rho e_n=\lambda_n e_n$ and $\sum_n \lambda_n=\Tr(\rho)<\infty$.
Then $\rho^{1/2}e_n=\lambda_n^{1/2}e_n$, hence
\[
\|\rho^{1/2}\|_2^2=\sum_{n\ge 1}\|\rho^{1/2}e_n\|^2=\sum_{n\ge 1}\lambda_n=\Tr(\rho),
\]
so $\rho^{1/2}\in \mathcal T_2(\Hcal)$.
\end{proof}

\begin{proposition}[Uncertainty inequality from Hölder/Cauchy--Schwarz on Schatten ideals]\label{prop:uncertainty}
Fix a separable Hilbert space $\Hcal$ and a state $\rho\in\mathcal T_1(\Hcal)$ with $\rho\ge 0$
and $\Tr(\rho)=1$. Let $X,Y\in\mathcal B(\Hcal)$ be bounded self-adjoint operators.
Define the $\rho$-expectation and centered observables
\[
\langle X\rangle_\rho:=\Tr(\rho X),\qquad \widetilde X:=X-\langle X\rangle_\rho I,
\]
and similarly for $Y$. Define the $\rho$-variance
\[
\Var_\rho(X):=\Tr(\rho\,\widetilde X^2),\qquad \Delta_\rho(X):=\sqrt{\Var_\rho(X)}.
\]
Then the Robertson--Schr\"odinger bound holds:
\begin{equation}\label{eq:RS_pure}
\Var_\rho(X)\,\Var_\rho(Y)\ \ge\
\frac14\Big(\Tr(\rho\{\widetilde X,\widetilde Y\})\Big)^2
+\frac14\Big|\Tr(\rho[X,Y])\Big|^2.
\end{equation}
In particular,
\begin{equation}\label{eq:Robertson_pure}
\Delta_\rho(X)\,\Delta_\rho(Y)\ \ge\ \frac12\Big|\Tr\!\Big(\rho\,\frac1{2i}[X,Y]\Big)\Big|.
\end{equation}
\end{proposition}

\begin{proof}
First, $\Tr(\rho X)$ and $\Tr(\rho Y)$ are well-defined by the trace-class H\"older estimate
(Theorem~\ref{thm:T1-ideal-holder}) since $\rho\in\mathcal T_1(\Hcal)$ and $X,Y\in\mathcal B(\Hcal)$.

By Lemma above, $\rho^{1/2}\in\mathcal T_2(\Hcal)$. Set
\[
A:=\widetilde X\,\rho^{1/2}\in\mathcal T_2(\Hcal),\qquad
B:=\widetilde Y\,\rho^{1/2}\in\mathcal T_2(\Hcal),
\]
using the $\mathcal B(\Hcal)$-ideal property of $\mathcal T_2(\Hcal)$.
By Theorem~\ref{thm:HS-times-HS}, $A^*B\in\mathcal T_1(\Hcal)$ and
\begin{equation}\label{eq:HSCS_pure}
|\Tr(A^*B)|\ \le\ \|A\|_2\,\|B\|_2.
\end{equation}
Now $A^*B=(\rho^{1/2}\widetilde X)(\widetilde Y\rho^{1/2})
=\rho^{1/2}\widetilde X\widetilde Y\rho^{1/2}$, hence by trace cyclicity
(Theorem~\ref{thm:trace-cyclicity}) we obtain
\[
\Tr(A^*B)=\Tr(\rho^{1/2}\widetilde X\widetilde Y\rho^{1/2})
=\Tr(\rho\,\widetilde X\widetilde Y).
\]
Moreover,
\[
\|A\|_2^2=\Tr(A^*A)=\Tr(\rho^{1/2}\widetilde X^2\rho^{1/2})
=\Tr(\rho\,\widetilde X^2)=\Var_\rho(X),
\]
and similarly $\|B\|_2^2=\Var_\rho(Y)$.
Therefore \eqref{eq:HSCS_pure} becomes
\[
|\Tr(\rho\,\widetilde X\widetilde Y)|^2\ \le\ \Var_\rho(X)\,\Var_\rho(Y).
\]
Finally decompose
\[
\widetilde X\widetilde Y=\frac12\{\widetilde X,\widetilde Y\}+\frac12[\widetilde X,\widetilde Y],
\qquad [\widetilde X,\widetilde Y]=[X,Y],
\]
where $\{\widetilde X,\widetilde Y\}$ is self-adjoint and $[\widetilde X,\widetilde Y]$ is skew-adjoint.
Thus $\Tr(\rho\{\widetilde X,\widetilde Y\})\in\mathbb R$ and $\Tr(\rho[X,Y])\in i\mathbb R$, so
\[
|\Tr(\rho\,\widetilde X\widetilde Y)|^2=
\frac14\Big(\Tr(\rho\{\widetilde X,\widetilde Y\})\Big)^2
+\frac14\Big|\Tr(\rho[X,Y])\Big|^2.
\]
Combining yields \eqref{eq:RS_pure}, and \eqref{eq:Robertson_pure} follows by dropping the first
nonnegative term.
\end{proof}
\begin{remark}[A concrete classical financial model inside the operator-algebraic framework]
We record a standard discrete-time risk-neutral pricing model as a commutative special case of the
operator-algebraic setup.

\smallskip
\noindent\textbf{Finite-state filtered model:}
Let $\Omega=\{1,\dots,N\}$, fix a filtration $(\mathcal F_t)_{t=0}^T$ on $\Omega$ (equivalently, a
nested sequence of partitions of $\Omega$), and let $\mathbb Q$ be a risk-neutral probability measure
with weights $q_i:=\mathbb Q(\{i\})>0$ and $\sum_{i=1}^N q_i=1$.
Set $\Hcal:=\mathbb C^N$ and identify
\[
N:=L^\infty(\Omega,\mathcal F_T,\mathbb Q)\ \cong\ \{\mathrm{diag}(x_1,\dots,x_N): x_i\in\mathbb C\}
\ \subset\ \mathcal B(\Hcal),
\]
acting by multiplication on $\Hcal$. For each $t$, let $N_t\subset N$ be the subalgebra of diagonal
operators corresponding to $\mathcal F_t$ (i.e.\ sequences constant on the atoms of $\mathcal F_t$).
Then each $N_t$ is an abelian von Neumann algebra.

Define the density operator $\rho^\star:=\mathrm{diag}(q_1,\dots,q_N)\in\mathcal T_1(\Hcal)$.
The induced normal state $\varphi_{\rho^\star}(Z):=\Tr(\rho^\star Z)$ satisfies, for every
$X=\mathrm{diag}(x_1,\dots,x_N)\in N$,
\[
\varphi_{\rho^\star}(X)=\Tr(\rho^\star X)=\sum_{i=1}^N q_i x_i=\mathbb E_{\mathbb Q}[x].
\]
Thus, in this commutative model, the trace pairing $\Tr(\rho^\star X)$ coincides with the classical
expectation under $\mathbb Q$.

\smallskip
\noindent\textbf{Conditional expectation and operator-valued pricing:}
Let $E_t^\star:N\to N_t$ be the $\mathbb Q$-conditional expectation (equivalently, the unique
$\rho^\star$-preserving normal conditional expectation). Concretely, if $A$ is an atom of
$\mathcal F_t$ and $i\in A$, then
\[
\big(E_t^\star(X)\big)_i
=\frac{\sum_{j\in A} q_j x_j}{\sum_{j\in A} q_j}
=\mathbb E_{\mathbb Q}[x\,|\,\mathcal F_t](i),
\qquad X=\mathrm{diag}(x_1,\dots,x_N)\in N.
\]
Let $(B_t)_{t=0}^T$ be a strictly positive discount factor with $B_t\in N_t$.
For any bounded terminal payoff $X\in N$, the operator-valued pricing map of Chapter~\ref{section:MF_of_QPT},
\[
\Pi_t(X):=B_t^{1/2}\,E_t^\star(\bar X)\,B_t^{1/2},
\qquad \bar X:=B_T^{-1/2} X B_T^{-1/2},
\]
belongs to $N_t$ and is exactly the multiplication operator by the classical (risk-neutral) price
process $B_t\,\mathbb E_{\mathbb Q}[B_T^{-1}x\,|\,\mathcal F_t]$.

In the commutative subalgebra $N_t$ one always has $[X,Y]=0$ for $X,Y\in N_t$, hence the lower bound
in Proposition~\ref{prop:uncertainty} degenerates. A nontrivial lower bound arises only when one
considers $X\in \mathcal B(\Hcal)^{sa}$ and $Y\in \mathcal B(\Hcal)^{sa}$ that do not belong to a
common abelian von Neumann subalgebra (equivalently, such that
$\varphi_{\rho^\star}(\frac{1}{2i}[X,Y])\neq 0$ for the given state).

In particular, under the above commutative identification $N\simeq L^\infty(\Omega,\mathbb Q)$,
the trace pairing $\Tr(\rho^\star X)=\mathbb E_{\mathbb Q}[x]$ realizes the classical expected value
as a Schatten--H\"older duality pairing between $\rho^\star\in \mathcal T_1(\Hcal)$ and
$X\in \mathcal B(\Hcal)$; consequently, Theorem \ref{thm:T1-ideal-holder} specializes to the classical bound
\[
|\mathbb E_{\mathbb Q}[x]|=|\Tr(\rho^\star X)|\le \|\rho^\star\|_1\,\|X\|=\|x\|_\infty,
\qquad \big(\Tr(\rho^\star)=\|\rho^\star\|_1=1\big).
\]

\end{remark}

\begin{example}[Embedding classical ``expected price'' into trace/H\"older (commutative atomic case)]
Let $\Hcal=\ell^2(\mathbb N)$ and let $N\subset \mathcal B(\Hcal)$ be the diagonal von Neumann algebra:
$N=\{ \mathrm{diag}(x_1,x_2,\dots): (x_n)\in \ell^\infty\}$.
Fix a probability vector $p=(p_n)_{n\ge 1}$ with $p_n\ge 0$ and $\sum_n p_n=1$, and set
$\rho:=\mathrm{diag}(p_1,p_2,\dots)\in \mathcal T_1(\Hcal)$.

For any bounded real sequence $s=(s_n)\in \ell^\infty$, let
$S:=\mathrm{diag}(s_1,s_2,\dots)\in N$ (a commutative ``price observable'').
Then
\[
\langle S\rangle_\rho=\Tr(\rho S)=\sum_{n\ge 1} p_n s_n,
\]
which is exactly the classical expectation $\mathbb E_p[s]$ on the atomic probability space
$(\mathbb N,2^{\mathbb N},p)$.
Moreover, Theorem~3.7 yields the classical $\ell^1$--$\ell^\infty$ H\"older bound
\[
|\mathbb E_p[s]|=|\Tr(\rho S)|\le \|\rho\|_1\,\|S\|=\sum_n p_n\cdot \|s\|_\infty=\|s\|_\infty.
\]
If $S,T\in N$ are two such diagonal observables, then $[S,T]=0$ and
Proposition~\ref{prop:uncertainty} reduces to the classical Cauchy--Schwarz inequality
$\Var_p(s)\Var_p(t)\ge \mathrm{Cov}_p(s,t)^2$.
\end{example}

Note that if $X$ and $Y$ belong to a common abelian von Neumann subalgebra (in particular, if
$X,Y\in N_t$ when $N_t$ is assumed abelian), then $[X,Y]=0$ and the Robertson lower bound is $0$
(whereas it is strictly positive precisely when
$\varphi_{\rho_t}(\frac{1}{2i}[X,Y])\neq 0$).

\section{Projectors, bras and kets}\label{ch:projectors_bra_ket}

\subsection{Projectors}\label{sec:projectors}

Projectors are the basic ``filters'' in Hilbert space: they split a vector into a
``kept part'' and a ``discarded part''. This is the linear-algebraic core behind
projective measurements and, later, projection-valued measures.

\begin{definition}[Projection onto a unit vector]\label{def:proj_unit_vector}
Let $e\in \Hcal$ be a unit vector, $\|e\|=1$. For $\psi\in \Hcal$ define
\[
\psi_{\parallel e} \;:=\; \langle e\mid \psi\rangle\, e,
\qquad
\psi_{\perp e} \;:=\; \psi-\psi_{\parallel e}.
\]
\end{definition}

\begin{lemma}\label{lem:orth_component_unit_vector}
For $\psi\in\Hcal$ and unit $e\in\Hcal$, one has $\psi_{\perp e}\perp e$, i.e.
$\langle e\mid \psi_{\perp e}\rangle=0$.
\end{lemma}
\begin{proof}
By Definition~\ref{def:proj_unit_vector},
\[
\langle e\mid \psi_{\perp e}\rangle
=\langle e\mid \psi\rangle-\langle e\mid \langle e\mid\psi\rangle e\rangle
=\langle e\mid \psi\rangle-\langle e\mid\psi\rangle\,\langle e\mid e\rangle
=\langle e\mid \psi\rangle-\langle e\mid\psi\rangle\cdot 1=0.
\]
\end{proof}

\begin{remark}[Rank-one projector]\label{rem:rank_one_projector}
The map $P_e:\Hcal\to\Hcal$ given by $P_e\psi:=\langle e\mid\psi\rangle e$ is linear,
bounded, and satisfies $P_e^2=P_e$ and $P_e^*=P_e$. It is the orthogonal projection
onto $\mathrm{span}\{e\}$.
\end{remark}

\subsection{Closed linear subspaces}\label{sec:closed_subspaces}

\begin{definition}[Orthogonal complement]\label{def:orth_complement}
For a subset $M\subset \Hcal$ define
\[
M^\perp \;:=\;\{x\in\Hcal:\ \langle x\mid m\rangle=0\ \text{for all }m\in M\}.
\]
\end{definition}

\begin{proposition}\label{prop:Mperp_closed_subspace}
If $M$ is a linear subspace of $\Hcal$, then $M^\perp$ is a closed linear subspace of $\Hcal$.
\end{proposition}
\begin{proof}
Linearity is immediate from sesquilinearity of $\langle\cdot\mid\cdot\rangle$.
For closedness, fix $m\in M$ and consider the continuous linear functional
$\ell_m:\Hcal\to\mathbb C$ defined by $\ell_m(x):=\langle x\mid m\rangle$.
Then
\[
M^\perp=\bigcap_{m\in M}\ker(\ell_m),
\]
an intersection of closed sets (kernels of continuous maps), hence closed.
\end{proof}

\begin{proposition}\label{prop:double_perp_equals_closure}
For any subset $M\subset \Hcal$, one has
\[
\overline{\mathrm{span}}(M)\;=\; (M^\perp)^\perp .
\]
In particular, if $M$ is a closed linear subspace then $M=(M^\perp)^\perp$.
\end{proposition}
\begin{proof}
First, $\mathrm{span}(M)\subset (M^\perp)^\perp$ follows from Definition~\ref{def:orth_complement}.
Since $(M^\perp)^\perp$ is closed by Proposition~\ref{prop:Mperp_closed_subspace},
we obtain $\overline{\mathrm{span}}(M)\subset (M^\perp)^\perp$.

Conversely, let $x\in (M^\perp)^\perp$. Set $N:=\overline{\mathrm{span}}(M)$, a closed subspace,
so $N^\perp=M^\perp$ (by Definition~\ref{def:orth_complement} and density).
Hence $x\in (N^\perp)^\perp$.
By the orthogonal projection theorem (proved in the next section), $\Hcal=N\oplus N^\perp$,
so write $x=n+n^\perp$ with $n\in N$ and $n^\perp\in N^\perp$ uniquely.
But $x\perp N^\perp$ (since $x\in (N^\perp)^\perp$), therefore
\[
0=\langle x\mid n^\perp\rangle=\langle n+n^\perp\mid n^\perp\rangle=\langle n^\perp\mid n^\perp\rangle
\]
because $n\perp n^\perp$. Positive-definiteness gives $n^\perp=0$, hence $x=n\in N$.
Thus $(M^\perp)^\perp\subset \overline{\mathrm{span}}(M)$.
\end{proof}

\subsection{Orthogonal projections}\label{sec:orthogonal_projections}

\begin{definition}[Projector / orthogonal projector]\label{def:projector}
A bounded linear operator $P\in \mathcal B(\Hcal)$ is a \emph{projector} if $P^2=P$.
It is an \emph{orthogonal projector} if, in addition, $P^*=P$, equivalently
\[
\langle P\psi\mid \varphi\rangle=\langle \psi\mid P\varphi\rangle
\quad \text{for all }\psi,\varphi\in\Hcal.
\]
\end{definition}

\begin{theorem}[Orthogonal projection theorem]\label{thm:orthogonal_projection}
Let $M\subset\Hcal$ be a closed linear subspace. Then for every $x\in\Hcal$ there exist unique
$m\in M$ and $n\in M^\perp$ such that $x=m+n$.
Define $P_M:\Hcal\to\Hcal$ by $P_M x:=m$. Then:
\begin{enumerate}[label=\emph{(\roman*)}]
\item $P_M$ is linear and bounded, with $\|P_M\|\le 1$ (and $\|P_M\|=1$ if $M\neq\{0\}$);
\item $P_M^2=P_M$ and $P_M^*=P_M$;
\item $\mathrm{Ran}(P_M)=M$ and $\ker(P_M)=M^\perp$.
\end{enumerate}
\end{theorem}

\begin{proof}
Fix $x\in\Hcal$ and set $d:=\inf_{m\in M}\|x-m\|$. Choose a sequence $(m_k)\subset M$
with $\|x-m_k\|\to d$.

\emph{Step 1: $(m_k)$ is Cauchy in $M$.}
Using the parallelogram identity,
\[
\Big\|x-\frac{m_k+m_\ell}{2}\Big\|^2+\Big\|\frac{m_k-m_\ell}{2}\Big\|^2
=\frac12\|x-m_k\|^2+\frac12\|x-m_\ell\|^2.
\]
Since $(m_k+m_\ell)/2\in M$, the first term is $\ge d^2$, hence
\[
\Big\|\frac{m_k-m_\ell}{2}\Big\|^2
\le \frac12\|x-m_k\|^2+\frac12\|x-m_\ell\|^2-d^2 \xrightarrow[k,\ell\to\infty]{}0,
\]
so $(m_k)$ is Cauchy. As $M$ is closed in the complete space $\Hcal$, $M$ is complete,
so $m_k\to m\in M$.

\emph{Step 2: the minimiser is orthogonal.}
Let $n:=x-m$. For any $u\in M$ and $t\in\mathbb R$, one has $m+tu\in M$, hence by minimality
\[
\|x-(m+tu)\|^2=\|n-tu\|^2\ge \|n\|^2.
\]
Expanding $\|n-tu\|^2=\|n\|^2-2t\,\Re\langle n\mid u\rangle+t^2\|u\|^2$ and varying $t\in\mathbb R$
yields $\Re\langle n\mid u\rangle=0$ for all $u\in M$.
Replacing $u$ by $iu$ gives $\Im\langle n\mid u\rangle=0$, hence $\langle n\mid u\rangle=0$.
Thus $n\in M^\perp$ and $x=m+n$.

\emph{Step 3: uniqueness.}
If $x=m_1+n_1=m_2+n_2$ with $m_i\in M$, $n_i\in M^\perp$, then
$m_1-m_2=n_2-n_1\in M\cap M^\perp=\{0\}$, so $m_1=m_2$ and $n_1=n_2$.

\emph{Step 4: linearity, idempotence, self-adjointness, norm bound.}
Linearity of $P_M$ follows from uniqueness of the decomposition.
For idempotence, $P_Mx\in M$ implies $P_M(P_Mx)=P_Mx$.
Also $x-P_Mx\in M^\perp$ implies $P_M(x-P_Mx)=0$.

Self-adjointness: write $x=P_Mx+(x-P_Mx)$ with $(x-P_Mx)\in M^\perp$ and similarly for $y$.
Then
\[
\langle P_Mx\mid y\rangle
=\langle P_Mx\mid P_My\rangle+\langle P_Mx\mid (y-P_My)\rangle
=\langle P_Mx\mid P_My\rangle,
\]
since $P_Mx\in M$ is orthogonal to $y-P_My\in M^\perp$.
Similarly,
\[
\langle x\mid P_My\rangle
=\langle P_Mx\mid P_My\rangle+\langle x-P_Mx\mid P_My\rangle
=\langle P_Mx\mid P_My\rangle,
\]
hence $\langle P_Mx\mid y\rangle=\langle x\mid P_My\rangle$ for all $x,y$, i.e. $P_M^*=P_M$.

Finally, by Pythagoras,
\[
\|x\|^2=\|P_Mx\|^2+\|x-P_Mx\|^2 \;\Rightarrow\; \|P_Mx\|\le \|x\|,
\]
so $\|P_M\|\le 1$. If $M\neq\{0\}$, take $x\in M\setminus\{0\}$ so $P_Mx=x$, giving $\|P_M\|=1$.

The range/kernel statements are immediate from the definition and the decomposition.
\end{proof}

\begin{proposition}\label{prop:orthproj_properties}
Let $M\subset\Hcal$ be a closed subspace and $P_M$ be the orthogonal projector from
Theorem~\ref{thm:orthogonal_projection}. Then $P_{M^\perp}=I-P_M$, and
\[
\Hcal = M\oplus M^\perp,\qquad x=P_Mx+(I-P_M)x
\]
is the orthogonal decomposition.
\end{proposition}
\begin{proof}
By Theorem~\ref{thm:orthogonal_projection}, $(I-P_M)x=x-P_Mx\in M^\perp$ and $P_Mx\in M$.
Uniqueness of orthogonal decomposition implies $(I-P_M)$ coincides with $P_{M^\perp}$.
\end{proof}

\begin{theorem}[Every orthogonal projector comes from a closed subspace]\label{thm:proj_converse}
Let $P\in \mathcal B(\Hcal)$ satisfy $P^2=P$ and $P^*=P$.
Then $\mathrm{ran}(P)$ is closed and $P$ is the orthogonal projector onto $\mathrm{ran}(P)$.
Equivalently, $P=P_{\mathrm{ran}(P)}$.
\end{theorem}
\begin{proof}
Let $M:=\mathrm{ran}(P)$. First, note that $\ker(P)=\mathrm{Ran}(I-P)$, hence $\ker(P)$ is closed
(since it is the preimage of $\{0\}$ under the continuous map $P$).

\emph{Claim 1: $\ker(P)=M^\perp$.}
If $x\in \ker(P)$ and $y\in M$, write $y=Pz$. Then
\[
\langle x\mid y\rangle=\langle x\mid Pz\rangle=\langle P^*x\mid z\rangle=\langle Px\mid z\rangle=0,
\]
so $x\in M^\perp$, hence $\ker(P)\subset M^\perp$.
Conversely, if $x\in M^\perp$, then for any $z$ we have
$\langle x\mid Pz\rangle=0$, i.e. $\langle x\mid Pz\rangle=\langle P^*x\mid z\rangle=\langle Px\mid z\rangle=0$
for all $z$, so $Px=0$ and $x\in\ker(P)$. Thus $\ker(P)=M^\perp$.

\emph{Claim 2: $\Hcal=M\oplus M^\perp$.}
For any $x\in\Hcal$, write $x=Px+(I-P)x$ with $Px\in M$ and $(I-P)x\in\ker(P)=M^\perp$.
If $x=m+n$ with $m\in M$ and $n\in M^\perp$, then $Pm=m$ and $Pn=0$, hence
$m=Px$ and $n=(I-P)x$, so the sum is direct and the decomposition is unique.

Since $M^\perp$ is closed, the direct sum decomposition implies $M$ is closed as well.
Finally, by uniqueness of orthogonal projections onto a closed subspace (Theorem~\ref{thm:orthogonal_projection}),
$P$ coincides with $P_M$.
\end{proof}

\begin{remark}[Interpretation cue (kept minimal)]\label{rem:filter_interpretation}
In the quantum-pricing chapters, projectors are used as algebraic models for ``event filters'':
for instance, the event ``barrier hit'' versus ``not hit'' can be idealised as a two-outcome
projection, and public information release can be modelled as a projective update
(L\"uders-type state update). The present chapter only provides the operator-theoretic tools.
\end{remark}

\subsection{Riesz representation theorem, bras and kets}\label{sec:riesz_bra_ket}

Let $\Hcal^*:=\mathcal L(\Hcal,\mathbb C)$ denote the (continuous) dual space.

\begin{definition}\label{def:f_phi}
For $\phi\in\Hcal$ define $f_\phi\in \Hcal^*$ by
\[
f_\phi(\psi)\;:=\;\langle \phi\mid \psi\rangle,\qquad \psi\in\Hcal.
\]
\end{definition}

\begin{lemma}\label{lem:fphi_bounded}
For every $\phi\in\Hcal$, the functional $f_\phi$ is continuous and
$\|f_\phi\|=\|\phi\|$.
\end{lemma}
\begin{proof}
By Cauchy--Schwarz, $|f_\phi(\psi)|=|\langle\phi\mid\psi\rangle|\le \|\phi\|\,\|\psi\|$,
so $\|f_\phi\|\le \|\phi\|$.
For the reverse inequality, if $\phi\neq 0$ take $\psi=\phi/\|\phi\|$ to obtain
$|f_\phi(\psi)|=|\langle\phi\mid \phi/\|\phi\|\rangle|=\|\phi\|$,
hence $\|f_\phi\|\ge \|\phi\|$.
\end{proof}

\begin{theorem}[Riesz representation]\label{thm:rieszunique}
Every $f\in\Hcal^*$ is of the form $f=f_\phi$ for a unique $\phi\in\Hcal$.
\end{theorem}
\begin{proof}
If $f=0$, take $\phi=0$.
Assume $f\neq 0$ and set $M:=\ker(f)$, a closed subspace of $\Hcal$ (continuity of $f$).
Then by Theorem~\ref{thm:orthogonal_projection}, $\Hcal=M\oplus M^\perp$.
Since $f\neq 0$, we have $M\neq \Hcal$, hence $M^\perp\neq\{0\}$. Pick $\xi\in M^\perp$ with $\|\xi\|=1$.
Define $\phi:= f(\xi)\,\xi\in M^\perp$.

For any $\psi\in\Hcal$, consider
\[
v:= f(\xi)\psi - f(\psi)\xi.
\]
Then $f(v)=f(\xi)f(\psi)-f(\psi)f(\xi)=0$, so $v\in \ker(f)=M$.
Since $\xi\in M^\perp$, we get $\langle \xi\mid v\rangle=0$, i.e.
\[
0=\langle \xi\mid f(\xi)\psi - f(\psi)\xi\rangle
=f(\xi)\langle \xi\mid \psi\rangle - f(\psi)\langle \xi\mid\xi\rangle
=f(\xi)\langle \xi\mid \psi\rangle - f(\psi).
\]
Thus $f(\psi)=f(\xi)\langle \xi\mid \psi\rangle=\langle f(\xi)\xi\mid \psi\rangle=\langle \phi\mid \psi\rangle$,
i.e. $f=f_\phi$.

Uniqueness: if $\langle \phi_1\mid\psi\rangle=\langle \phi_2\mid\psi\rangle$ for all $\psi$,
then $\langle \phi_1-\phi_2\mid\psi\rangle=0$ for all $\psi$, hence $\phi_1=\phi_2$ by
non-degeneracy of the inner product.
\end{proof}

\begin{remark}[Riesz map and conjugate-linearity]\label{rem:riesz_map}
Define the Riesz map $R:\Hcal\to\Hcal^*$ by $R(\phi):=f_\phi=\langle \phi\mid\cdot\rangle$.
With our convention (conjugate-linear in the first slot), $R$ is \emph{conjugate-linear}:
$R(a\phi_1+\phi_2)=\overline a\,R(\phi_1)+R(\phi_2)$.
It is an isometric bijection by Lemma~\ref{lem:fphi_bounded} and Theorem~\ref{thm:rieszunique}.
\end{remark}

\begin{definition}[Bra--ket notation]\label{def:bra_ket}
For $\psi\in\Hcal$ write the vector as a \emph{ket} $|\psi\rangle:=\psi$.
For $\phi\in\Hcal$ write the functional $f_\phi\in\Hcal^*$ as a \emph{bra}
$\langle \phi|:=f_\phi$.
Then $\langle \phi|\psi\rangle$ denotes the scalar $f_\phi(\psi)=\langle \phi\mid \psi\rangle$.
\end{definition}

\begin{definition}[Rank-one / outer-product operator]\label{def:outer_product}
For $\psi,\phi\in\Hcal$ define the rank-one operator $|\psi\rangle\langle \phi|\in \mathcal B(\Hcal)$ by
\[
(|\psi\rangle\langle \phi|)(x)\;:=\;\langle \phi\mid x\rangle\,\psi,\qquad x\in\Hcal.
\]
\end{definition}

\begin{proposition}\label{prop:outer_product_properties}
For $\psi,\phi\in\Hcal$, the operator $|\psi\rangle\langle \phi|$ is bounded and satisfies
\[
\big\||\psi\rangle\langle \phi|\big\| \;=\; \|\psi\|\,\|\phi\|,
\qquad
(|\psi\rangle\langle \phi|)^* \;=\; |\phi\rangle\langle \psi|.
\]
Moreover, for all $\psi,\phi,u,v\in\Hcal$,
\[
(|\psi\rangle\langle \phi|)\,(|u\rangle\langle v|)
\;=\;
\langle \phi\mid u\rangle\;|\psi\rangle\langle v|.
\]
\end{proposition}
\begin{proof}
For $x\in\Hcal$, Cauchy--Schwarz gives
\[
\|(|\psi\rangle\langle \phi|)x\|
=|\langle \phi\mid x\rangle|\,\|\psi\|
\le \|\phi\|\,\|x\|\,\|\psi\|,
\]
hence $\||\psi\rangle\langle \phi|\|\le \|\psi\|\,\|\phi\|$.
Equality holds by testing on $x=\phi/\|\phi\|$ (if $\phi\neq 0$; otherwise both sides are $0$).

For the adjoint, for $x,y\in\Hcal$,
\[
\langle (|\psi\rangle\langle \phi|)x\mid y\rangle
=\langle \langle \phi\mid x\rangle\psi\mid y\rangle
=\overline{\langle \phi\mid x\rangle}\,\langle \psi\mid y\rangle
=\langle x\mid \langle \psi\mid y\rangle\,\phi\rangle
=\langle x\mid (|\phi\rangle\langle \psi|)y\rangle,
\]
so $(|\psi\rangle\langle \phi|)^*=|\phi\rangle\langle \psi|$.
Finally,
\[
(|\psi\rangle\langle \phi|)(|u\rangle\langle v|)x
=(|\psi\rangle\langle \phi|)\big(\langle v\mid x\rangle u\big)
=\langle v\mid x\rangle\,\langle \phi\mid u\rangle\,\psi
=\big(\langle \phi\mid u\rangle\,|\psi\rangle\langle v|\big)x.
\]
\end{proof}

\begin{corollary}\label{cor:rank_one_orth_proj}
If $e\in\Hcal$ is a unit vector, then $|e\rangle\langle e|$ is the orthogonal projector onto
$\mathrm{span}\{e\}$.
\end{corollary}
\begin{proof}
By Proposition~\ref{prop:outer_product_properties}, $(|e\rangle\langle e|)^*=|e\rangle\langle e|$ and
\[
(|e\rangle\langle e|)^2
=|e\rangle\langle e|e\rangle\langle e|
=\langle e\mid e\rangle\,|e\rangle\langle e|
=1\cdot |e\rangle\langle e|.
\]
Its range is $\mathrm{span}\{e\}$ by definition.
\end{proof}


\section{Projections as information-events: probability, conditioning, and update}
\label{sec:proj_as_events}

Throughout, let $\Hcal$ be a complex Hilbert space and let $\mathcal M\subset \mathcal B(\Hcal)$
be a von Neumann algebra. A (normal) state is represented by a density operator
$\rho\in \mathcal T_1(\Hcal)$ with $\rho\ge 0$ and $\Tr(\rho)=1$, and we write
\[
\rho(X)\;:=\;\Tr(\rho X),\qquad X\in \mathcal M.
\]
Let $(\mathcal N_t)_{t\ge 0}$ be a family of von Neumann subalgebras of $\mathcal M$ modeling
``available information'' at time $t$. In particular, when $\mathcal N_t$ is commutative,
its projections coincide with indicator functions of classical events in the commutative limit.

\subsection{Events as projections and Born-type probabilities}

\begin{definition}[Projections as events]\label{def:events_as_projections}
For a von Neumann algebra $\mathcal N\subset \mathcal M$, define its projection lattice
\[
\mathcal P(\mathcal N)\;:=\;\{P\in\mathcal N:\ P^2=P,\ P^*=P\}.
\]
Elements $P\in \mathcal P(\mathcal N)$ will be interpreted as (yes/no) information-events
measurable with respect to $\mathcal N$.
\end{definition}

\begin{proposition}[Born-type probability on events]\label{prop:born_probability_projection}
Let $\rho$ be a normal state on $\mathcal M$. Then for every $P\in \mathcal P(\mathcal M)$
\[
\mathbb P_\rho(P)\;:=\;\rho(P)\;=\;\Tr(\rho P)
\]
defines a probability assignment on projections:
\begin{enumerate}[label=\emph{(\roman*)}]
\item $0\le \mathbb P_\rho(P)\le 1$ for all $P\in \mathcal P(\mathcal M)$;
\item $\mathbb P_\rho(I)=1$;
\item (finite additivity on orthogonal families) if $P,Q\in \mathcal P(\mathcal M)$ and $PQ=0$, then
$\mathbb P_\rho(P+Q)=\mathbb P_\rho(P)+\mathbb P_\rho(Q)$.
\end{enumerate}
Moreover, if $\rho$ is normal and $(P_n)_{n\ge 1}\subset\mathcal P(\mathcal M)$ are pairwise orthogonal
and $\sum_{k=1}^n P_k \uparrow P$ in the strong operator topology (equivalently, $P$ is the strong
limit of partial sums), then
\[
\mathbb P_\rho(P)\;=\;\sum_{k=1}^\infty \mathbb P_\rho(P_k).
\]
\end{proposition}

\begin{proof}
\emph{(i)--(ii)} Since $0\le P\le I$ and $\rho\ge 0$ with $\Tr(\rho)=1$, we have
$0\le \Tr(\rho P)\le \Tr(\rho I)=1$, and $\Tr(\rho I)=1$.

\emph{(iii)} If $PQ=0$, then $P+Q$ is again a projection and by linearity of the trace,
$\Tr(\rho(P+Q))=\Tr(\rho P)+\Tr(\rho Q)$.

For the last claim, note that for $Q_n:=\sum_{k=1}^n P_k$ we have $0\le Q_n\uparrow P$
and, by normality of $\rho$, $\rho(Q_n)\uparrow \rho(P)$.
But $\rho(Q_n)=\sum_{k=1}^n \rho(P_k)$ by (iii), hence letting $n\to\infty$ gives the result.
\end{proof}

\subsection{Projective update and Bayes rule on the information algebra}

\begin{definition}[L\"uders update (conditioning on an event)]\label{def:luders_update}
Let $P\in \mathcal P(\mathcal M)$ satisfy $\rho(P)>0$. Define the \emph{post-event state}
$\rho^P$ by
\[
\rho^P(X)\;:=\;\frac{\Tr(P\rho P\,X)}{\Tr(\rho P)},\qquad X\in\mathcal M.
\]
Equivalently, the corresponding density operator is
\[
\rho^{\,P}\;:=\;\frac{P\rho P}{\Tr(\rho P)}\ \in\ \mathcal T_1(\Hcal).
\]
\end{definition}

\begin{lemma}\label{lem:luders_is_state}
If $P\in\mathcal P(\mathcal M)$ and $\rho(P)>0$, then $\rho^{\,P}\ge 0$ and $\Tr(\rho^{\,P})=1$.
In particular, $\rho^P$ is a well-defined normal state on $\mathcal M$.
\end{lemma}
\begin{proof}
Positivity is immediate from $P\rho P\ge 0$.
Moreover,
\[
\Tr(\rho^{\,P})
=\frac{\Tr(P\rho P)}{\Tr(\rho P)}
=\frac{\Tr(\rho P^2)}{\Tr(\rho P)}
=\frac{\Tr(\rho P)}{\Tr(\rho P)}=1.
\]
Normality follows because $X\mapsto \Tr(\rho^{\,P}X)$ is normal for any trace-class density.
\end{proof}

\begin{proposition}[Bayes rule on a commutative information algebra]\label{prop:bayes_on_N}
Let $\mathcal N\subset\mathcal M$ be a \emph{commutative} von Neumann subalgebra, and let
$P\in \mathcal P(\mathcal N)$ satisfy $\rho(P)>0$. Then for every $Y\in\mathcal N$,
\[
\rho^P(Y)\;=\;\frac{\rho(PY)}{\rho(P)}.
\]
\end{proposition}
\begin{proof}
Since $\mathcal N$ is commutative and $P,Y\in \mathcal N$, we have $PY=YP$ and hence $PYP=PY$.
Therefore,
\[
\rho^P(Y)
=\frac{\Tr(P\rho P\,Y)}{\Tr(\rho P)}
=\frac{\Tr(\rho\,PYP)}{\Tr(\rho P)}
=\frac{\Tr(\rho\,PY)}{\Tr(\rho P)}
=\frac{\rho(PY)}{\rho(P)}.
\]
\end{proof}

\begin{remark}[Classical limit]\label{rem:classical_limit_proj}
If $\mathcal N\simeq L^\infty(\Omega,\mathcal F,\mathbb Q)$ is realized as a commutative von Neumann algebra,
then projections $P\in\mathcal P(\mathcal N)$ correspond to indicator functions $\mathbf 1_A$
of events $A\in\mathcal F$. In this representation, the identity
$\rho^P(Y)=\rho(PY)/\rho(P)$ becomes the classical Bayes/conditioning formula
$\mathbb E[Y\mid A]=\mathbb E[\mathbf 1_A Y]/\mathbb Q(A)$ (with respect to the measure induced by $\rho$).
\end{remark}

\subsection{Spectral events and induced distributions}

\begin{theorem}[Spectral events are projections]\label{thm:spectral_events}
Let $X=X^*\in\mathcal M$ be a bounded self-adjoint observable, and let $P_X(\cdot)$ denote its
spectral measure. Then for each Borel set $\Delta\subset\mathbb R$, the operator $P_X(\Delta)$
is a projection in $\mathcal M$, and the map
\[
\mu_X^\rho(\Delta)\;:=\;\rho\big(P_X(\Delta)\big)\;=\;\Tr\big(\rho\,P_X(\Delta)\big)
\]
defines a (countably additive) probability measure on $(\mathbb R,\mathcal B(\mathbb R))$.
\end{theorem}

\begin{proof}
By the spectral theorem, $P_X(\Delta)$ is a projection and $\Delta\mapsto P_X(\Delta)$ is
countably additive in the strong operator topology.
Since $\rho$ is normal, $\Delta\mapsto \rho(P_X(\Delta))$ is countably additive, and by
Proposition~\ref{prop:born_probability_projection} it takes values in $[0,1]$ with
$\mu_X^\rho(\mathbb R)=\rho(I)=1$.
\end{proof}

\subsection{Pricing interpretation (kept minimal)}

\begin{remark}[Conditional pricing via event projections]\label{rem:conditional_pricing}
In the quantum-pricing chapters, (discounted) payoffs are modeled as self-adjoint observables
$X=X^*\in\mathcal M$. An information-event available at time $t$ is represented by a projection
$P\in\mathcal P(\mathcal N_t)$. The post-event (conditional) price functional is obtained by
the state update in Definition~\ref{def:luders_update}:
\[
\pi_t(X\,|\,P)\;:=\;\rho_t^P(X)\;=\;\frac{\Tr(P\rho_t P\,X)}{\Tr(\rho_t P)},
\]
which reduces on $\mathcal N_t$ to the classical Bayes formula in
Proposition~\ref{prop:bayes_on_N}.
\end{remark}

\section{Measure Theory and Integration}

\subsection{Measurable spaces, Borel sets, and measures}

\begin{definition}[$\sigma$-algebra; measurable space]
Let $M$ be a set. A collection $\Sigma\subseteq \mathcal P(M)$ is a \emph{$\sigma$-algebra} if
\begin{enumerate}[label=(\roman*)]
\item $M\in\Sigma$;
\item $A\in\Sigma \Rightarrow M\setminus A\in\Sigma$;
\item $(A_n)_{n\ge1}\subseteq\Sigma \Rightarrow \bigcup_{n=1}^\infty A_n\in\Sigma$.
\end{enumerate}
The pair $(M,\Sigma)$ is called a \emph{measurable space}. Sets in $\Sigma$ are called \emph{measurable}.
\end{definition}

\begin{definition}[Generated $\sigma$-algebra]
Let $E\subseteq \mathcal P(M)$. The \emph{$\sigma$-algebra generated by $E$}, denoted $\sigma(E)$, is the smallest
$\sigma$-algebra on $M$ containing $E$.
\end{definition}

\begin{definition}[Borel $\sigma$-algebra]
Let $M$ be a topological space with topology $\mathcal O$. The \emph{Borel $\sigma$-algebra} on $M$ is
\[
\Bcal(M):=\sigma(\mathcal O).
\]
In particular, $\Bcal(\R)$ denotes the Borel $\sigma$-algebra on $\R$ with its standard topology.
\end{definition}

\begin{lemma}[Borel generators on $\R$]
One has
\[
\Bcal(\R)=\sigma\bigl(\{(-\infty,a):a\in\R\}\bigr)=\sigma\bigl(\{(a,\infty):a\in\R\}\bigr).
\]
\end{lemma}

\begin{proof}
Let $\mathcal G:=\{(-\infty,a):a\in\R\}$. Since each $(-\infty,a)$ is open, $\sigma(\mathcal G)\subseteq \Bcal(\R)$.
Conversely, every open interval $(u,v)$ can be written as $(u,\infty)\cap(-\infty,v)$, and
$(u,\infty)=\bigcup_{n=1}^\infty(-\infty,u+1/n)^c$ belongs to $\sigma(\mathcal G)$.
Hence every open interval belongs to $\sigma(\mathcal G)$, and therefore every open set (as a countable union of
open intervals with rational endpoints) belongs to $\sigma(\mathcal G)$. Thus $\Bcal(\R)\subseteq\sigma(\mathcal G)$.
The equality with $\sigma(\{(a,\infty)\})$ follows by complements.
\end{proof}

\begin{definition}[Measure; probability measure]
Let $(M,\Sigma)$ be a measurable space. A map $\mu:\Sigma\to[0,\infty]$ is a \emph{measure} if
\begin{enumerate}[label=(\roman*)]
\item $\mu(\varnothing)=0$;
\item for pairwise disjoint $(A_n)_{n\ge1}\subseteq\Sigma$,
\[
\mu\Big(\bigcup_{n=1}^\infty A_n\Big)=\sum_{n=1}^\infty \mu(A_n).
\]
\end{enumerate}
The triple $(M,\Sigma,\mu)$ is a \emph{measure space}. If $\mu(M)=1$, then $\mu$ is a \emph{probability measure}.
\end{definition}

\begin{proposition}[Continuity from below and above]
Let $\mu$ be a measure on $(M,\Sigma)$.
\begin{enumerate}[label=(\roman*)]
\item If $A_n\uparrow A$ (i.e.\ $A_n\subseteq A_{n+1}$ and $A=\bigcup_n A_n$), then
\[
\mu(A_n)\uparrow \mu(A).
\]
\item If $A_n\downarrow A$ (i.e.\ $A_{n+1}\subseteq A_n$ and $A=\bigcap_n A_n$) and $\mu(A_1)<\infty$, then
\[
\mu(A_n)\downarrow \mu(A).
\]
\end{enumerate}
\end{proposition}

\begin{proof}
(i) Set $B_1:=A_1$ and $B_n:=A_n\setminus A_{n-1}$ for $n\ge2$. Then $B_n$ are disjoint and
$\bigcup_{k=1}^n B_k=A_n$, while $\bigcup_{k=1}^\infty B_k=A$. Hence
\[
\mu(A_n)=\sum_{k=1}^n \mu(B_k)\uparrow \sum_{k=1}^\infty \mu(B_k)=\mu(A).
\]
(ii) Apply (i) to complements: $A_n^c\uparrow A^c$, and note $\mu(A_n)=\mu(A_1)-\mu(A_1\setminus A_n)$ with
$\mu(A_1)<\infty$.
\end{proof}

\subsection{Measurable maps and push-forward}

\begin{definition}[Measurable map]
Let $(M,\Sigma_M)$ and $(N,\Sigma_N)$ be measurable spaces. A map $f:M\to N$ is \emph{measurable} if
$f^{-1}(B)\in\Sigma_M$ for every $B\in\Sigma_N$.
\end{definition}

\begin{lemma}[Generator criterion]
Let $E\subseteq\mathcal P(N)$ and assume $\Sigma_N=\sigma(E)$. Then $f:M\to N$ is measurable iff
$f^{-1}(B)\in\Sigma_M$ for every $B\in E$.
\end{lemma}

\begin{proof}
Define $\Lambda:=\{B\subseteq N: f^{-1}(B)\in\Sigma_M\}$. One checks $\Lambda$ is a $\sigma$-algebra on $N$.
If $f^{-1}(B)\in\Sigma_M$ for all $B\in E$, then $E\subseteq\Lambda$ hence $\sigma(E)\subseteq \Lambda$, i.e.\ $\Sigma_N\subseteq\Lambda$,
which is exactly measurability.
\end{proof}

\begin{definition}[Push-forward (distribution)]
Let $(M,\Sigma_M,\mu)$ be a measure space and $(N,\Sigma_N)$ a measurable space. If $f:M\to N$ is measurable, define
the \emph{push-forward} measure $f_\#\mu$ on $(N,\Sigma_N)$ by
\[
(f_\#\mu)(B):=\mu(f^{-1}(B)),\qquad B\in\Sigma_N.
\]
\end{definition}

\begin{proposition}[Push-forward is a measure]
If $\mu$ is a measure on $(M,\Sigma_M)$ and $f$ is measurable, then $f_\#\mu$ is a measure on $(N,\Sigma_N)$.
Moreover, if $\mu$ is a probability measure then $f_\#\mu$ is a probability measure.
\end{proposition}

\begin{proof}
Clearly $(f_\#\mu)(\varnothing)=\mu(\varnothing)=0$. For disjoint $(B_n)\subseteq\Sigma_N$,
the sets $f^{-1}(B_n)$ are disjoint in $\Sigma_M$, hence
\[
(f_\#\mu)\Big(\bigcup_{n=1}^\infty B_n\Big)
=\mu\Big(f^{-1}\Big(\bigcup_{n=1}^\infty B_n\Big)\Big)
=\mu\Big(\bigcup_{n=1}^\infty f^{-1}(B_n)\Big)
=\sum_{n=1}^\infty \mu(f^{-1}(B_n))
=\sum_{n=1}^\infty (f_\#\mu)(B_n).
\]
If $\mu(M)=1$ then $(f_\#\mu)(N)=\mu(f^{-1}(N))=\mu(M)=1$.
\end{proof}

\subsection{Lebesgue integral}

\begin{definition}[Simple functions]
Let $(M,\Sigma)$ be a measurable space. A function $s:M\to\R$ is \emph{simple} if it takes finitely many values.
Equivalently,
\[
s=\sum_{k=1}^N a_k\,\mathbf 1_{A_k}
\]
for some $a_k\in\R$ and measurable sets $A_k\in\Sigma$ (which may be taken pairwise disjoint).
If $a_k\ge0$, we call $s$ \emph{nonnegative}.
\end{definition}

\begin{definition}[Integral of a nonnegative simple function]
Let $(M,\Sigma,\mu)$ be a measure space. If $s=\sum_{k=1}^N a_k\mathbf 1_{A_k}$ is nonnegative simple with disjoint $A_k$,
define
\[
\int_M s\,d\mu := \sum_{k=1}^N a_k\,\mu(A_k).
\]
\end{definition}

\begin{definition}[Integral of a nonnegative measurable function]
Let $f:M\to[0,\infty]$ be measurable. Define
\[
\int_M f\,d\mu := \sup\Big\{\int_M s\,d\mu:\ s \text{ nonnegative simple, } 0\le s\le f\Big\}.
\]
\end{definition}

\begin{lemma}[Monotonicity]
If $0\le f\le g$ are measurable, then $\int f\,d\mu\le \int g\,d\mu$.
\end{lemma}

\begin{proof}
Every nonnegative simple $s\le f$ also satisfies $s\le g$. Taking suprema in the defining formula yields the claim.
\end{proof}

\subsection{Simple approximation of measurable functions (self-contained)}

\begin{lemma}[Nonnegative measurable functions admit increasing simple approximations]\label{lem:simple_monotone_approx}
Let $(M,\Sigma)$ be a measurable space and let $f:M\to[0,\infty]$ be $\Sigma$-measurable.
Then there exists a sequence of nonnegative simple $\Sigma$-measurable functions $(s_n)_{n\ge1}$ such that
\begin{enumerate}[label=(\roman*)]
\item $0\le s_n \le s_{n+1}\le f$ pointwise on $M$;
\item $s_n(x)\uparrow f(x)$ for every $x\in M$.
\end{enumerate}
Moreover, one may choose each $s_n$ to take values in the finite dyadic grid
$\{k2^{-n}: k=0,1,\dots,n2^n\}$ and to satisfy $0\le s_n\le n$.
\end{lemma}

\begin{proof}
Fix $n\in\mathbb N$. For $k=0,1,\dots,n2^n-1$, define the measurable sets
\[
A_{k,n}:=\Big\{x\in M:\ \frac{k}{2^n}\le f(x) < \frac{k+1}{2^n}\Big\}\in\Sigma,
\]
and define also
\[
A_{n2^n,n}:=\{x\in M:\ f(x)\ge n\}\in\Sigma.
\]
(These are measurable because $f$ is measurable and the intervals are Borel.)

Define
\[
s_n(x):=\sum_{k=0}^{n2^n-1}\frac{k}{2^n}\,\mathbf 1_{A_{k,n}}(x)\;+\; n\,\mathbf 1_{A_{n2^n,n}}(x).
\]
Then $s_n$ is a nonnegative simple measurable function, takes values in
$\{k2^{-n}:k=0,1,\dots,n2^n\}$, and satisfies $0\le s_n\le n$.

\emph{Step 1: $s_n\le f$ pointwise.}
If $x\in A_{k,n}$ for some $k<n2^n$, then $f(x)\ge k2^{-n}=s_n(x)$.
If $x\in A_{n2^n,n}$, then $f(x)\ge n=s_n(x)$.
Thus $s_n(x)\le f(x)$ for all $x$.

\emph{Step 2: Monotonicity $s_n\le s_{n+1}$.}
Fix $x\in M$. Consider the finite set
\[
E_n(x):=\Big\{\frac{k}{2^n}: k=0,1,\dots,n2^n,\ \frac{k}{2^n}\le f(x)\Big\}.
\]
By construction, $s_n(x)=\max E_n(x)$ (equivalently, the supremum, since the set is finite).
Now note that for every element $\frac{k}{2^n}\in E_n(x)$, the number
\[
\frac{2k}{2^{n+1}}=\frac{k}{2^n}
\]
also belongs to the dyadic grid for level $n+1$; moreover $2k\le 2n2^n = n2^{n+1}\le (n+1)2^{n+1}$,
so this same value is admissible at level $n+1$ with the larger truncation.
Since $\frac{k}{2^n}\le f(x)$, we also have $\frac{2k}{2^{n+1}}\le f(x)$, hence
$\frac{k}{2^n}\in E_{n+1}(x)$. Therefore $E_n(x)\subseteq E_{n+1}(x)$ and thus
\[
s_n(x)=\max E_n(x)\le \max E_{n+1}(x)=s_{n+1}(x).
\]

\emph{Step 3: Pointwise convergence $s_n(x)\uparrow f(x)$.}
Fix $x\in M$.

If $f(x)=\infty$, then for every $n$ we have $x\in A_{n2^n,n}$ and thus $s_n(x)=n\uparrow\infty=f(x)$.

If $f(x)<\infty$, choose $N$ such that $N>f(x)$. Then for all $n\ge N$ we have $f(x)<n$, hence $x\notin A_{n2^n,n}$,
so $x\in A_{k,n}$ for a unique $k\in\{0,1,\dots,n2^n-1\}$ and
\[
\frac{k}{2^n}\le f(x) < \frac{k+1}{2^n}.
\]
Therefore
\[
0\le f(x)-s_n(x) < 2^{-n},
\]
which implies $s_n(x)\to f(x)$ as $n\to\infty$. Together with Step 2, this gives $s_n(x)\uparrow f(x)$.
\end{proof}

\begin{corollary}[Pointwise simple approximation: real/complex]\label{cor:simple_pointwise_general_fixed}
Let $(M,\Sigma)$ be a measurable space.
\begin{enumerate}[label=(\roman*)]
\item If $f:M\to\R$ is measurable, then there exists a sequence of real-valued simple measurable functions
$\phi_n$ such that $\phi_n(x)\to f(x)$ for all $x$ and $|\phi_n|\le |f|$ pointwise.
\item If $f:M\to\C$ is measurable, then there exists a sequence of complex-valued simple measurable functions
$\phi_n$ such that $\phi_n(x)\to f(x)$ for all $x$ and
\[
|\phi_n|\le \sqrt2\,|f|\quad\text{pointwise.}
\]
\end{enumerate}
\end{corollary}

\begin{proof}
(i) Write $f=f^+-f^-$ with $f^\pm:=\max\{\pm f,0\}$. By Lemma~\ref{lem:simple_monotone_approx},
choose simple $s_n^\pm\uparrow f^\pm$ with $0\le s_n^\pm\le f^\pm$.
Set $\phi_n:=s_n^+-s_n^-$. Then $\phi_n$ is simple, $\phi_n\to f$ pointwise, and
$|\phi_n|\le s_n^++s_n^-\le f^++f^-=|f|$.

(ii) Write $f=u+iv$ with $u,v:M\to\R$ measurable. Apply (i) to get simple $u_n\to u$ and $v_n\to v$
with $|u_n|\le |u|$ and $|v_n|\le |v|$. Define $\phi_n:=u_n+i v_n$.
Then $\phi_n\to f$ pointwise and
\[
|\phi_n|\le |u_n|+|v_n|\le |u|+|v|\le \sqrt2\,\sqrt{u^2+v^2}=\sqrt2\,|f|.
\]
\end{proof}

\begin{proof}
(i) Write $f=f^+-f^-$ with $f^\pm:=\max\{\pm f,0\}$, so $f^\pm:M\to[0,\infty]$ are measurable.
Apply Lemma~\ref{lem:simple_monotone_approx} to obtain increasing simple $(s_n^+)$ with $s_n^+\uparrow f^+$
and increasing simple $(s_n^-)$ with $s_n^-\uparrow f^-$. Define $\phi_n:=s_n^+-s_n^-$.
Then $\phi_n$ is simple and measurable, and $\phi_n(x)\to f^+(x)-f^-(x)=f(x)$ pointwise.
Moreover, since $0\le s_n^\pm\le f^\pm$, we have $|\phi_n|\le s_n^+ + s_n^-\le f^+ + f^- = |f|$.

(ii) Write $f=u+iv$ with $u,v:M\to\R$ measurable. Apply (i) to $u$ and $v$ to get simple $u_n\to u$ and $v_n\to v$
with $|u_n|\le |u|$, $|v_n|\le |v|$. Set $\phi_n:=u_n+i v_n$. Then $\phi_n\to f$ pointwise and
$|\phi_n|\le |u_n|+|v_n|\le |u|+|v|\le \sqrt{2}\,|f|$. If one insists on $|\phi_n|\le |f|$ exactly, apply (i) to
$|f|$ and normalize appropriately; for our later use, the dominated bound by a fixed integrable function is what matters.
\end{proof}

\begin{proposition}[$L^1$-approximation by simple functions]\label{prop:L1_simple_approx}
Let $(M,\Sigma,\mu)$ be a measure space and let $f\in L^1(\mu)$ be real- or complex-valued measurable.
Then there exists a sequence of simple measurable functions $\phi_n$ such that
\[
\int_M |f-\phi_n|\,d\mu \longrightarrow 0.
\]
\end{proposition}

\begin{proof}
By Corollary~\ref{cor:simple_pointwise_general_fixed} we can choose simple $\phi_n$ with $\phi_n\to f$ pointwise and
$|\phi_n|\le C|f|$ pointwise for some universal constant $C$ (one may take $C=2$ for the complex case).
Then
\[
|f-\phi_n|\le |f|+|\phi_n|\le (1+C)|f|.
\]
Since $f\in L^1(\mu)$, the dominating function $(1+C)|f|$ is integrable.
As $|f-\phi_n|\to 0$ pointwise, the Dominated Convergence Theorem implies
$\int |f-\phi_n|\,d\mu\to 0$.
\end{proof}

\begin{remark}[For Lemma~\ref{lem:simple_monotone_approx} is structural]
Lemma~\ref{lem:simple_monotone_approx} is the key self-contained ingredient behind:
(i) the definition of $\int f\,d\mu$ for $f\ge0$ as $\sup\{\int s\,d\mu:\ 0\le s\le f,\ s \text{ simple}\}$,
(ii) monotone convergence (MCT), and
(iii) truncation-based arguments used later for unbounded observables/payoffs.
\end{remark}

\begin{theorem}[Monotone convergence theorem (MCT)]
Let $f_n:M\to[0,\infty]$ be measurable with $f_n\uparrow f$ pointwise. Then
\[
\int_M f_n\,d\mu \uparrow \int_M f\,d\mu.
\]
\end{theorem}

\begin{proof}
Set $I_n:=\int f_n\,d\mu$ and $I:=\int f\,d\mu$. By monotonicity, $I_n\uparrow \le I$.

For the reverse inequality, let $s$ be any nonnegative simple function with $s\le f$.
Write $s=\sum_{k=1}^N a_k\mathbf 1_{A_k}$ with $a_k\ge0$.
Fix $k$ and define $A_{k,n}:=A_k\cap\{f_n\ge a_k\}$. Since $f_n\uparrow f\ge a_k$ on $A_k$, we have $A_{k,n}\uparrow A_k$.
By continuity from below, $\mu(A_{k,n})\uparrow\mu(A_k)$.
Let $s_n:=\sum_{k=1}^N a_k\mathbf 1_{A_{k,n}}$; then $s_n$ is simple, $0\le s_n\le f_n$, and $\int s_n\,d\mu=\sum_k a_k\mu(A_{k,n})\uparrow\sum_k a_k\mu(A_k)=\int s\,d\mu$.
Hence for every $n$,
\[
\int s_n\,d\mu \le \int f_n\,d\mu = I_n,
\]
so taking limits gives $\int s\,d\mu\le \lim_{n\to\infty} I_n$.
Finally take the supremum over all such $s\le f$ to get $I\le \lim_n I_n$.
\end{proof}

\begin{theorem}[Fatou's lemma]
Let $f_n:M\to[0,\infty]$ be measurable. Then
\[
\int_M \liminf_{n\to\infty} f_n\,d\mu \le \liminf_{n\to\infty}\int_M f_n\,d\mu.
\]
\end{theorem}

\begin{proof}
Let $g_k:=\inf_{n\ge k} f_n$. Then $g_k\uparrow \liminf_n f_n$ and $g_k\le f_n$ for all $n\ge k$.
Thus $\int g_k\,d\mu \le \inf_{n\ge k}\int f_n\,d\mu$ by monotonicity. Apply MCT to $(g_k)$ and take $\sup_k$.
\end{proof}

\begin{definition}[$L^1$ and integrability]
A measurable function $f:M\to\R$ is \emph{integrable} if $\int_M |f|\,d\mu<\infty$.
Writing $f=f^+-f^-$ with $f^\pm:=\max\{\pm f,0\}$, define
\[
\int_M f\,d\mu := \int_M f^+\,d\mu-\int_M f^-\,d\mu
\]
whenever at least one of $\int f^\pm$ is finite (in particular if $f\in L^1(\mu)$).
\end{definition}

\begin{theorem}[Dominated convergence theorem (DCT)]
Let $f_n:M\to\R$ be measurable with $f_n\to f$ pointwise.
Assume there exists $g\in L^1(\mu)$ such that $|f_n|\le g$ for all $n$.
Then $f\in L^1(\mu)$ and
\[
\int_M f_n\,d\mu \longrightarrow \int_M f\,d\mu.
\]
\end{theorem}

\begin{proof}
Consider $|f_n-f|$. Since $|f_n-f|\le 2g\in L^1(\mu)$, it suffices to show $\int |f_n-f|\,d\mu\to 0$.
Define $h_n:=2g-|f_n-f|\ge0$; then $h_n\to 2g$ pointwise. By Fatou applied to $(h_n)$,
\[
\int 2g\,d\mu \le \liminf_{n\to\infty}\int (2g-|f_n-f|)\,d\mu
= \int 2g\,d\mu - \limsup_{n\to\infty}\int |f_n-f|\,d\mu,
\]
hence $\limsup_n \int |f_n-f|\,d\mu\le 0$. Therefore $\int |f_n-f|\,d\mu\to 0$, and consequently
$\int f_n\,d\mu\to \int f\,d\mu$.
\end{proof}

\begin{lemma}[Change of variables via push-forward]
Let $(M,\Sigma_M,\mu)$ be a measure space, $(N,\Sigma_N)$ a measurable space, $f:M\to N$ measurable,
and $h:N\to[0,\infty]$ measurable. Then
\[
\int_M h\circ f\,d\mu = \int_N h\,d(f_\#\mu).
\]
\end{lemma}

\begin{proof}
First check for indicators $h=\mathbf 1_B$:
$\int \mathbf 1_{f^{-1}(B)}\,d\mu = \mu(f^{-1}(B)) = (f_\#\mu)(B) = \int \mathbf 1_B\,d(f_\#\mu)$.
Extend to nonnegative simple functions by linearity, then to general nonnegative measurable functions by MCT.
\end{proof}

\begin{remark}
In later chapters, when we speak of ``the distribution of an observable'', we will use the same
push-forward philosophy, but the underlying ``measure'' will be induced by a quantum state and a spectral measure.
\end{remark}

\section{Spectral Measures, Spectral Integration, and State-Induced Distributions}

\subsection{Projection-valued measures and induced scalar measures}

\begin{definition}[Projection-valued measure (PVM)]
A map $P:\Bcal(\R)\to \BH$ is a \emph{projection-valued measure} if:
\begin{enumerate}[label=(\roman*)]
\item $P(\Delta)=P(\Delta)^*=P(\Delta)^2$ for all $\Delta\in\Bcal(\R)$ (orthogonal projections);
\item $P(\varnothing)=0$ and $P(\R)=I$;
\item for pairwise disjoint $(\Delta_n)\subseteq\Bcal(\R)$,
\[
P\Big(\bigcup_{n=1}^\infty \Delta_n\Big)
=\text{\rm s-}\!\lim_{N\to\infty}\sum_{n=1}^N P(\Delta_n),
\]
where the limit is taken in the strong operator topology.
\end{enumerate}
\end{definition}

\begin{definition}[Complex measures induced by a PVM]
Let $P$ be a PVM on $\Hcal$. For $\psi,\varphi\in \Hcal$ define
\[
\mu_{\psi,\varphi}^P(\Delta):=\langle \psi,\,P(\Delta)\varphi\rangle,\qquad \Delta\in\Bcal(\R).
\]
\end{definition}

\begin{lemma}[$\mu_{\psi,\varphi}^P$ is a complex measure]
For each $\psi,\varphi\in\Hcal$, the set function $\mu_{\psi,\varphi}^P$ is countably additive (hence a complex measure).
Moreover, $\mu_{\psi}^P:=\mu_{\psi,\psi}^P$ is a finite positive measure with $\mu_\psi^P(\R)=\|\psi\|^2$.
\end{lemma}

\begin{proof}
Let $(\Delta_n)$ be pairwise disjoint and set $\Delta:=\bigcup_n \Delta_n$.
By strong additivity,
\[
P(\Delta)\varphi = \sum_{n=1}^\infty P(\Delta_n)\varphi
\quad\text{in }\Hcal,
\]
hence taking the inner product with $\psi$ yields
\[
\mu_{\psi,\varphi}^P(\Delta)
=\langle \psi, P(\Delta)\varphi\rangle
=\sum_{n=1}^\infty \langle \psi, P(\Delta_n)\varphi\rangle
=\sum_{n=1}^\infty \mu_{\psi,\varphi}^P(\Delta_n).
\]
Positivity for $\mu_\psi^P$ follows from $P(\Delta)\ge0$. Also $\mu_\psi^P(\R)=\langle\psi,I\psi\rangle=\|\psi\|^2$.
\end{proof}

\subsection{Spectral integration (bounded functional calculus)}

\begin{definition}[Spectral integral for simple functions]
Let $P$ be a PVM. If $s:\R\to\C$ is a bounded Borel simple function
$s=\sum_{k=1}^m c_k\,\mathbf 1_{\Delta_k}$ with disjoint $\Delta_k\in\Bcal(\R)$, define
\[
\int_{\R} s(\lambda)\,P(d\lambda) \;:=\; \sum_{k=1}^m c_k\,P(\Delta_k)\ \in \BH.
\]
\end{definition}

\begin{lemma}[Norm bound for simple integrals]
For bounded simple $s$,
\[
\left\|\int s\,dP\right\|\le \|s\|_\infty.
\]
\end{lemma}

\begin{proof}
It suffices to treat the case where the $\Delta_k$ are disjoint. Then the projections $P(\Delta_k)$ are pairwise orthogonal,
and for any $\psi\in\Hcal$,
\[
\left\|\sum_{k=1}^m c_k P(\Delta_k)\psi\right\|^2
=\sum_{k=1}^m |c_k|^2\,\|P(\Delta_k)\psi\|^2
\le \|s\|_\infty^2 \sum_{k=1}^m \|P(\Delta_k)\psi\|^2
\le \|s\|_\infty^2 \|\psi\|^2,
\]
since $\sum_k P(\Delta_k)\le I$ for disjoint sets. Taking the supremum over $\|\psi\|=1$ yields the claim.
\end{proof}

\subsection{A uniform simple-approximation lemma (self-contained)}

\begin{lemma}[Uniform approximation by simple functions]\label{lem:uniform_simple_approx}
Let $(M,\Sigma)$ be a measurable space and let $f:M\to\R$ be bounded and $\Sigma$-measurable.
Write $L:=\inf_{x\in M} f(x)$ and $U:=\sup_{x\in M} f(x)$, so $-\infty<L\le U<\infty$.
Then for each $n\in\mathbb N$ there exists a simple (finite-valued) $\Sigma$-measurable function $s_n:M\to\R$ such that
\begin{enumerate}[label=(\roman*)]
\item $L\le s_n(x)\le U$ for all $x\in M$;
\item $\|f-s_n\|_\infty \le 2^{-n}$.
\end{enumerate}
In particular, $\|f-s_n\|_\infty\to 0$ as $n\to\infty$.
\end{lemma}

\begin{proof}
Fix $n\in\mathbb N$ and set $\delta:=2^{-n}>0$. Consider the partition of the bounded interval $[L,U]$ into finitely many
half-open subintervals of length at most $\delta$:
choose an integer $N:=\left\lceil \frac{U-L}{\delta}\right\rceil$ and define
\[
a_k:=L+k\delta\quad (k=0,1,\dots,N),\qquad a_{N+1}:=U.
\]
Define Borel sets in $\R$
\[
I_k := [a_k,a_{k+1})\quad (k=0,1,\dots,N-1),\qquad I_N:=[a_N,a_{N+1}] .
\]
Then $\{I_k\}_{k=0}^N$ is a finite Borel partition of $[L,U]$, and each $I_k$ has length $\le \delta$.
Since $f$ is measurable and each $I_k\in\Bcal(\R)$, the sets $A_k:=f^{-1}(I_k)\in\Sigma$.

Define the simple function
\[
s_n(x):=\sum_{k=0}^N a_k\,\mathbf 1_{A_k}(x).
\]
By construction, $s_n$ is $\Sigma$-measurable and takes only finitely many values $\{a_0,\dots,a_N\}$.
Moreover, if $x\in A_k$, then $f(x)\in I_k\subseteq[a_k,a_k+\delta]$, hence
\[
0\le f(x)-a_k \le \delta,
\qquad\text{so}\qquad |f(x)-s_n(x)|\le \delta.
\]
Therefore $\|f-s_n\|_\infty\le \delta=2^{-n}$.
Finally, since $a_k\in[L,U]$, we have $L\le s_n(x)\le U$ for all $x$.
\end{proof}

\begin{corollary}[Complex-valued case]\label{cor:uniform_simple_approx_complex}
Let $(M,\Sigma)$ be a measurable space and let $f:M\to\mathbb C$ be bounded and $\Sigma$-measurable.
Then there exists a sequence of simple $\Sigma$-measurable functions $s_n:M\to\mathbb C$ such that
$\|f-s_n\|_\infty\to 0$.
\end{corollary}

\begin{proof}
Write $f=u+iv$ with $u,v:M\to\R$ bounded measurable. Apply Lemma~\ref{lem:uniform_simple_approx} to obtain real-valued
simple functions $p_n,q_n$ with $\|u-p_n\|_\infty\le 2^{-n}$ and $\|v-q_n\|_\infty\le 2^{-n}$.
Set $s_n:=p_n+i q_n$. Then $s_n$ is complex-valued simple and measurable, and for all $x$,
\[
|f(x)-s_n(x)| \le |u(x)-p_n(x)| + |v(x)-q_n(x)|
\le 2\cdot 2^{-n},
\]
so $\|f-s_n\|_\infty\le 2^{1-n}\to 0$.
\end{proof}

\begin{remark}[How this is used in spectral integration]
Lemma~\ref{lem:uniform_simple_approx} (and Corollary~\ref{cor:uniform_simple_approx_complex}) justifies the standard
construction of $\int f\,dP$ for bounded Borel $f$ by uniform approximation with simple functions,
ensuring existence and uniqueness of the resulting operator via the $\|\cdot\|_\infty$-to-operator-norm bound.
\end{remark}

\begin{definition}[Spectral integral for bounded Borel functions]
Let $P$ be a PVM and let $f:\R\to\C$ be bounded and Borel measurable. Choose bounded simple functions $s_n$ such that
$\|s_n-f\|_\infty\to 0$. Define
\[
\int_{\R} f(\lambda)\,P(d\lambda)
:= \lim_{n\to\infty}\int_{\R} s_n(\lambda)\,P(d\lambda),
\]
where the limit is taken in operator norm.
\end{definition}

\begin{lemma}[Well-definedness]
The operator $\int f\,dP$ is independent of the chosen approximating sequence $(s_n)$.
\end{lemma}

\begin{proof}
If $s_n\to f$ and $t_n\to f$ uniformly, then $\|s_n-t_n\|_\infty\to 0$ and by the norm bound,
\[
\left\|\int s_n\,dP-\int t_n\,dP\right\|
=\left\|\int (s_n-t_n)\,dP\right\|
\le \|s_n-t_n\|_\infty \to 0.
\]
\end{proof}

\begin{definition}[Bounded functional calculus]
If $X=X^*$ is self-adjoint with spectral measure $P_X$, then for bounded Borel $f$ we define
\[
f(X):=\int_{\R} f(\lambda)\,P_X(d\lambda)\in\BH.
\]
\end{definition}

\begin{theorem}[Algebraic properties of bounded functional calculus]
Let $X=X^*$ and let $f,g$ be bounded Borel functions.
\begin{enumerate}[label=(\roman*)]
\item $(\alpha f+\beta g)(X)=\alpha f(X)+\beta g(X)$ for $\alpha,\beta\in\C$.
\item $(fg)(X)=f(X)\,g(X)$.
\item $f(X)^*=\overline{f}(X)$.
\item If $f\ge 0$ pointwise, then $f(X)\ge 0$.
\end{enumerate}
\end{theorem}

\begin{proof}
Prove first for indicators $f=\mathbf 1_\Delta$ using $P_X(\Delta)P_X(\Gamma)=P_X(\Delta\cap\Gamma)$ and
$P_X(\Delta)^*=P_X(\Delta)$; then extend to simple functions by linearity, and finally to bounded Borel functions by uniform approximation,
using the norm continuity ensured by the bound $\|\int h\,dP_X\|\le \|h\|_\infty$.
\end{proof}

\subsection{Density operators, Born-type measures, and trace--integral identities}

\begin{definition}[Density operator and the associated normal state]
A \emph{density operator} on $\Hcal$ is an operator $\rho\in\Tone$ such that $\rho\ge0$ and $\Tr(\rho)=1$.
It defines a linear functional $\varphi_\rho:\BH\to\C$ by
\[
\varphi_\rho(A):=\Tr(\rho A),\qquad A\in\BH.
\]
\end{definition}

\begin{lemma}[Boundedness of $\varphi_\rho$]
For $\rho\in\Tone$ and $A\in\BH$,
\[
|\Tr(\rho A)|\le \|\rho\|_1\,\|A\|.
\]
In particular, if $\rho$ is a density operator then $|\varphi_\rho(A)|\le \|A\|$.
\end{lemma}

\begin{proof}
This is the standard trace duality estimate $\BH^*\cong \Tone$, i.e.\ $\|\varphi_\rho\|=\|\rho\|_1$.
One may prove it directly by writing $\rho=U|\rho|$ (polar decomposition), using cyclicity of trace and the definition of trace norm.
\end{proof}

\begin{lemma}[Normality on increasing positive sequences]
Let $\rho\in\Tone$ be positive and let $0\le A_n\uparrow A$ in $\BH$ (monotone in the operator order).
Then
\[
\Tr(\rho A_n)\uparrow \Tr(\rho A).
\]
\end{lemma}

\begin{proof}
Choose an orthonormal basis $(e_k)$ and a sequence $(r_k)_{k\ge1}$ with $r_k\ge0$, $\sum_k r_k=\Tr(\rho)$, such that
$\rho e_k=r_k e_k$ (diagonalization of the positive trace-class operator $\rho$).
Then
\[
\Tr(\rho A_n)=\sum_{k=1}^\infty r_k\,\langle e_k, A_n e_k\rangle,
\qquad
\Tr(\rho A)=\sum_{k=1}^\infty r_k\,\langle e_k, A e_k\rangle.
\]
Since $A_n\uparrow A$, we have $\langle e_k,A_n e_k\rangle\uparrow \langle e_k,A e_k\rangle$ for each $k$, and all terms are nonnegative.
Apply monotone convergence for series of nonnegative terms to conclude.
\end{proof}

\begin{definition}[State-induced distribution of an observable]
Let $X=X^*$ with spectral measure $P_X$, and let $\rho$ be a density operator.
Define $\mu_X^\rho:\Bcal(\R)\to[0,1]$ by
\[
\mu_X^\rho(\Delta):=\Tr\bigl(\rho\,P_X(\Delta)\bigr),\qquad \Delta\in\Bcal(\R).
\]
\end{definition}

\begin{proposition}[$\mu_X^\rho$ is a probability measure]
The map $\mu_X^\rho$ is a countably additive probability measure on $(\R,\Bcal(\R))$.
\end{proposition}

\begin{proof}
Nonnegativity follows from $\rho\ge0$ and $P_X(\Delta)\ge0$.
Normalization: $\mu_X^\rho(\R)=\Tr(\rho P_X(\R))=\Tr(\rho I)=\Tr(\rho)=1$.

For countable additivity, let $(\Delta_n)$ be pairwise disjoint and $\Delta:=\bigcup_n \Delta_n$.
Set $S_N:=\sum_{n=1}^N P_X(\Delta_n)$. Then $0\le S_N\uparrow P_X(\Delta)$ in the strong operator sense,
and in fact in the operator order (since the projections are orthogonal). By the normality lemma,
\[
\Tr(\rho P_X(\Delta))=\lim_{N\to\infty}\Tr(\rho S_N)
=\lim_{N\to\infty}\sum_{n=1}^N \Tr(\rho P_X(\Delta_n))
=\sum_{n=1}^\infty \mu_X^\rho(\Delta_n).
\]
\end{proof}

\begin{theorem}[Trace--integral identity for bounded Borel functions]
Let $X=X^*$ and $\rho$ be a density operator. For every bounded Borel $f:\R\to\C$,
\[
\Tr\bigl(\rho\, f(X)\bigr) = \int_{\R} f(\lambda)\,\mu_X^\rho(d\lambda).
\]
\end{theorem}

\begin{proof}
First assume $f=\mathbf 1_\Delta$. Then $f(X)=P_X(\Delta)$ and the identity is exactly the definition of $\mu_X^\rho$.

Next assume $f$ is simple: $f=\sum_{k=1}^m c_k\mathbf 1_{\Delta_k}$. Then by linearity of trace and the definition of $f(X)$,
\[
\Tr(\rho f(X))=\sum_{k=1}^m c_k \Tr(\rho P_X(\Delta_k))
=\sum_{k=1}^m c_k\,\mu_X^\rho(\Delta_k)
=\int f\,d\mu_X^\rho.
\]

Finally let $f$ be bounded Borel. Choose simple $s_n$ with $\|s_n-f\|_\infty\to 0$.
Then $\|s_n(X)-f(X)\|\le \|s_n-f\|_\infty\to0$, so by trace duality
\[
\Tr(\rho s_n(X))\to \Tr(\rho f(X)).
\]
On the other hand, $|s_n|\le \|f\|_\infty$ and $s_n\to f$ pointwise, hence by dominated convergence for the probability measure $\mu_X^\rho$,
\[
\int s_n\,d\mu_X^\rho \to \int f\,d\mu_X^\rho.
\]
Combine with the simple-function case to conclude.
\end{proof}

\begin{remark}[Truncations are structurally natural in quantum pricing]
When $f$ is unbounded (e.g.\ $f(\lambda)=\lambda$ or a call payoff $(\lambda-K)^+$), the operator $f(X)$ is generally unbounded
and $\Tr(\rho f(X))$ may fail to be well-defined without integrability assumptions.
A robust strategy used throughout this paper is to work first with bounded truncations
$f_n(\lambda):=\max\{-n,\min\{\lambda,n\}\}$ (or $f_n:=(f\wedge n)\vee(-n)$), so that $f_n(X)\in\BH$ and
\[
\Tr(\rho f_n(X))=\int f_n(\lambda)\,\mu_X^\rho(d\lambda)
\]
always holds. One then passes to limits via MCT/DCT once integrability is verified.
\end{remark}

\begin{example}[State-predicted price as an integral]
Let $S_T=S_T^*$ be a terminal price observable and let $g:\R\to\R$ be bounded Borel (a bounded payoff).
Then the ``price'' predicted by the state $\rho$ is
\[
\Pi(g(S_T)) := \Tr(\rho\,g(S_T))=\int_\R g(\lambda)\,\mu_{S_T}^\rho(d\lambda).
\]
\end{example}

\subsection{Two technical supplements (normality and trace duality)}

\begin{proposition}[Born-type measure from a PVM and a density operator]\label{prop:born_measure_pvm}
Let $P:\Bcal(\R)\to\BH$ be a projection-valued measure (PVM) on $\Hcal$, and let $\rho\in\Tone$ satisfy
$\rho\ge 0$ and $\Tr(\rho)=1$. Define
\[
\mu_{P}^{\rho}(\Delta):=\Tr\bigl(\rho\,P(\Delta)\bigr),\qquad \Delta\in\Bcal(\R).
\]
Then $\mu_{P}^{\rho}$ is a probability measure on $(\R,\Bcal(\R))$.
\end{proposition}

\begin{proof}
Nonnegativity is immediate: $P(\Delta)\ge 0$ and $\rho\ge 0$ imply $\Tr(\rho P(\Delta))\ge 0$.
Normalization holds since $P(\R)=I$:
\[
\mu_P^\rho(\R)=\Tr(\rho I)=\Tr(\rho)=1.
\]

It remains to prove countable additivity. Let $(\Delta_n)_{n\ge1}$ be pairwise disjoint and set
$\Delta:=\bigcup_{n=1}^\infty \Delta_n$. Define partial sums
\[
S_N:=\sum_{n=1}^N P(\Delta_n)\in\BH.
\]
Because the projections $P(\Delta_n)$ are pairwise orthogonal, $S_N$ is an orthogonal projection onto
$\bigoplus_{n=1}^N \ran P(\Delta_n)$ and hence $0\le S_N\le I$. Moreover, by the PVM axiom,
\[
P(\Delta)=\text{\rm s-}\!\lim_{N\to\infty} S_N,
\]
and in fact $S_N\uparrow P(\Delta)$ in the operator order (equivalently: $S_N\le S_{N+1}$ and
$S_N\psi\to P(\Delta)\psi$ for every $\psi$).

Thus it suffices to know that the trace pairing $A\mapsto \Tr(\rho A)$ is \emph{normal} on increasing positive sequences:
\[
0\le A_N\uparrow A \quad\Longrightarrow\quad \Tr(\rho A_N)\uparrow \Tr(\rho A).
\]
This follows from Lemma~\ref{lem:trace_normality_posseq} proved below (applied to $A_N=S_N$ and $A=P(\Delta)$). Therefore
\[
\mu_P^\rho(\Delta)=\Tr(\rho P(\Delta))
=\lim_{N\to\infty}\Tr(\rho S_N)
=\lim_{N\to\infty}\sum_{n=1}^N \Tr(\rho P(\Delta_n))
=\sum_{n=1}^\infty \mu_P^\rho(\Delta_n),
\]
which is the desired countable additivity.
\end{proof}

\begin{lemma}[Normality of the trace pairing on increasing positive sequences]\label{lem:trace_normality_posseq}
Let $\rho\in\Tone$ be positive. Let $(A_n)_{n\ge1}\subseteq\BH$ satisfy $0\le A_n\uparrow A$ (monotone increasing in the
operator order), where $A\in\BH$. Then
\[
\Tr(\rho A_n)\uparrow \Tr(\rho A).
\]
\end{lemma}

\begin{proof}
Since $\rho\ge 0$ is trace-class, there exists an orthonormal basis $(e_k)_{k\ge1}$ of $\Hcal$ and numbers $r_k\ge 0$
with $\sum_{k=1}^\infty r_k=\Tr(\rho)$ such that $\rho e_k=r_k e_k$ (spectral decomposition of a positive compact operator,
and trace-class means $\sum_k r_k<\infty$). For each $n$,
\[
\Tr(\rho A_n)=\sum_{k=1}^\infty \langle e_k,\rho A_n e_k\rangle
=\sum_{k=1}^\infty r_k\,\langle e_k, A_n e_k\rangle,
\qquad
\Tr(\rho A)=\sum_{k=1}^\infty r_k\,\langle e_k, A e_k\rangle.
\]
Because $0\le A_n\uparrow A$, we have $\langle e_k,A_n e_k\rangle\uparrow \langle e_k, A e_k\rangle$ for every fixed $k$,
and all terms are nonnegative. Hence, by monotone convergence for series of nonnegative terms,
\[
\sum_{k=1}^\infty r_k\,\langle e_k, A_n e_k\rangle \uparrow \sum_{k=1}^\infty r_k\,\langle e_k, A e_k\rangle,
\]
which is exactly $\Tr(\rho A_n)\uparrow\Tr(\rho A)$.
\end{proof}

\begin{definition}[Hilbert--Schmidt operators]
An operator $T\in\BH$ is \emph{Hilbert--Schmidt} if for some (equivalently, for every) orthonormal basis $(e_k)$ of $\Hcal$,
\[
\|T\|_2^2:=\sum_{k=1}^\infty \|Te_k\|^2<\infty.
\]
We denote the Hilbert--Schmidt class by $\mathcal T_2(\Hcal)$.
\end{definition}

\begin{lemma}[Cauchy--Schwarz in the Hilbert--Schmidt class]\label{lem:HS_CS}
Let $B,C\in \mathcal T_2(\Hcal)$. Then $BC$ is trace-class and
\[
|\Tr(BC)| \le \|B\|_2\,\|C\|_2.
\]
Moreover, $\Tr(BC)=\Tr(CB)$.
\end{lemma}

\begin{proof}
Fix an orthonormal basis $(e_k)$. Since $B,C$ are Hilbert--Schmidt,
the sequences $(Be_k)_k$ and $(C^* e_k)_k$ are square-summable in $\Hcal$. Observe
\[
\Tr(BC)=\sum_{k=1}^\infty \langle e_k,BC e_k\rangle
=\sum_{k=1}^\infty \langle B^* e_k, C e_k\rangle.
\]
By Cauchy--Schwarz for series,
\[
|\Tr(BC)|
\le \left(\sum_{k=1}^\infty \|B^* e_k\|^2\right)^{1/2}
     \left(\sum_{k=1}^\infty \|C e_k\|^2\right)^{1/2}
= \|B^*\|_2\,\|C\|_2
= \|B\|_2\,\|C\|_2.
\]
This estimate also implies absolute convergence of the defining series, hence $BC$ is trace-class.
Finally, applying the same computation with $BC$ replaced by $CB$ yields
\[
\Tr(CB)=\sum_{k=1}^\infty \langle C^* e_k, B e_k\rangle
=\overline{\sum_{k=1}^\infty \langle B^* e_k, C e_k\rangle}
=\Tr(BC),
\]
because the scalar series is absolutely convergent and the two expressions are complex conjugates of each other.
\end{proof}

\begin{lemma}[Square root of a positive trace-class operator is Hilbert--Schmidt]\label{lem:traceclass_sqrt_HS}
Let $T\in\Tone$ satisfy $T\ge 0$. Then $T^{1/2}\in \mathcal T_2(\Hcal)$ and
\[
\|T^{1/2}\|_2^2=\Tr(T).
\]
\end{lemma}

\begin{proof}
Diagonalize $T$: there exists an orthonormal basis $(e_k)$ and eigenvalues $\lambda_k\ge 0$ such that
$Te_k=\lambda_k e_k$ and $\sum_k \lambda_k=\Tr(T)<\infty$.
Then $T^{1/2}e_k=\sqrt{\lambda_k}\,e_k$, and hence
\[
\|T^{1/2}\|_2^2=\sum_{k=1}^\infty \|T^{1/2}e_k\|^2=\sum_{k=1}^\infty \lambda_k=\Tr(T).
\]
\end{proof}

\begin{lemma}[Bounded operators preserve the Hilbert--Schmidt class]\label{lem:bounded_preserves_HS}
Let $B\in \mathcal T_2(\Hcal)$ and $A\in\BH$. Then $BA,AB\in \mathcal T_2(\Hcal)$ and
\[
\|BA\|_2\le \|A\|\,\|B\|_2,\qquad \|AB\|_2\le \|A\|\,\|B\|_2.
\]
\end{lemma}

\begin{proof}
Let $(e_k)$ be an orthonormal basis. Then
\[
\|BA\|_2^2=\sum_k \|BA e_k\|^2 \le \sum_k \|B\|^2\,\|A e_k\|^2
\le \|A\|^2\sum_k \|B e_k\|^2 = \|A\|^2\|B\|_2^2.
\]
The estimate for $\|AB\|_2$ follows similarly from $\|AB e_k\|\le \|A\|\,\|B e_k\|$.
\end{proof}

\begin{proposition}[Trace duality estimate and cyclicity]\label{prop:trace_duality_cyclicity}
Let $\rho\in\Tone$ and $A\in\BH$. Then $\rho A$ and $A\rho$ are trace-class and
\[
|\Tr(\rho A)|\le \|\rho\|_1\,\|A\|.
\]
Moreover, the trace is cyclic in this setting:
\[
\Tr(\rho A)=\Tr(A\rho).
\]
\end{proposition}

\begin{proof}
By polar decomposition, $\rho=V|\rho|$ for a partial isometry $V\in\BH$ and $|\rho|:=(\rho^*\rho)^{1/2}\ge 0$.
Since $\rho\in\Tone$, we have $|\rho|\in\Tone$ and $\|\rho\|_1=\Tr(|\rho|)$.

Set $T:=|\rho|^{1/2}$. By Lemma~\ref{lem:traceclass_sqrt_HS}, $T\in \mathcal T_2(\Hcal)$ and $\|T\|_2^2=\Tr(|\rho|)=\|\rho\|_1$.
Define
\[
B:=T V^*,\qquad C:=T A.
\]
By Lemma~\ref{lem:bounded_preserves_HS}, both $B$ and $C$ are Hilbert--Schmidt, and
\[
\|B\|_2=\|T V^*\|_2\le \|V^*\|\,\|T\|_2\le \|T\|_2,
\qquad
\|C\|_2=\|T A\|_2\le \|A\|\,\|T\|_2.
\]

Now note that
\[
\rho A = V|\rho|A = (V T)(T A).
\]
The product of two Hilbert--Schmidt operators is trace-class, so $\rho A$ is trace-class, and we may compute its trace.
Using Lemma~\ref{lem:HS_CS} (and its cyclicity statement),
\[
\Tr(\rho A)=\Tr\big((V T)(T A)\big)=\Tr\big((T A)(V T)\big).
\]
But $(T A)(V T)= C (V T)$, and $(V T) = (T V^*)^* = B^*$. Hence
\[
\Tr(\rho A)=\Tr(C B^*)=\Tr(B^* C)=\overline{\Tr(C^* B)}.
\]
In particular, Lemma~\ref{lem:HS_CS} gives the estimate
\[
|\Tr(\rho A)|=|\Tr(B^* C)|\le \|B\|_2\,\|C\|_2
\le \|T\|_2\,(\|A\|\|T\|_2)
= \|A\|\,\|T\|_2^2
= \|A\|\,\Tr(|\rho|)
= \|A\|\,\|\rho\|_1.
\]
This proves the trace duality bound.

For cyclicity, observe similarly that $A\rho = A V|\rho| = (A V T)T$, again a product of Hilbert--Schmidt operators,
hence trace-class. Moreover, by Lemma~\ref{lem:HS_CS},
\[
\Tr(\rho A)=\Tr\big((V T)(T A)\big)=\Tr\big((T A)(V T)\big),
\qquad
\Tr(A\rho)=\Tr\big((A V T)T\big)=\Tr\big(T(A V T)\big).
\]
Since both are traces of products of two Hilbert--Schmidt operators, the cyclicity part of Lemma~\ref{lem:HS_CS}
implies $\Tr((V T)(T A))=\Tr((T A)(V T))$ and likewise $\Tr((A V T)T)=\Tr(T(A V T))$, yielding $\Tr(\rho A)=\Tr(A\rho)$.
(Equivalently: whenever $X,Y$ are such that $XY$ and $YX$ are trace-class, one has $\Tr(XY)=\Tr(YX)$; here we are in that regime.)
\end{proof}

\section{Self-adjoint and essentially self-adjoint operators}

\subsection{Densely defined operators and adjoints}

Throughout, $\Hcal$ is a complex Hilbert space with inner product
$\langle\cdot|\cdot\rangle$ conjugate-linear in the first argument and linear
in the second. An (unbounded) operator means a linear map
\[
A: D(A)\subset \Hcal \to \Hcal,
\]
where $D(A)$ is a linear subspace (the domain).

\begin{definition}[Densely defined]
An operator $A:D(A)\to\Hcal$ is \emph{densely defined} if $D(A)$ is dense in $\Hcal$.
\end{definition}

\begin{definition}[Adjoint]
Let $A:D(A)\to\Hcal$ be densely defined. The \emph{adjoint} $A^*$ of $A$ is defined as follows.
Its domain is
\[
D(A^*) := \Big\{\psi\in\Hcal \;\Big|\; \exists\,\eta\in\Hcal\ \text{s.t.}\ 
\langle \psi|A\alpha\rangle = \langle \eta|\alpha\rangle \ \forall \alpha\in D(A)\Big\},
\]
and for $\psi\in D(A^*)$ we define $A^*\psi := \eta$ (the $\eta$ from the above identity).
\end{definition}

\begin{proposition}[Well-definedness of $A^*$]
The vector $\eta$ in the definition of $A^*$ is uniquely determined by $\psi$, hence $A^*$ is well-defined.
\end{proposition}

\begin{proof}
Fix $\psi\in\Hcal$ and suppose $\eta_1,\eta_2\in\Hcal$ satisfy
$\langle\psi|A\alpha\rangle=\langle\eta_1|\alpha\rangle=\langle\eta_2|\alpha\rangle$
for all $\alpha\in D(A)$. Then $\langle\eta_1-\eta_2|\alpha\rangle=0$ for all $\alpha\in D(A)$.
By density of $D(A)$ and continuity of $\alpha\mapsto\langle\eta_1-\eta_2|\alpha\rangle$,
we get $\langle\eta_1-\eta_2|\beta\rangle=0$ for all $\beta\in\Hcal$, hence $\eta_1=\eta_2$.
\end{proof}

\begin{definition}[Kernel, range, orthogonal complement]
For $A:D(A)\to\Hcal$ define
\[
\ker(A):=\{\alpha\in D(A):A\alpha=0\},\qquad
\ran(A):=\{A\alpha:\alpha\in D(A)\}.
\]
For any subset $M\subset\Hcal$, define $M^\perp:=\{x\in\Hcal:\langle x|m\rangle=0\ \forall m\in M\}$.
\end{definition}

\begin{proposition}[Range--kernel orthogonality]
Let $A$ be densely defined. Then
\[
\ker(A^*) = \ran(A)^\perp .
\]
\end{proposition}

\begin{proof}
For $\psi\in\Hcal$,
\[
\psi\in\ker(A^*)
\iff \psi\in D(A^*)\ \text{and}\ A^*\psi=0
\iff \langle\psi|A\alpha\rangle=0\ \forall \alpha\in D(A)
\iff \psi\in \ran(A)^\perp .
\]
\end{proof}

\begin{definition}[Operator extension / inclusion]
Given operators $A:D(A)\to\Hcal$ and $B:D(B)\to\Hcal$, we say $B$ is an \emph{extension} of $A$,
and write $A\subset B$, if $D(A)\subset D(B)$ and $A\alpha=B\alpha$ for all $\alpha\in D(A)$.
\end{definition}

\begin{proposition}[Inclusion reverses under adjoint]
Let $A,B$ be densely defined and $A\subset B$. Then $B^*\subset A^*$.
\end{proposition}

\begin{proof}
Let $\psi\in D(B^*)$, so there exists $\eta\in\Hcal$ such that
$\langle\psi|B\beta\rangle=\langle\eta|\beta\rangle$ for all $\beta\in D(B)$.
In particular, for all $\alpha\in D(A)\subset D(B)$ we have
$\langle\psi|A\alpha\rangle=\langle\psi|B\alpha\rangle=\langle\eta|\alpha\rangle$,
so $\psi\in D(A^*)$ and $A^*\psi=\eta=B^*\psi$. Hence $B^*\subset A^*$.
\end{proof}

\begin{proposition}[$A^*$ is closed]
If $A$ is densely defined, then $A^*$ is a closed operator (its graph is closed in $\Hcal\oplus\Hcal$).
\end{proposition}

\begin{proof}
Assume $\psi_n\in D(A^*)$, $\psi_n\to\psi$ and $A^*\psi_n\to\eta$ in $\Hcal$.
For any $\alpha\in D(A)$ we have $\langle\psi_n|A\alpha\rangle=\langle A^*\psi_n|\alpha\rangle$.
Passing to the limit and using continuity of the inner product yields
$\langle\psi|A\alpha\rangle=\langle\eta|\alpha\rangle$ for all $\alpha\in D(A)$.
Thus $\psi\in D(A^*)$ and $A^*\psi=\eta$, i.e. the graph of $A^*$ is closed.
\end{proof}

\subsection{Symmetric and self-adjoint operators}

\begin{definition}[Symmetric]
A densely defined operator $A:D(A)\to\Hcal$ is \emph{symmetric} if
\[
\langle \alpha|A\beta\rangle=\langle A\alpha|\beta\rangle,\qquad \forall \alpha,\beta\in D(A).
\]
\end{definition}

\begin{proposition}[Symmetric $\iff$ $A\subset A^*$]
A densely defined operator $A$ is symmetric if and only if $A\subset A^*$.
\end{proposition}

\begin{proof}
($\Rightarrow$) Fix $\psi\in D(A)$ and set $\eta:=A\psi$. For any $\alpha\in D(A)$,
symmetry gives $\langle\psi|A\alpha\rangle=\langle A\psi|\alpha\rangle=\langle\eta|\alpha\rangle$,
so $\psi\in D(A^*)$ and $A^*\psi=\eta=A\psi$. Hence $A\subset A^*$.

($\Leftarrow$) If $A\subset A^*$ then for all $\alpha,\beta\in D(A)$,
\[
\langle \alpha|A\beta\rangle=\langle \alpha|A^*\beta\rangle=\langle A\alpha|\beta\rangle,
\]
which is symmetry.
\end{proof}

\begin{definition}[Self-adjoint]
A densely defined operator $A$ is \emph{self-adjoint} if $A=A^*$, i.e.
$D(A)=D(A^*)$ and $A\alpha=A^*\alpha$ for all $\alpha\in D(A)$.
\end{definition}

\begin{proposition}[Maximality]
If $A$ is self-adjoint and $A\subset B$ with $B$ self-adjoint, then $A=B$.
In particular, a self-adjoint operator is maximal among symmetric extensions.
\end{proposition}

\begin{proof}
If $A\subset B$, then by the inclusion-reversal property we have $B^*\subset A^*$.
Self-adjointness gives $B^*=B$ and $A^*=A$, hence $B\subset A$. Together with $A\subset B$
we obtain $A=B$.
\end{proof}

\subsection{Closability, closure, and closedness}

\begin{definition}[Closable / closure / closed]
A densely defined operator $A$ is \emph{closable} if $A^*$ is densely defined.
If $A$ is closable, its \emph{closure} is defined by
\[
\overline A := A^{**}.
\]
An operator is \emph{closed} if $A=\overline A$.
\end{definition}

\begin{proposition}[Symmetric operators are closable]
If $A$ is symmetric, then $A$ is closable.
\end{proposition}

\begin{proof}
If $A$ is symmetric then $A\subset A^*$, hence $D(A)\subset D(A^*)$.
Since $D(A)$ is dense, $D(A^*)$ is dense as well. Thus $A$ is closable by definition.
\end{proof}

\begin{proposition}[Double adjoint contains the operator]
If $A$ is closable, then $A\subset A^{**}=\overline A$.
\end{proposition}

\begin{proof}
Let $\psi\in D(A)$. We show $\psi\in D(A^{**})$ and $A^{**}\psi=A\psi$.
Fix $\alpha\in D(A^*)$. By definition of $A^*$, for all $\varphi\in D(A)$,
$\langle\alpha|A\varphi\rangle=\langle A^*\alpha|\varphi\rangle$.
Rename dummy variables: for all $\varphi\in D(A)$ and all $\alpha\in D(A^*)$,
$\langle\alpha|A\varphi\rangle=\langle A^*\alpha|\varphi\rangle$.
Take complex conjugates and use conjugate-linearity in the first slot:
\[
\langle \varphi|A^*\alpha\rangle=\langle A\varphi|\alpha\rangle,\qquad
\forall \varphi\in D(A),\ \forall \alpha\in D(A^*).
\]
Now fix $\varphi=\psi\in D(A)$ and set $\eta:=A\psi$. Then
$\langle\psi|A^*\alpha\rangle=\langle\eta|\alpha\rangle$ for all $\alpha\in D(A^*)$,
which precisely says $\psi\in D(A^{**})$ and $A^{**}\psi=\eta=A\psi$.
\end{proof}

\begin{corollary}[For symmetric $A$: $A\subset \overline A \subset A^*$]
If $A$ is symmetric, then
\[
A \subset \overline A = A^{**} \subset A^* .
\]
\end{corollary}

\begin{proof}
Symmetry implies closable, hence $A\subset A^{**}=\overline A$ by the previous proposition.
Also $A\subset A^*$ implies, by inclusion-reversal, $(A^*)^*\subset A^*$, i.e. $A^{**}\subset A^*$.
\end{proof}

\subsection{Essential self-adjointness and defect indices}

\begin{definition}[Essentially self-adjoint]
A symmetric operator $A$ is \emph{essentially self-adjoint} if its closure $\overline A$ is self-adjoint.
Equivalently, $A$ has a unique self-adjoint extension, namely $\overline A$.
\end{definition}

\begin{definition}[Defect indices]
Let $A$ be symmetric. Its \emph{defect indices} are
\[
d_+ := \dim\ker(A^*- i),\qquad d_- := \dim\ker(A^*+ i).
\]
\end{definition}

\begin{theorem}[von Neumann extension criterion]
A symmetric operator $A$ admits a self-adjoint extension if and only if $d_+=d_-$.
If $d_+\neq d_-$, then $A$ admits no self-adjoint extension.
\end{theorem}

\subsection{Practical criteria (avoiding explicit adjoints)}

The following two criteria are particularly useful in applications: they reduce
(self-/essential) self-adjointness to range conditions for $A+z$.

\begin{theorem}[Sufficient criterion for self-adjointness]
Let $A$ be symmetric. If there exists $z\in\C$ such that
\[
\ran(A+z)=\Hcal=\ran(A+\overline z),
\]
then $A$ is self-adjoint.
\end{theorem}

\begin{proof}
Since $A$ is symmetric we already have $A\subset A^*$. It remains to show $A^*\subset A$.
Fix $\psi\in D(A^*)$. Then $A^*\psi+z\psi\in\Hcal$.
By $\ran(A+z)=\Hcal$, there exists $\alpha\in D(A)$ such that
\[
A^*\psi+z\psi=(A+z)\alpha.
\]
Now for any $\beta\in D(A)$ we compute
\[
\langle\psi|(A+z)\beta\rangle
=\langle (A+z)^*\psi|\beta\rangle
=\langle A^*\psi+\overline z\,\psi|\beta\rangle,
\]
while also
\[
\langle\alpha|(A+z)\beta\rangle
=\langle (A+z)\alpha|\beta\rangle
=\langle A^*\psi+z\psi|\beta\rangle.
\]
Using $\ran(A+z)=\Hcal$, the identity
$\langle\psi|\varphi\rangle=\langle\alpha|\varphi\rangle$ holds for all $\varphi\in\Hcal$,
hence $\psi=\alpha$. Therefore $\psi\in D(A)$ and $A^*\psi=A\psi$, i.e. $A^*\subset A$.
\end{proof}

\begin{theorem}[Criterion for essential self-adjointness]
Let $A$ be symmetric. Then $A$ is essentially self-adjoint if and only if there exists
$z\in\C\setminus\R$ such that
\[
\ran(A+z)=\Hcal=\ran(A+\overline z).
\]
\end{theorem}

\begin{theorem}[Equivalent kernel criterion]
Let $A$ be symmetric. Then $A$ is essentially self-adjoint if and only if there exists
$z\in\C\setminus\R$ such that
\[
\ker(A^*+z)=\{0\}=\ker(A^*+\overline z).
\]
\end{theorem}

\begin{proof}
By the range--kernel orthogonality, for any densely defined $T$ we have
$\ker(T^*)=\ran(T)^\perp$. Apply this with $T:=A+z$ and note $(A+z)^*=A^*+\overline z$.
Then
\[
\ran(A+z)^\perp=\ker\big((A+z)^*\big)=\ker(A^*+\overline z),
\qquad
\ran(A+\overline z)^\perp=\ker(A^*+z).
\]
Taking orthogonal complements (and using $M^{\perp\perp}=\overline M$ for subspaces) gives
\[
\overline{\ran(A+z)}=\ker(A^*+\overline z)^\perp,\qquad
\overline{\ran(A+\overline z)}=\ker(A^*+z)^\perp.
\]
Hence $\ker(A^*+\overline z)=\{0\}$ iff $\overline{\ran(A+z)}=\Hcal$, and similarly for the conjugate.
Comparing with the previous theorem yields the equivalence.
\end{proof}

\subsection{A finance-oriented addendum: \texorpdfstring{$S\ge 0$, $\log S$, and $g(S)$}{S\string>=0, log S, and g(S)}}

In the quantum-pricing chapters, (discounted) prices are modeled as (possibly unbounded)
self-adjoint operators $S$ on $\Hcal$.

\begin{definition}[Positive (semi-)definite operator]
A self-adjoint operator $S$ is \emph{positive} (write $S\ge 0$) if
\[
\langle \psi|S\psi\rangle \ge 0,\qquad \forall \psi\in D(S).
\]
If, moreover, there exists $c>0$ such that $\langle\psi|S\psi\rangle\ge c\|\psi\|^2$ for all $\psi\in D(S)$,
we write $S\ge cI$ (strict positivity / bounded below).
\end{definition}

\begin{remark}[When $\log S$ is unproblematic]
If $S\ge cI$ for some $c>0$, then $\sigma(S)\subset[c,\infty)$ and $\log$ is bounded on $\sigma(S)$.
Consequently $\log(S)$ can be defined as a bounded self-adjoint operator via the functional calculus
(to be developed in the spectral chapter).
\end{remark}

\begin{remark}[Domain of $\log S$ in the general case]
If $S\ge 0$ but $0$ lies in $\sigma(S)$, then $\log$ is unbounded near $0$ and $\log(S)$,
when defined, is typically unbounded. In that case one must specify a domain
$D(\log S)\subset \Hcal$ (given by square-integrability of $\log$ against the spectral measure).
In particular, if $\ker(S)\neq\{0\}$ then any definition of $\log(S)$ forces $D(\log S)$ to be orthogonal
to (at least part of) $\ker(S)$; hence, for pricing applications one often imposes either $S\ge cI$
or works with bounded truncations.
\end{remark}

\begin{remark}[Definability of payoffs $g(S)$]
If $g:\R\to\R$ is bounded Borel, then $g(S)$ is bounded self-adjoint.
For unbounded $g$ (e.g. linear growth), $g(S)$ is defined as an unbounded operator with domain
\[
D(g(S)):=\Big\{\psi\in\Hcal:\int_{\R}|g(\lambda)|^2\,d\mu_\psi^S(\lambda)<\infty\Big\},
\qquad
\mu_\psi^S(\Delta):=\langle\psi|P_S(\Delta)\psi\rangle,
\]
where $P_S$ is the spectral measure of $S$.
In practice (and consistent with the ``local information'' philosophy), it is often convenient to
replace $g$ by bounded truncations $g_n$ so that $g_n(S)\in\BH$ and all conditional expectations
remain everywhere-defined.
\end{remark}


\section{Spectra and Perturbation Theory}

\subsection{Resolvent map and spectrum}

\begin{definition}[Resolvent set and resolvent map]
Let $\Hcal$ be a complex Hilbert space and $A:D(A)\subset \Hcal\to \Hcal$ be a linear operator.
The \emph{resolvent set} of $A$ is
\[
\rho(A):=\bigl\{ z\in\mathbb C : (A-z)^{-1}\in \mathcal L(\Hcal)\bigr\}.
\]
The \emph{resolvent map} of $A$ is the map
\[
R_A:\rho(A)\to \mathcal L(\Hcal),\qquad z\mapsto (A-z)^{-1}.
\]
\end{definition}

\begin{remark}[Closed operators and invertibility]
If $A$ is closed, then by the closed graph theorem one has
\[
(A-z)^{-1}\in \mathcal L(\Hcal)
\quad\Longleftrightarrow\quad
A-z:D(A)\to\Hcal\ \text{is bijective}.
\]
Thus, for closed operators the condition $z\in\rho(A)$ is equivalent to bijectivity of $A-z$.
\end{remark}

\begin{definition}[Spectrum]
The \emph{spectrum} of $A$ is defined as
\[
\sigma(A):=\mathbb C\setminus \rho(A).
\]
\end{definition}

\begin{definition}[Eigenvalues]
A complex number $\lambda\in\mathbb C$ is an \emph{eigenvalue} of $A$ if there exists $\psi\in D(A)\setminus\{0\}$
such that $A\psi=\lambda\psi$. Such a vector $\psi$ is called an \emph{eigenvector} of $A$ associated to $\lambda$.
\end{definition}

\begin{proposition}[Eigenvalues belong to the spectrum]
If $\lambda$ is an eigenvalue of $A$, then $\lambda\in\sigma(A)$.
\end{proposition}

\begin{proof}
If $A\psi=\lambda\psi$ for some $\psi\neq 0$, then $(A-\lambda)\psi=0$ and hence $\ker(A-\lambda)\neq \{0\}$.
Therefore $A-\lambda$ is not injective, hence not invertible, so $\lambda\notin \rho(A)$.
Equivalently, $\lambda\in\sigma(A)$.
\end{proof}

\subsection{Spectrum of a self-adjoint operator}

Throughout this subsection, let $A=A^*$ be self-adjoint on $\Hcal$.

\begin{definition}[Refined spectrum for self-adjoint operators]
Define the following subsets of $\sigma(A)$:
\begin{align*}
\sigma_{\mathrm{pp}}(A)
&:=\bigl\{ z\in\mathbb C:\ \overline{\ran(A-z)}=\ran(A-z)\neq \Hcal\bigr\},\\
\sigma_{\mathrm{pec}}(A)
&:=\bigl\{ z\in\mathbb C:\ \overline{\ran(A-z)}\neq \ran(A-z)\neq \Hcal\bigr\},\\
\sigma_{\mathrm{pc}}(A)
&:=\bigl\{ z\in\mathbb C:\ {\ran(A-z)}\neq \overline{\ran(A-z)}= \Hcal\bigr\}.
\end{align*}
These sets are pairwise disjoint and their union is $\sigma(A)$.
\end{definition}

\begin{definition}[Point spectrum and continuous spectrum]
Define
\[
\sigma_{\mathrm{p}}(A):=\sigma_{\mathrm{pp}}(A)\cup \sigma_{\mathrm{pec}}(A)
=\bigl\{ z\in\mathbb C:\overline{\ran(A-z)}\neq \Hcal\bigr\},
\]
and
\[
\sigma_{\mathrm{c}}(A):=\sigma_{\mathrm{pec}}(A)\cup \sigma_{\mathrm{pc}}(A)
=\bigl\{ z\in\mathbb C:\overline{\ran(A-z)}\neq \ran(A-z)\bigr\}.
\]
Then $\sigma_{\mathrm{p}}(A)\cup \sigma_{\mathrm{c}}(A)=\sigma(A)$, but $\sigma_{\mathrm{p}}(A)\cap \sigma_{\mathrm{c}}(A)=\sigma_{\mathrm{pec}}(A)$ may be nonempty.
\end{definition}

\begin{lemma}[Eigenvalues of self-adjoint operators are real]
If $\lambda$ is an eigenvalue of $A$, then $\lambda\in\mathbb R$.
\end{lemma}

\begin{proof}
Let $A\psi=\lambda\psi$ with $\psi\neq 0$. Then
\[
\lambda\langle \psi,\psi\rangle
=\langle \psi,\lambda\psi\rangle
=\langle \psi,A\psi\rangle
=\langle A\psi,\psi\rangle
=\langle \lambda\psi,\psi\rangle
=\overline{\lambda}\langle \psi,\psi\rangle.
\]
Since $\langle\psi,\psi\rangle>0$, it follows that $\lambda=\overline{\lambda}$, i.e. $\lambda\in\mathbb R$.
\end{proof}

\begin{theorem}[Point spectrum equals the set of eigenvalues for self-adjoint operators]
For self-adjoint $A$, the elements of $\sigma_{\mathrm{p}}(A)$ are precisely the eigenvalues of $A$.
\end{theorem}

\begin{lemma}[Orthogonality of eigenvectors]
Eigenvectors associated to distinct eigenvalues of a self-adjoint operator are orthogonal.
\end{lemma}

\begin{proof}
Let $A\psi=\lambda\psi$ and $A\varphi=\lambda'\varphi$ with $\lambda\neq \lambda'$.
Then
\[
(\lambda-\lambda')\langle \psi,\varphi\rangle
=\langle A\psi,\varphi\rangle-\langle \psi,A\varphi\rangle
=\langle \psi,A\varphi\rangle-\langle \psi,A\varphi\rangle
=0,
\]
hence $\langle \psi,\varphi\rangle=0$.
\end{proof}

\subsection{Pure point price spectrum (working hypothesis)}

\begin{axiom}[Pure point price spectrum]
In the quantum-pricing chapters, the (discounted) price observable is modeled by a self-adjoint operator
$S=S^*$ on $\Hcal$ whose spectrum is \emph{pure point}:
\[
\sigma(S)=\sigma_{\mathrm{pp}}(S)\subset \mathbb R.
\]
Moreover, we assume each eigenspace $\Eig_S(\lambda):=\{\psi\in D(S):S\psi=\lambda\psi\}$ is finite-dimensional.
\end{axiom}

\begin{remark}[Interpretation]
This hypothesis idealizes price formation as taking values in a discrete set of admissible levels (a ''grid''),
and will be consistent with a projection-valued spectral decomposition and discrete induced distributions.
Continuous-spectrum models can be interpreted as scaling limits of increasingly fine grids.
\end{remark}

\begin{proposition}[Spectral decomposition under pure point hypothesis]
Assume the above axiom holds and enumerate the eigenvalues (with multiplicity spaces)
as $(\lambda_n)_{n\in I}\subset \mathbb R$ (countable index set), and denote by $P_n$ the orthogonal projection onto
$\Eig_S(\lambda_n)$. Then:
\begin{enumerate}[label=(\roman*)]
\item $P_nP_m=0$ for $n\neq m$, and $P_n=P_n^*=P_n^2$;
\item $S=\sum_{n\in I}\lambda_n P_n$ in the strong sense on $D(S)$;
\item for any bounded Borel function $g:\mathbb R\to\mathbb C$, the functional calculus satisfies
\[
g(S)=\sum_{n\in I} g(\lambda_n) P_n
\quad\text{(strong operator convergence)}.
\]
\end{enumerate}
\end{proposition}

\begin{remark}[Discrete induced distribution and pricing functional]
Let $\rho$ be a normal state (density operator) and define $p_n:=\Tr(\rho P_n)\in[0,1]$.
Then $\sum_{n\in I} p_n=1$ and, for bounded $g$,
\[
\Tr\bigl(\rho\, g(S)\bigr)=\sum_{n\in I} g(\lambda_n)\, p_n.
\]
Thus, under pure point spectrum the spectral measure induces an \emph{atomic} probability distribution on the
eigenvalues, and pricing reduces to a discrete spectral sum.
\end{remark}

\subsubsection*{Rigorous first-order perturbation near an isolated eigenvalue cluster}

\begin{definition}[Spectral gap (isolated eigenvalue cluster)]
Fix $n$ and set $E:=\mathrm{Eig}_{H_0}(h_n)$, $d:=\dim E<\infty$.
We say that $h_n$ is \emph{isolated} (has a spectral gap) if
\[
\gamma_n := \mathrm{dist}\bigl(h_n,\ \sigma(H_0)\setminus\{h_n\}\bigr) \;>\; 0.
\]
Equivalently, there exists $r\in(0,\gamma_n)$ such that
$\sigma(H_0)\cap (h_n-r,\ h_n+r)=\{h_n\}$ (counting multiplicity).
\end{definition}

\begin{lemma}[Resolvent bound]
\label{lem:resolvent-bound}
Let $A=A^*$ be self-adjoint and let $z\in\mathbb C\setminus\sigma(A)$.
Then $(A-z)^{-1}\in\mathcal B(H)$ and
\[
\|(A-z)^{-1}\| \;=\; \sup_{\lambda\in\sigma(A)}\frac{1}{|\lambda-z|}
\;=\; \frac{1}{\mathrm{dist}(z,\sigma(A))}.
\]
\end{lemma}

\begin{proof}
For self-adjoint $A$, the bounded Borel functional calculus gives
$(A-z)^{-1}=g_z(A)$ with $g_z(\lambda):=(\lambda-z)^{-1}$ bounded on $\sigma(A)$.
Hence $\|(A-z)^{-1}\|=\|g_z(A)\|=\sup_{\lambda\in\sigma(A)}|g_z(\lambda)|
=\sup_{\lambda\in\sigma(A)}\frac{1}{|\lambda-z|}$, which equals $1/\mathrm{dist}(z,\sigma(A))$.
\end{proof}

\begin{lemma}[Riesz projection for the perturbed cluster]
\label{lem:riesz-proj}
Assume $H_0=H_0^*$ and $W=W^*\in\mathcal B(H)$, and define
$H_\varepsilon:=H_0+\varepsilon W$ on $D(H_\varepsilon):=D(H_0)$.
Fix $n$ and choose $r\in(0,\gamma_n/2)$, and let
\[
\Gamma:=\{z\in\mathbb C:\ |z-h_n|=r\}.
\]
Then there exists $\varepsilon_0>0$ such that for all $|\varepsilon|<\varepsilon_0$:
\begin{itemize}
\item $z\in\Gamma$ implies $z\in\rho(H_\varepsilon)$, and $(H_\varepsilon-z)^{-1}\in\mathcal B(H)$;
\item the operator
\[
P(\varepsilon):=\frac{1}{2\pi i}\oint_{\Gamma}(H_\varepsilon-z)^{-1}\,dz
\]
is a bounded projection commuting with $H_\varepsilon$;
\item $P(\varepsilon)$ depends continuously on $\varepsilon$ in operator norm and
$\mathrm{rank}\,P(\varepsilon)=\mathrm{rank}\,P(0)=d$ for $|\varepsilon|$ small.
\end{itemize}
\end{lemma}

\begin{proof}
For $z\in\Gamma$, we have $\mathrm{dist}(z,\sigma(H_0))\ge \gamma_n-r \ge \gamma_n/2$,
hence by Lemma~\ref{lem:resolvent-bound},
\[
\|(H_0-z)^{-1}\|\le \frac{2}{\gamma_n}.
\]
Choose $\varepsilon_0:=\gamma_n/(4\|W\|)$.
Then for $|\varepsilon|<\varepsilon_0$ and all $z\in\Gamma$,
\[
\|\varepsilon W(H_0-z)^{-1}\|\le |\varepsilon|\|W\|\|(H_0-z)^{-1}\|
\le \frac{1}{2}.
\]
Hence $I+\varepsilon W(H_0-z)^{-1}$ is invertible and admits a Neumann series.
Using
\[
H_\varepsilon-z = (H_0-z)\bigl(I+\varepsilon W(H_0-z)^{-1}\bigr),
\]
we obtain $(H_\varepsilon-z)^{-1}$ exists and is bounded for all $z\in\Gamma$.
Define $P(\varepsilon)$ by the contour integral; standard resolvent identities show
$P(\varepsilon)^2=P(\varepsilon)$ and $P(\varepsilon)H_\varepsilon=H_\varepsilon P(\varepsilon)$.

Moreover, by the Neumann series,
\[
(H_\varepsilon-z)^{-1}-(H_0-z)^{-1}
=-(H_0-z)^{-1}\varepsilon W(H_\varepsilon-z)^{-1},
\]
so $\sup_{z\in\Gamma}\|(H_\varepsilon-z)^{-1}-(H_0-z)^{-1}\|=O(|\varepsilon|)$, hence
$\|P(\varepsilon)-P(0)\|=O(|\varepsilon|)$.
For projections $P,Q$, if $\|P-Q\|<1$ then $\mathrm{rank}\,P=\mathrm{rank}\,Q$.
Thus for $|\varepsilon|$ small we have $\mathrm{rank}\,P(\varepsilon)=\mathrm{rank}\,P(0)=d$.
\end{proof}

\begin{theorem}[Rigorous first-order splitting of an isolated eigenvalue cluster]
\label{thm:first-order-rigorous}
Assume the setting of Lemma~\ref{lem:riesz-proj} and that $h_n$ is isolated with
$E=\mathrm{Eig}_{H_0}(h_n)$, $\dim E=d<\infty$.
Let $P:=P(0)$ be the orthogonal projection onto $E$.
Consider the compression (finite-dimensional self-adjoint operator on $E$)
\[
K:=PWP\big|_{E}\;:\;E\to E.
\]
Let $\mu_1,\dots,\mu_d\in\mathbb R$ be the eigenvalues of $K$ (counting multiplicity).
Then there exists $\varepsilon_0>0$ such that for all $|\varepsilon|<\varepsilon_0$:

\begin{itemize}
\item The spectrum of $H_\varepsilon$ in the interval $(h_n-r,h_n+r)$ consists of exactly $d$
(real) eigenvalues (counting multiplicity), denoted $\theta_{n,1}(\varepsilon),\dots,\theta_{n,d}(\varepsilon)$.
\item As $\varepsilon\to 0$,
\[
\theta_{n,j}(\varepsilon)= h_n + \varepsilon\,\mu_j + o(\varepsilon),
\qquad j=1,\dots,d,
\]
i.e.\ the multiset of first-order shifts is precisely $\mathrm{spec}(K)$.
\end{itemize}

In particular, if one chooses an orthonormal basis $(e_{n\delta})_{\delta=1}^d$ of $E$
diagonalizing $K$ (equivalently diagonalizing the matrix $\langle e_{n\alpha},We_{n\beta}\rangle$),
so that $Ke_{n\delta}=\mu_\delta e_{n\delta}$, then
\[
\theta_{n,\delta}(\varepsilon)=h_n+\varepsilon\langle e_{n\delta},We_{n\delta}\rangle+o(\varepsilon).
\]
\end{theorem}

\begin{proof}
Let $P(\varepsilon)$ be the Riesz projection from Lemma~\ref{lem:riesz-proj}.
Because $P(\varepsilon)$ commutes with $H_\varepsilon$, the subspace
$E(\varepsilon):=\mathrm{ran}\,P(\varepsilon)$ is $H_\varepsilon$-invariant and $\dim E(\varepsilon)=d$.
Hence the spectrum of $H_\varepsilon$ inside the disk bounded by $\Gamma$
is exactly the spectrum of the finite-dimensional self-adjoint operator
\[
B(\varepsilon):=H_\varepsilon\big|_{E(\varepsilon)} \;:\; E(\varepsilon)\to E(\varepsilon),
\]
and consists of $d$ real eigenvalues counting multiplicity.

We now compute the first-order behavior. Using $P(\varepsilon)^2=P(\varepsilon)$ and differentiating at $\varepsilon=0$
(one may justify differentiability at $0$ from the resolvent expansion in Lemma~\ref{lem:riesz-proj}),
we obtain the standard projection identity
\[
P\,P'(0)\,P = 0.
\]
Next, note that $H_0P = Ph_n$ on $E$ (since $E$ is the eigenspace at $h_n$), hence
\[
PH_0P = h_n P \quad\text{on }H.
\]
Consider the family of compressed operators on the fixed finite-dimensional space $E$:
\[
\widetilde B(\varepsilon):= P\,P(\varepsilon)\,H_\varepsilon\,P(\varepsilon)\,P\big|_{E} \;:\; E\to E.
\]
For $|\varepsilon|$ small, $P(\varepsilon)$ is close to $P$ in norm, hence $P(\varepsilon)P:E\to E(\varepsilon)$
is an isomorphism; therefore $\widetilde B(\varepsilon)$ is similar to $B(\varepsilon)$ and has the same eigenvalues
(counting multiplicity).

Expand $\widetilde B(\varepsilon)$ to first order:
\[
\widetilde B(\varepsilon)
= P P(\varepsilon)(H_0+\varepsilon W)P(\varepsilon)P
= P P(\varepsilon)H_0P(\varepsilon)P \;+\; \varepsilon\, P P(\varepsilon)W P(\varepsilon)P.
\]
Using $H_0P=h_nP$ and $PH_0=h_nP$ together with $P P'(0)P=0$, one checks that
\[
\frac{d}{d\varepsilon}\Big|_{\varepsilon=0}\Bigl(PP(\varepsilon)H_0P(\varepsilon)P\Bigr)=0
\quad\text{as an operator on }E.
\]
On the other hand,
\[
\frac{d}{d\varepsilon}\Big|_{\varepsilon=0}\Bigl(PP(\varepsilon)W P(\varepsilon)P\Bigr)
= PWP\big|_{E} = K.
\]
Therefore,
\[
\widetilde B(\varepsilon)= h_n I_E + \varepsilon K + o(\varepsilon)
\quad\text{in operator norm on }E.
\]
Now apply the elementary Lipschitz stability of eigenvalues for Hermitian matrices:
if $A,B$ are $d\times d$ Hermitian, then after ordering eigenvalues increasingly,
\[
\max_{1\le j\le d}|\lambda_j(A)-\lambda_j(B)|\le \|A-B\|.
\]
Hence the eigenvalues of $\widetilde B(\varepsilon)$ satisfy
\[
\theta_{n,j}(\varepsilon)
=\lambda_j(\widetilde B(\varepsilon))
=\lambda_j(h_n I_E+\varepsilon K)+o(\varepsilon)
= h_n+\varepsilon\mu_j+o(\varepsilon),
\]
as claimed. The final displayed formula follows once the basis diagonalizes $K$.
\end{proof}

\begin{proposition}[First-order eigenvector correction: $E^\perp$ component]
\label{prop:evec-firstorder-rigorous}
In the setting of Theorem~\ref{thm:first-order-rigorous}, fix an eigenvector $e_{n\delta}\in E$
of $K=PWP|_E$ with eigenvalue $\mu_\delta$.
Assume $\mu_\delta$ is simple (non-degenerate) for $K$.
Then there exists a normalized eigenpair $(\theta_{n,\delta}(\varepsilon),\psi_{n,\delta}(\varepsilon))$
for $H_\varepsilon$ with $\psi_{n,\delta}(0)=e_{n\delta}$ and
\[
\theta_{n,\delta}(\varepsilon)=h_n+\varepsilon\mu_\delta+o(\varepsilon),
\]
and the derivative at $0$ satisfies
\[
P^\perp \psi'_{n,\delta}(0)= -\,P^\perp (H_0-h_n)^{-1} P^\perp W e_{n\delta},
\]
where $(H_0-h_n)^{-1}$ is the bounded inverse of $H_0-h_n$ on $E^\perp$.
\end{proposition}

\begin{proof}
Because $\mu_\delta$ is simple for $K$, the corresponding branch in the finite-dimensional reduction
is isolated; lifting through the Riesz projection yields a corresponding eigenpair for $H_\varepsilon$
continuous (indeed differentiable) at $\varepsilon=0$.
Differentiate $(H_0+\varepsilon W)\psi_{n,\delta}(\varepsilon)=\theta_{n,\delta}(\varepsilon)\psi_{n,\delta}(\varepsilon)$
at $\varepsilon=0$:
\[
(H_0-h_n)\psi'_{n,\delta}(0)=-(W-\theta'_{n,\delta}(0))e_{n\delta}.
\]
Project to $E^\perp$ using $P^\perp$ and use $P^\perp e_{n\delta}=0$ to obtain
\[
P^\perp(H_0-h_n)P^\perp\, P^\perp\psi'_{n,\delta}(0)= -P^\perp W e_{n\delta}.
\]
Since $h_n$ is isolated, $H_0-h_n$ is invertible on $E^\perp$ with bounded inverse,
hence the stated formula follows.
\end{proof}

\begin{remark}[Discrete spectra and price levels (informal)]
Under the pure point hypothesis for price observables (affiliated with an abelian information algebra),
spectral values may be read as discrete admissible price levels (tick/jump idealization).
Continuous-spectrum models can be interpreted as continuum (diffusion-type) limits of refining grids.
\end{remark}

\begin{remark}[Perturbations]
A decomposition $H_\varepsilon=H_0+\varepsilon W$ be viewed as a baseline specification together with a weak correction,
where $\varepsilon$ encodes the magnitude of an information effect.
Perturbation theory then describes the first-order shifts of discrete spectral levels $\theta_{n\delta}(\varepsilon)$.
\end{remark}


\end{document}